\documentclass[10pt]{amsart}
\usepackage{amsthm,amsmath,amssymb,amscd,graphics,enumerate, stmaryrd,xspace,verbatim, epic, eepic,url}

\usepackage[usenames, dvipsnames]{color}

\usepackage[colorlinks=true]{hyperref}

\usepackage[active]{srcltx}
\usepackage[all]{xypic}
\SelectTips{cm}{}
{
   \newtheorem{theorem}[subsubsection]{Theorem}
      \newtheorem*{theorem*}{Theorem}
   \newtheorem{proposition}[subsubsection]{Proposition}

   \newtheorem{lemma}[subsubsection]{Lemma}

   \newtheorem{corollary}[subsubsection]{Corollary}
   
   \newtheorem*{conjecture*}{Conjecture}

}
{\theoremstyle{definition}
          \newtheorem*{exercise*}{Exercise}
   
   \newtheorem{example}[subsubsection]{Example}
   \newtheorem*{example*}{Example}
   \newtheorem{definition}[subsubsection]{Definition}
   
   \newtheorem*{definition*}{Definition}
   
   \newtheorem{remark}[subsubsection]{Remark}

}
\newcommand{\QQ}{{\mathbb{Q}}}

\newcommand{\NN}{{\mathbb{N}}}

\newcommand{\ZZ}{{\mathbb{Z}}}

\renewcommand{\AA}{{\mathbb{A}}}

\newcommand{\ce}{{\operatorname{Z}}}

\newcommand{\bA}{{\mathbf{A}}}

\newcommand{\cI}{{\mathcal I}}
\newcommand{\cJ}{{\mathcal J}}
\newcommand{\cK}{{\mathcal K}}
\renewcommand{\cL}{{\mathcal L}}
\newcommand{\cM}{{\mathcal M}}

\newcommand{\cO}{{\mathcal O}}

\def\<{\langle}
\def\>{\rangle}
\newcommand{\Ver}{\operatorname{Vert}}
\newcommand{\Link}{\operatorname{Link}}
\newcommand{\inte}{{\operatorname{int}}}
\newcommand{\integ}{{\operatorname{integ}}}
\newcommand{\Reg}{\operatorname{Reg}}

\newcommand{\strat}{\operatorname{strat}}
\newcommand{\val}{\operatorname{val}}

\newcommand{\bl}{\operatorname{bl}}

\newcommand{\Spec}{\operatorname{Spec}}

\newcommand{\Proj}{\operatorname{Proj}}

\newcommand{\Sing}{{\operatorname{Sing}}}
\newcommand{\Star}{\operatorname{Star}}

\newcommand{\Aut}{{\operatorname{Aut}}}

\newcommand{\spa}{{\operatorname{span}}}

\newcommand{\codim}{\operatorname{codim}}

\newcommand{\rank}{{\operatorname{rank}}}

\newcommand{\inn}{{\operatorname{in}}}

\newcommand{\sing}{{\operatorname{sing}}}

\def\:{{\colon}}
\def\.{{,\dots,}}
\def\dim{{\rm dim}}

\newcommand{\double}{\genfrac..{0pt}1
{\raise -1pt\hbox{$\scriptstyle\longrightarrow$}}{\raise 3pt\hbox
{$\scriptstyle\longrightarrow$}}}

\renewcommand{\setminus}{\smallsetminus}

\def\sat{{\rm sat}}

\def\et{_{\rm \acute et}}

\def\tototi{\mathbin{\mathop{\otimes}\limits^{\raise-1pt\hbox
{$\scriptscriptstyle {\rm L}$}}}}

\def\indlim{\mathop{\vrule width0pt height7pt depth
4pt\smash{\lim\limits_{\raise 1pt\hbox to 14.5pt
{\rightarrowfill}}}}}
\def\projlim{\mathop{\vrule width0pt height7pt depth
4pt\smash{\lim\limits_{\raise 1pt\hbox to 14.5pt
{\leftarrowfill}}}}}

\newcommand\displaceamount{3pt}

\newcommand{\doubledown}{\ar@<\displaceamount>[d]\ar@<-\displaceamount>[d]}

\newcommand{\doubleup}{\ar@<\displaceamount>[u]\ar@<-\displaceamount>[u]}

\newcommand{\doubleright}{\ar@<\displaceamount>[r]\ar@<-\displaceamount>[r]}

\newcommand{\tor}{{\operatorname{tor}}}

\newcommand{\exc}{{\operatorname{exc}}}

\newcommand{\Cart}{{\operatorname{Cart}}}

\def\reg{{\operatorname{reg}}}

\begin{document}
\title[Functorial  resolution except for log smooth locus.]{Functorial  resolution except  for toroidal locus. Toroidal compactification}
%\title[Functorial  resolution of locally binomial and log smooth varieties  ]{ Functorial  resolution of locally binomial and log smooth varieties}

\author[J. W{\l}odarczyk] {Jaros{\l}aw W{\l}odarczyk}
\address{Department of Mathematics, Purdue University\\
150 N. University Street,\\ West Lafayette, IN 47907-2067}
\thanks{This research is supported by  BSF grant 2014365}

\date{\today}
%\large
\begin{abstract}
Let $X$ be any variety in characteristic zero. Let $V \subset X$ be an open subset that has toroidal singularities. We show the existence of a canonical desingularization of $X$ except for V.  It is a morphism  $f: Y \to X$ , which  does not modify the subset $ V $ and transforms  $X$ into a toroidal embedding $Y$, with singularities extending those on $V$. Moreover, the exceptional divisor has simple normal crossings on $Y$. 

The theorem naturally generalizes the Hironaka canonical desingularization.  It  does not modify the nonsingular locus $V$ and transforms $X$ into a nonsingular variety $Y$.

The proof uses, in particular, the canonical desingularization of logarithmic varieties recently proved by Abramovich -Temkin-W\l odarczyk.  It also  relies on  the established  here canonical functorial desingularization of locally toric varieties with an unmodified open toroidal subset.
As an application, we show the existence of a toroidal equisingular compactification of toroidal varieties.
All the results here can be linked  to  a simple functorial combinatorial desingularization algorithm developed in this paper.

  \end{abstract}
\maketitle
\setcounter{tocdepth}{1}

\tableofcontents

%\addtocounter{section}{-1}

 \section{Introduction} 
 
 A question of   good compactifications of algebraic varieties is of fundamental nature. It was studied in  different contexts, by Nagata \cite{Nagata}, Satake \cite{Satake}, Baily-Borel \cite{BB}, Mumford \cite{AMRT}, Sumihiro \cite{Sumihiro} and many others.  
 In a smooth case or the case of isolated singularities it is certainly possible, using Nagata imbedding and Hironaka desingularization, to compactify the variety $X\subset \overline{X}$, such that $\overline{X}$ is smooth outside of the singularities on $X$, and the complement $\overline{X}\setminus X$ is a simple normal crossing divisor. 
 (\cite{Hironaka},\cite{Bierstone-Milman},\cite{Villamayor},\cite{Wlodarczyk},\cite{Kollar})
If $X$ admits some mild non-isolated singularities, we still would like to   have a good compactification with the boundary divisor having  simple intersections.

 The problem  is directly related to the existence of  good or partial resolutions of singularities. 
  The question of good desingularization was studied in the context of log resolution, in particular,  by Szab\'o \cite{Szabo}, K\'ollar in \cite{Kollar2}, and Bierstone-Milman \cite{B-M-except-I}. On the other hand in the papers  of Gonzalez-Perez-Teissier \cite{PT}, \, Tevelev \cite{Tevelev}, \, K\'ollar \cite{Kollar-toroidal} ,  different versions of good combinatorial partial desingularization were introduced.

 %In the paper we address this problem in much more general si
 
 In this  paper we consider a problem  of a good partial  resolution, which does not modify  a given open  subset with a certain type of singularities, and no new type of singularities is introduced. Moreover 
 the boundary divisor have some simple intersections. 
 This type of partial resolution can be used, in particular for constructing good compactifications.
 
 %In other words, a variety with an open subset with an allowed certain type of singularities shall be resolved by a birational projective modification 

 The problem generalizes the Hironaka  desingularization
 theorem which does not modify the smooth locus of the scheme and the exceptional locus  is an SNC divisor.
 
 A particular question of the existence of  partial desingularization except for  normal crossing (NC)  locus was posed by K\'ollar in \cite{Kollar2}.

 Some results in this direction were proven by Bierstone-Milman in \cite{B-M-except-I}. They show that such a (partial) resolution except for   SNC  locus exists for any reduced and reducible scheme of finite type over a field of characteristic zero. Moreover, the SNC locus on the resolved (reducible) variety is the closure of the SNC locus on the given variety.
  
  They also observed in the example of the "pinch point" or "Whitney umbrella", that the partial resolution except for NC locus, should allow some
 more general singularities.

In their paper(s)  \cite{B-M-except-I}, and \cite{B-M-except-II} (jointly with Lairez) they  give  a complete list of possible singularities in a low dimension  and a low codimension- "more general pinch points" which need to be allowed  to resolve the schemes except for NC locus. 
  
  The present paper addresses the problem of good resolution in a  more general situation where the unmodified set is defined by a strict toroidal embedding.
This, in particular excludes the case of NC locus, which cannot be resolved. 
 One shall mention, that  Bierstone and Milman in their approach use their singularity invariant
 and alter some steps of their proof of Hironaka desingularization to prevent modification of  the NC and SNC locus. 
  Unfortunately, in general, the Bierstone-Milman invariant gives a very little geometric information, as it is specifically designed for the inductive structure of the resolution process.

 An alternative resolution tool giving a more  precise, more geometric and more efficient  control over the resolution process  was introduced by Mumford and others in \cite{KKMS}. They consider the language of  toroidal embeddings, which allows to translate the resolution problems into the  language of simpler combinatorial objects - conical complexes.  The method was initially introduced in \cite{KKMS} to solve  the problem of semistable reduction of a dominant morphism to a curve, and proved successful for solving many fundamental problems in  birational geometry, weak semistable reduction \cite{AK}, weak factorization theorem\cite{Wlodarczyk-toroidal},\cite{AKMW},\cite{Wlodarczyk-toroidal2}, 
  and many others.
  %In particular, Mumford and his collaborators introduced toroidal compactifications
%of locally symmetric varieties 
The language of toroidal embeddings was also used by Mumford and his collaborators in \cite{AMRT},  to construct toroidal compactifications of locally symmetric varieties. Toroidal compactifications have many nice properties properties, and posses   simple, easy to describe singularities. They found many applications in the theory of Shimura Varieties.

 One of the technical novelties of the paper is a simple functorial combinatorial desingularization which can potentially lead to many new applications using combinatorial methods (Theorem \ref{can des}). In fact, all the results here can be linked to versions of this algorithm.
 
 The main disadvantage of the toroidal resolution method  is that it can be applied to very special toroidal singularities defined by the binomial equations. 
 %or alternatively by the monoids of monomials ( represented by cones). 
 In practice, it means that  to apply the method in general situation one needs to transform singularities  to  toroidal ones first.

 The structures on  toroidal embeddings defined by the divisor or, equivalently, monomials were further generalized in the language of Fontaine-Illusie logarithmic schemes founded in the papers of Kato \cite{Kato-log}.  This gives a more general viewpoint, where toroidal embeddings are called logarithmically smooth as they form a class of objects similar 
 to the smooth  varieties in the category of   reduced schemes of finite type over the field. Similarly to 
 smooth case the logarithmically smooth varieties have a relatively simple structure of the completions of local rings. They are generated by free parameters and an algebraically independent monomial part which forms a monoid. Moreover, similar to the smooth case,  the module of the logarithmic differentials is free of the rank equal to the dimension of the ring.
 The latter is the direct sum of the free parameters part and a free monomial part corresponding to the groupification of the monoid.

In the recent papers by Abramovich-Temkin-W\l odarczyk \cite{ATW-principalization}, \cite{ATW-toroidal} the authors prove 
 the canonical desingularization of logarithmic varieties.  The resolution is functorial with respect  to arbitrary logarithmically smooth morphisms. 
  The resulting resolved object is, as  dictated by 
  the strong functoriality properties,  a quasi-log smooth variety as in \cite{ATW-toroidal}, or a toroidal orbifold with a locally toric coarse moduli space, as in \cite{ATW-principalization}. The result is thus a counterpart of the Hironaka desingularization  in the logarithmic category. The  functoriality properties imply that the 
  log-smooth (toroidal) locus is unmodified in the process. Quasi log smooth varieties are log smooth in, so-called, Kummer \'etale topology (which allows to extract roots from the monomials). They are  locally toric varieties which are  finite toric quotients of log smooth varieties (toroidal embeddings). In particular, their logarithmic structure is defined by a smooth open subset which is the complement of a certain locally toric divisor.

 In the paper, we give a proof of 
 the canonical (partial) desingularizations of varieties with unmodified  an open toroidal subset: 
 
 %The canonical partial desingularization of logarithmic varieties, Theorem \ref{th: resolution5}.

  We obtain several results in this direction.
 
 We prove a canonical desingularization of locally binomial varieties, and toroidal embeddings Theorems \ref{th: resolution4}, \ref{des toroidal}. Any locally binomial  variety or a toroidal embedding over a field of any characteristic can be resolved canonically by a projective morphism from a smooth variety such that the exceptional divisor is SNC. 
 
 \begin{theorem} \label{th: resolution A} Let $X$ be
 any \'etale locally binomial variety or a toroidal embedding  over a  field $K$ of any characteristic.  There exists a canonical resolution of singularities i.e. a birational projective $f: Y\to X$ such that
 \begin{enumerate}
 \item $Y$ is smooth over $K$.
 \item $f$ is an isomorphism over the open set of the  nonsingular points.
 
 \item The inverse image $f^{-1}(\Sing(X)$ of the singular locus $\Sing(X)$ is a  simple normal crossing  divisor on $Y$.
 %\item If $X$ is a  (Zariski) locally binomial variety then $f^{-1}(\Sing(X)$ is an SNC divisor.
 \item If $(X,D_X)$ is toroidal then $(Y,D_Y)$ is strictly toroidal with $D_Y$ having SNC. Moreover the birational morphism $f:(Y,D_Y)\to (X,D_X)$ is  toroidal. 
 
 \item $f$ is a composition of the normalization and the normalized blow-ups of the locally monomial filtered centers $\{\cJ_{in}\}_{n\in \NN}$ \footnote{Definition \ref{filtered}} 
 defined locally by  valuations.

 \item
 $f$ commutes with smooth morphisms and field extensions, in the sense that the centers are transformed functorially, and the trivial blow-ups are omitted.

%\item Moreover, if $D$  is an \'etale  locally toric divisor $D$ on a \'etale locally toric $X$\footnote{Definition \ref{divisor}} then there is functorial desingularization of $(X, D)$ as above such that the strict transform of $D$ has SNC with $f^{-1}(\Sing(X)$.
 \end{enumerate}

\end{theorem}
Recall that the problem of a functorial canonical resolution of locally binomial or locally toric  varieties was open in positive characteristic. The existence of noncanonical  resolution of locally toric varieties preserving the nonsingular locus was proven over any algebraically closed fields in \cite[Theorem 8.3.2]{Wlodarczyk-toroidal}. 

The Hironaka approach of the  embedded resolution to this problem works well only in the case of toroidal embeddings but fails in the locally binomial or the locally toric situation. We show  in Example \ref{Hir} that so-called  maximal contact which constitutes a basis of the Hironaka method, does not exist even in the locally binomial situation. On the other hand, a version of the "combinatorial maximal contact"  considered for "combinatorial blow-ups",  was constructed by Bierstone- Milman in the case of toroidal embeddings or locally binomial varieties defined by the divisors(\cite{B-M-toric}, see also \cite{Blanco},\cite{Blanco2},\cite{BE}). Using the Hironaka method they prove the
case of embedded desingularization of  toroidal embeddings,
or locally binomial varieties defined by the divisors.

 By modifying the algorithm we  
 show the canonical partial desingularization  except for toroidal subset for  locally binomial varieties, toroidal embeddings, logarithmic varieties, and for any varieties with a Weil divisor: Theorems \ref{th: resolution4}, \ref{des toroidal2}, \ref{th: resolution5}, and \ref{th: resolution6}.
 
 \begin{theorem} \label{th: resolutionB} Let $X$ be a variety
 over a  field $K$ of characteristic zero, and $V\subset X$ be its open subset which is a strict toroidal embedding $(V,D_V)$, defined by the divisor $D_V$.
 Assume that the divisor 
 $D_V$  on $V$  defining the  toroidal structure  on $V$ has  locally ordered components.
  Assume furthemore that one of the following holds
 \begin{enumerate}
 \item	$X$ is a logarithmic variety and $char(K)=0$.
 \item $X$ is any variety with a Weil divisor $D$ and $char(K)=0$.
 \item $X$ is a locally toric variety with a locally toric divisor $D$.
 \item $(X,D)$ is a toroidal embedding.
 \end{enumerate}

 We assume that the divisor $D_V$ is the restriction 
 of the divisor $D$ on $X$, or, as in the case (1), defines the 
 restricted logarithmic structure.
  In the second case we assume that the strict toroidal embedding $(V,D_V)$ is extendable.
  \footnote{Definition  \ref{order}} 
    Then there exists a canonical  resolution of singularities of $X$ except for  $V$ i.e. a birational projective morphism $f: Y\to X$ such that
 \begin{enumerate}
 \item $f$ is an isomorphism over the open set $V$.
 %\item The variety $(Y,V)$ is a toroidal embedding.
\item The variety $(Y, D_Y)$ is a strict toroidal embedding, where $D_Y:=\overline{D_{V,Y}}$ is the closure of the divisor $D_V$ in $Y$.
 Moreover, $(Y, D_Y)$  is the  saturation\footnote{Definition \ref{saturation}}
 of the toroidal subset $(V, D_V)$ in $Y$. (In particular, $(Y, D_Y)$  has the same singularities as $(V, D_V)$.)
 %\item The complement  $E_{V,Y}:=Y\setminus V$ of $V$ in $Y$ is a divisor with  simple normal crossings (SNC)  with $D_Y$ \footnote{Definition \ref{nc}}.
 \item The exceptional divisor $E_{\exc}\subseteq E_{V, Y}$ has simple normal crossings (SNC)  with $D_Y$ \footnote{Definition \ref{nc}}.
 % defined by the sets of valuations.
%\item If $V$ is the toroidal (i.e. log smooth) locus of $(X,\cM)$ then  $E=E_0$. 

%\item If  $(V, D_V)$  is a strict toroidal embedding then $(Y, D_Y)$ is such. 

%\item In particular, if $V$ is smooth and $D_V$ is  an SNC divisor on  $V$ then 
%$Y$ is smooth, and $D_Y\cup E_{V, Y}$ is an  SNC divisor.

 %\item $f$ is a composition of the normalization and the normalized blow-ups of the locally monomial filtered centers $\{\cJ_{in}\}_{n\in \NN}$.
 \item $f$ is a composition of a sequence of the blow-ups at  functorial centers.
 \item
 $f$ commutes with field extensions  and smooth morphisms  respecting  the  structure on $X$, the  subset $V$, and the order of the components of $D_V$.
 % in the sense that the centers of the blow-ups are transformed functorially, and the trivial blow-ups are omitted. 
 \item In particular, if $G$ is an algebraic group acting on $(X,\cM)$ and preserving the logarithmic structure $\cM$ and the subset $V$,  and the components of $D_V$ on (a $G$-stable) $V$ then the action of $G$ on $X$ lifts to $Y$, and $f:Y\to X$ is $G$-equivariant.
%\item The exceptional locus $E$ can be further modified to a relative SNC divisor $E'$ (so that $E'$ has SNC with $D_Y$) by a sequence of blow-ups at relatively smooth centers. 
 \end{enumerate}

%\begin{remark} Note that the singularities on $Y$ are the same as on $V$. It means there exists an equisingular stratification on a toroidal embedding $Y$ , and all its strata interesect open subset $V$.

%\end{remark}

\end{theorem}

  %Theorem \ref{th: resolution5} Theorems \ref{th: resolution2},  \ref{th: resolution4},\ref{th: resolution2}, .
 In other words any such variety containing an open subset which is a strict toroidal embedding can be modified canonically by  a birational projective modification into a toroidal embedding.
 In the process the open subset is untouched  and the resulting variety is a  toroidal embedding with the singularities identical as on 
 the open subset. This means that the (irreducible) equisingular strata of the resolved toroidal embedding extend the strata on  the open unmodified set. Moreover, 
 the exceptional boundary divisor has {\it relatively simple normal crossings}.
 By a {\it relative simple normal crossing divisor} we mean here a divisor
  %in particular, it has normal crossings with the (nonsingular) strata. 
 %Note that the notion of the simple normal crossings divisor on  a smooth variety easilly extends to the logarithmically smooth varieties, as the divisors 
 whose components are locally  described by a part of the coordinate system of free parameters.
 (see  Definition \ref{nc})

The desingularization theorems are proven in the following order.   
 First, we prove the canonical desingularization for locally binomial varieties and partial desingularization for locally binomial varieties with  locally toric Weil divisors over a %perfect 
 field of any characteristic (Theorems \ref{th: resolution2},  \ref{th: resolution4}) \footnote{ Note that our definition of toroidal embeddings differs slightly from the definition of Abramovich-Denef-Karu \cite[Section 2.2]{ADK}.
Both definitions agree over a perfect field.}

 Then, using the canonical desingularization of logarithmic varieties combined with the desingularization of locally binomial varieties with an unmodified open toroidal subset we show a canonical desingularization of logarithmic varieties except for  an open toroidal subset. (Theorem \ref{th: resolution5})
 % logarithmic variety containing open toroidal subset. 
 To further generalize the result for arbitrary varieties with Weil divisors we canonically extend the logarithmic structure  from  the toroidal subset to the whole variety. (Lemma \ref{extension}).
 
 This can be done  for arbitrary varieties with  Weil divisor possessing an open  {\it extendable toroidal subset}. The  extendable toroidal embeddings satisfy certain conditions on restrictions of local Cartier divisors, like for example, toroidal varieties with quotient singularities, toroidal varieties with a single closed stratum or with extendable local Cartier divisors). (Definition \ref{trivia})
 
  As a consequence, we show the existence of
  a functorial partial resolution except for  open  extendable toroidal subsets for  arbitrary varieties with a given Weil divisor. (Theorem \ref{th: resolution6}).

 A particular version of this result shows that any variety with a locally toric singularity at a given point can be modified in such a way that  a neighborhood of the point will remain unchanged, and the resulting variety, when equipped with  a divisorial structure   will become a toroidal embedding with a unique closed 
 stratum passing through the point (Theorem \ref{compactification2}).
  
  Using the desingularization Theorem \ref{th: resolution6} we prove the existence of equisingular toroidal compactification of  extendable toroidal embeddings (Theorem \ref{compactification}).

%The proof relies on two theorems: Canonical desingularization of logarithmic varieties , and canonical desingularization of locally toric varieties with unmodified toroidal subset.  
 
 %In the algorithm we modify the variety with a divisor, and a fixed open toroidal subset in such a way that the logarithmic structure can be extended globally. 
 The canonical desingularization of logarithmic varieties from \cite{ATW-principalization}, \cite{ATW-toroidal} reduces the problem to the quasi-toroidal embeddings. The quasi-toroidal embeddings are,  in particular, locally toric,  with some locally toric divisors. This defines a natural  (non- smooth) stratification induced by a given divisor and the singularity type. The resulting variety is not a toroidal embedding but it is a stratified toroidal variety.  To deal with it we use the theory of stratified toroidal  varieties.
 The theory was developed as a tool in the proof of the Weak factorization theorem \cite{Wlodarczyk-toroidal}.
   It associates   with  the  variety a semicomplex and allows us to run certain (sufficiently functorial) algorithms. In particular, only very special centers of the modifications (star subdivisions) can be used. In section \ref{stratified} we give  a crash course on the theory of stratified toroidal varieties, recalling and reproving a few most  basic results used in the proof.

 To resolve locally toric singularities with an unmodified toroidal subset we develop a simple fast and efficient functorial desingularization combinatorial algorithm, and its relative version.  It  can  be applied to conical complexes and more general semicomplexes. The functoriality properties are critical  for gluing 
the algorithm on the more general objects. 

The method gives  a  functorial  resolution of locally toric or locally binomial  varieties with stratification over a field  in any characteristic and its relative version with an unmodified toroidal subset  (Theorems  \ref{th: resolution2}, \ref{th: resolution4}).

The functorial desingularization algorithm  of locally toric (or locally binomial) varieties over a field of any characteristic is much  simpler, more efficient, and more geometric than Hironaka resolution in characteristic zero. Combined  with the logarithmic desingularization of \cite{ATW-principalization} (\cite{ATW-toroidal}) it gives
  also a  more efficient algorithm of canonical desingularization of arbitrary varieties in characteristic zero. Moreover, the method gives a very good control over resolved singularities
 and allows  to avoid undesired modifications.

Another and perhaps most straightforward application of the combinatorial algorithm is for the toroidal embeddings.  In this case, we obtain a very simple and  
efficient method of  functorial resolution, and its relative version.  
%This, in particular,  preserves the normal crossings locus 
(Theorems \ref{can des}, \ref{th: resolution}, \ref{des toroidal}, and \ref{des toroidal2}).

Recall that an embedded  functorial desingularization of  (not necessarily normal) toroidal embeddings  over a perfect field was  proven, as it was mentioned earlier,  by Bierstone-Milman in \cite{B-M-toric} (see also \cite{Blanco},\cite{Blanco2},\cite{BE}). They  extended resolution methods developed in characteristic zero.
Similar results were shown  by Nizio\l. She was  using a combinatorial interpretation of the simplified Hironaka algorithm in 
\cite[Theorem 5.10]{Niziol-toric}. The non-embedded resolution in characteristic 0, preserving a simple normal crossing locus  was also proven by Illusie-Temkin \cite[Theorem~3.3.16]{Illusie-Temkin}, Gillam-Molcho\cite[Theorem 9.4.5]{Gillam-Molcho}.  
Another simple combinatorial method  was provided in \cite[Theorem 4.4.2]{ACMW} by  Abramovich-Chen-Marcus-Wise. Their method is, however nonfunctorial and  modifies the normal crossings locus.

The combinatorial algorithm in this paper is based upon the ideas developed in \cite{KKMS}, and is closely related to the method considered in \cite{ACMW}. It is  fully functorial  and preserves  
the normal crossings locus.
The method  combines a version of barycentric subdivision with  the lattice reduction algorithm of \cite[Theorem 11*]{KKMS} with respect to a certain natural order introduced in section \ref{marking}. Unlike other functorial desingularizations ours
%, in particular,  mentioned above Hironaka's  toroidal desingularization in positive characteristic,  ours 
does not 
depend upon the toroidal structure and is controlled by simple geometric invariants without additional bulk. 
When forgetting about the divisors defining the  toroidal (or log smooth) structure we are left with locally toric varieties, which reduces the language to the previous situation of stratified toroidal varieties (Theorems  \ref{th: resolution2}) without changing the algorithm.
This  also explains why the algorithm works in a locally toric case in positive characteristic.  On the other hand, the presentation of the algorithm in the paper from toroidal embeddings to locally binomial varieties, illustrates the main feature  of the theory of stratified toroidal varieties which  studies the combinatorial modifications independent of locally toric coordinates.

The paper is organized as follows. In Chapter \ref{main} we formulate and prove the main theorems using the desingularization theorems  from Chapter \ref{stratified}.
In Chapter \ref{complexes} we introduce the basic definitions and results on toroidal embeddings, and conical complexes.  Chapter \ref{Desi} is entirely devoted to proof for  canonical desingularization of conical complexes.
 
 In Chapter \ref{relative complexes}, we introduce the language of relative conical complexes, and 
prove the  relative version of the canonical desingularization of the conical complexes.
The algorithm in the relative version is nearly identical, and the introduced notions are perfectly analogous to the standard nonrelative situation. However, the language of relative complexes is somewhat more involved and thus perhaps less intuitive than the language of complexes. That is why we deal separately with the nonrelative and relative cases, although the first one is the particular case of the second, with the trivial relative structure.

Chapter \ref{desing-toroid} contains a proof of the functorial desingularization of toroidal embeddings, and its relative version with unmodified open subset. The results are quite immediate consequences of the canonical desingularization of complexes in Chapters  \ref{Desi}, \ref{relative complexes}. 

Finally, Chapter \ref{stratified} contains a proof of the functorial desingularization of locally toric varieties over the fields and its relative version. The main tool in the proof is a theory of the stratified toroidal varieties. One  constructs a stratification defined by the singularity type and by a given divisor (in the relative situation). These data  define  the associated  conical semicomplex, which is , roughly, a collection of cones defined up to automorphisms groups associated with strata and some rather sparse
face relation defined by the strata %(Lemma \ref{le: associated semicomplex}).

The functoriality of the algorithm developed in Chapters \ref{Desi}, \ref{relative complexes} for complexes
and the canonicity of the centers allow us to run it on semicomplexes to give rise to the desingularization of stratified toroidal varieties.

\section{Main results}\label{main}
\subsection{Desingularization of logarithmic varieties except for log smooth locus}
 \subsubsection{Logarithmic varieties}
 The logarithmic structures are used in this paper only in the formulation of Theorem \ref{th: resolution5}, and in the Extension Lemma \ref{extension}, in Chapter \ref{main}.

 Recall that a {\it logarithmic structure} on a scheme $X$ of  finite type is given by a sheaf of monoids $\cM_X$ admitting a morphism  of the sheaves of monoids $j:\cM_X\to \cO_X$ under multiplication, such that $j^{-1}(O_X^*)\simeq \cO^*_X$, where 
 $\cO^*_X$ is the sheaf of  monoids of the invertible regular functions. 
 A { logarithimic structure} is called {\it coherent}  if  for any point $x\in X$  there is  a certain \'etale neighborhood $i_U: U\to X$, and a map of monoids $ P\to \Gamma(U,i_U^*(\cM_X))$, called {\it chart},
where $P$, is a   finitely generated monoid.
The map of the monoids
   induces a map of sheaves of monoids 
 $\alpha: P^a\to i_U^*(\cM_X)$, where $P^a$ is a locally constant sheaf defined by $P$.
 Moreover, we assume that
 $i_U^*(\cM_X)$ is isomorphic to the push out of $P^a\leftarrow \alpha^{-1}(\cO_U^*)\to \cO_U^*$.
 \cite[Sections 1.1, 1.2]{Kato-log} \cite[Sections 1.2]{Kato2}
 
 A monoid $P$ is {\it fine }   if   $P$ is finitely generated and  has no zero divisors, so injects in its finitely generated groupification $P^{gp}$. It is {\it fine and saturated} if it is fine
 and for any $a\in P^{gp}$, such that $a^n\in P$ we have that $a\in P$. In particular, the monoids defined by the intersections of rationally generated cones with lattices is an example of fine and saturated monoids. 
 
 A { coherent logarithimic structure} 
 is called {\it  fine and saturated} (fs) if  any point admits a chart $ P\to \Gamma(U,i_U^*(\cM_X))$, with $P$ being fine and saturated.\cite[Sections 1.1, 1.2]{Kato-log}

 Any  scheme, with a coherent logarithmic structure can be made canonically into a variety or a scheme with a {fine and saturated} structure by the natural canonical procedure, called {\it saturation} \cite[Proposition 1.2.9]{Kato2}, \cite[Proposition 2.1.5]{Ogus-logbook}. It is \'etale locally described by 
 $$X^{\sat}:=X\times_{\Spec{\ZZ[P]}}{\Spec{\ZZ[P^{\sat}]}},$$ with $$P^{\sat}:=\{a\in P^{gp} \mid a^n\in i(P) \},$$
 where $i:P\to P^{gp}$ is the natural map defined by the groupification. \cite[Proposition 1.2.9]{Kato2}
\cite[Page 1-2]{Cailotto}. 

 A logarithmic structure, which is  a sheaf of monoids $\cM_X$, is  usually defined in the \'etale topology (for  functoriality properties).
Note that, when working over nonclosed fields or $\ZZ$ one needs to pass to \'etale neighborhoods to have  nice properties of local rings and the corresponding monoids.

 A {coherent  or fs logarithmic structure} will be called {\it strict} if the charts are defined in the Zariski topology.

  \subsubsection{Stratifications on  schemes with logarithmic structures}
  Let $X$ be a scheme. A collection $S$ of disjoint locally closed subsets  of $X$, called {\it strata} will be called {\it stratification} if 
the closure of a stratum is a union of strata, and $X$ contains an open dense stratum. This defines  a natural order on the strata induced  by  
generization:
$$s\leq s'\quad \mbox{iff} \quad   \overline{s}\subseteq \overline{s'}.$$
 The closures of strata in the stratification will be called the {\it closed strata}. The collection of the corresponding closed strata will be denoted by $\overline{S}$.
 
 In this paper, we  shall consider  the resolution of 
 logarithmic schemes in two different general situations: 
 
 By a {\it logarithmic variety} (respectively a  {\it strict logarithmic variety})  we mean a variety equipped with a fine and saturated logarithmic structure $\cM$ (respectively a strict fs logarithmic structure). Such a variety possesses 
 a natural stratification  defined by the {\it rank} of the associated monoids $\cM/\cO^*$ \cite{Illusie-Temkin}, \cite{AT1}.
 
 On the other hand, we consider varieties with the logarithmic structure
defined by an open subset $U\subset X$.   The logarithmic structure $\cM$ on  $(X,U)$ is defined (in \'etale topology) as $$\cM=(\cO_X)\et\cap j_*(\cO_{U}^*)\et,$$ where $j:U\to X$ is the open immersion. The strict logarithmic structure is defined in the Zariski topology as $$\cM=(\cO_X)\cap j_*(\cO_{U}^*).$$ 
  
  If the  open subset $U$ is the complement of  a Weil divisor $D:=X\setminus U$, then we shall often use the divisor $D$ to describe the logarithmic  structure $\cM_D$ associated with $U$ identifying the logarithmic structure on $(X,D)$ with $(X,U)$.
 
 \subsubsection{Divisorial stratification}\label{divisorial}
 In general, the logarithmic structure on  $(X, D)$ is not coherent and does not have the stratification given by the rank. % \cite{Illusie-Temkin}, \cite{AT1}. %It is so when $D$ is a Cartier
  However, one can  consider a divisorial stratification ${S}_D$ instead. Write
$D=\bigcup_{i\in J} D_i$ with irreducible Weil components $D_i$. The {\it closed strata} of $\overline{S}_D$  are defined by  the irreducible components of the intersections  $\bigcap_{i\in I} D_i$ of $D_i$. 
 The strata of  $S_D$  are the components of  $$\bigcap_{i\in I} D_i\setminus (\bigcup_{i\in (J\setminus I)} D_i).$$ 
 %and by $\overline{S}_D$ the set of the induced closed strata.

  % The  closed strata of the divisorial stratification $S_D$ are  defined by the irreducible components of the intersections of
%the components of the divisor $D$. Both stratifications coincide for toric varieties and strict toroidal embeddings.

 \subsubsection{Logarithmic structures on toroidal embeddings}
 
  Any toric variety $X\supset T$ over a base field $K$ admits a natural stratification by the orbits. It coincides with both stratifications: divisorial, defined by $D=X\setminus U$, and the one given by the rank of the rank of the logarithmic structure.
  
  Toroidal and strict  toroidal embeddings were introduced
 in \cite{KKMS} by Mumford and others (initially over an algebraically closed field). They are modeled by toric varieties. 
 
 We  review this theory over nonclosed fields in Section \ref{toroidal embeddings}. 
 Strict toroidal and  toroidal embeddings 
   are defined by an open subset $U\subset X$, and are locally (respectively locally in \'etale topology) \'etale isomorphic to toric varieties $(X_\sigma,T)$ with open torus $T$ corresponding  to open subset $U$.
  
  By definition, the sheaf of monoids $\cM$ is generated at any point $x\in X$ by the monoid of the effective Cartier divisors $\Cart^+_x(X, D)$ supported on $D$ and defined in a neighborhood (respectively \'etale neighborhood of $x$), so we can write  $$\cM_x=\Cart^+_x(X, D)\cdot \cO_{X,x}^* \subset \cO_X.$$

  The toric stratification by the orbits on a toric model $(X_\sigma, T)$ induces locally (respectively \'etale locally) the canonical stratification $S$ on  strict toroidal and toroidal embeddings. (see Section \ref{conical}). It coincides with the divisorial stratification in the case of strict toroidal embeddings.

 This stratification is equisingular, which  means that the completion of local rings at geometric $\overline{K}$-points on a stratum are isomorphic, where $\overline{K}$ is the closure of a base field $K$.

Each stratum $s$ on a strict toroidal embedding defines an open subset, called the {\it star} of $s$, defined as
$$\Star(s,S):=\bigcup_{ s'\geq {s} }{s'},$$
which is an open  neighborhood of $s$.

One can associate with a stratum $s$ the monoid $\Cart^+(s,S)$ of the effective Cartier divisors on $\Star(s,S)$ supported on  $\Star(s,S)\cap D$. This monoid is isomorphic to $\Cart^+_x(X, D)$ for any $x\in s$, with the natural isomorphism $$\Cart^+(s,S)\to \Cart^+_x(X, D)$$ given by the restriction.

The monoid $P:=\Cart^+(s,S)$ generates the logarithmic structure $\cM_{\Star(s,S)}$ on $\Star(s,S)$:
 $$\cM_{\Star(s,S)}=\Cart^+(s,S)\cdot \cO_{\Star(s,S)}^* \subset \cO_{\Star(s,S)}.$$ 
 For any point $x\in \Star(s,S)$ there is a natural local injective chart $P\to \cM(U)$, where $U\subset \Star(s,S)$ is an open neighborhood of $x$. Moreover
 the stratum $s$ on $U$ is defined as vanishing locus of $P\setminus \{1\}$.
 
 As it was observed in \cite{KKMS}, there exists a conical complex $\Sigma$ associated with a strict toroidal embedding $(X, D)$, with faces $\sigma\in \Sigma$  in bijective correspondence with strata $s=s(\sigma)\in S$ (see Section \ref{conical}, and Theorem \ref{cc}). 
 %\subsubsection{Logarithmic maps. Smooth toroidal maps}
 
 %By the maps of  logarithmic varieties $(X,\cM)$ we
 
 \subsubsection{Saturated toroidal subsets}
  
  \begin{definition}  \label{saturation} By a {\it toroidal subset}  of 
  a scheme with logarithmic structure $(X,\cM)$     
   we mean an open subset  $V\subset X$, with a    Weil divisor  $D_V$ such that $(V, D_V)$ is a strict toroidal embedding, which   defines the restricted log structure  $\cM_{|V}$. 
    By the {\it toroidal saturation} or simply {\it  saturation} of an open toroidal subset $(V, D_V)$ of  $(X,\cM)$ we mean the  maximal open toroidal subset $(V^0, D^0)$  of $(X,\cM)$ which       contains $(V, D_V)$, such that all the strata on $(V^0, D^0)$ intersect  the subset $V$ (so extend the strata in $(V, D_V)$). 
 
 A toroidal subset $(V, D_V)$ of  $(X,\cM)$ is called {\it  saturated} if it is equal to its  saturation. 
\end{definition}

 %Note that the logarithmic structure defined by the Weil divisors is not coherent, and the example of the "pinch point" shows that the operation of the restriction may behave quite badly.  This is reflected in the definition below.
 %\begin{definition}  By a {\it  toroidal subset}  of a variety with a Weil divisor $(X, D)$  we mean an open subset  $V\subset X$, with the induced divisor $D_V:=V\cap D$ such that  $(V, D_{V})$ is a strict toroidal embedding.    
 
 %By a {\it toroidal subset}  of a variety with a Weil divisor $(X, D)$  we mean an open subset  $(V, D_V)$, for which  there exists  an \'etale cover $\{ U_i\to X\}_{i\in I}$, such that  $(V_{U_i}, D_{V_{U_i}})$ is a strict toroidal embedding , where $\pi_i: V_{U_i}:=V\times_X U_i \to V$ is the natural projection, and $D_{V_{U_i}}:=\pi_i^{-1}(D_V)$ 
%\end{definition}

 %\begin{definition} \label{saturation} Let $(X,\cM)$ be a scheme with logarithmic structure.  
     %By the {\it toroidal saturation} or simply {\it  saturation} of an open toroidal subset $(V, D_V)$ of  $(X,\cM)$ we mean the  maximal open toroidal subset $(V^0, D^0)$  of $X$ which       contains $(V, D_V)$, such that all the strata on $(V^0, D^0)$ intersect  the subset $V$ (so extend the strata in $(V, D_V)$). 
 
 %A toroidal subset $(V, D_V)$ of  $(X,\cM)$ is called {\it  saturated} if it is equal to its  saturation.  
%\end{definition}

The above  definitions apply to logarithmic varieties, and   varieties with Weil divisor.

\begin{example} In the case of the strict toroidal embedding any open subset is toroidal. The saturated subsets are just those which are the unions of some stars.

\end{example}

\begin{definition} \label{locus} The largest saturated toroidal subset  of  $(X,\cM)$ is called the {\it toroidal locus} (or the  {\it log smooth locus}).  It is the set of all points of $X$ where   $(X,\cM)$ is a strict toroidal embedding. We denote it by $(X,\cM)^{\tor}$.

\end{definition}

\begin{remark}\label{locus2}
 Observe that since the strata on  toroidal embeddings 
 are equisingular, any toroidal subset $V\subset X$ 
  has the same singularities as its toroidal saturation $V^0$. So the saturated subsets, in particular, represent the sets of all the points with certain given types of toroidal singularities (including information on the divisor). For example, the smallest nonempty saturated toroidal subset on $X$ is the set of all nonsingular points $(X^0)^{ns}$ of the set $X^0$ of the points where the  logarithmic structure is trivial $\cM=\cO^*$ (resp. $X^0:=X\setminus D$). The largest saturated toroidal subset  on $X$ is its {\it toroidal locus} $X^{\tor}$ . There are finitely many  saturated toroidal subsets on $X$, as these are exactly the unions of the stars of  strata on $X^{\tor}$.
 %This follows from the lemma:  
  %In fact any toroidal stratum in $(\overline{X},\overline{D})$ intersects  $X$ at the toroidal stratum so the 

\end{remark}
\begin{lemma} \label{extension1} Let $(X, D)$ be a toroidal embedding  embedding (which is not necessarily strict). Let $(V, D_V)$ be its  toroidal subset intersecting all the strata of $(X, D)$. Then  $(X, D)$ is strict toroidal and it is the toroidal saturation of $(V, D_V)$.
%is  its saturation. Then the associated conical complexes (describing the singularities) are the same for both.
 \end{lemma}
\begin{proof} 
 
%By definition,  $(V^0, D_{V^0})$ is a  toroidal embedding. We need to show that it is strict.
 
 Let $x$ be any point in a certain stratum $s$ on $X$. Consider   a Weil divisor $E$ consisting of the components of $D$, passing through $x\in E$ . Let $\pi: (U, D_{U})\to (X, D)$ be an \'etale neighborhood of $x$ which is strict toroidal. This map being \'etale preserves the strata (of the rank stratification), so that the inverse image of the  strata on $(X, D)$ are the strata on $(U, D_{U})$. Let $\overline{x}\in \overline{s}$ be a point over $x\in s$, where $\overline{s}$ is the stratum  over $s$.
 
 It defines an \'etale map between the relevant strict toroidal embeddings $\pi_{|U}:(U_V, D_{U_V})\to  (V, D_V)$, where $U_V:=\pi^{-1}(V)$. By  Lemma \ref{Mum}, there exists a bijective correspondence between  the components of the Weil divisor of $E$ passing through  a generic point of $s$ and the components of $D_U$ passing through a generic point of $\overline{s}$. 
 
%There exists a bijective correspondence between  the components of the Weil divisor of $E$ (passing through $x$) and the components of $D_U$ passing through a point $\overline{x}\in \overline{s}$ over $x\in s$, where $\overline{s}$ is the inverse image of $s$.
%This follows from the  fact that such a correspondence exists at the  local rings of generic point of the strata $\overline{s}$ and $s$, where it is defined by $\pi_{|U}$ and both varieties are strict toroidal. In such a case the irreducible components on the  open subsets are normal, and the inverse image of an irreducible  component is also normal so locally irreducible.
%see Remark \ref{normal}).
%The components correspond to the toric irreducible divisors on the corresponding toric models.

%at their isomorphic completions.
But the components of  the divisor $E$ (respectively $D_U$) through $x$ and $s$ ( respectively $\overline{x}$ and $\overline{s}$) are the same. This follows from the construction of  the divisorial stratification on \'etale neighborhood $U$ and the stratification on $X$ with strata which are locally the images of strata on \'etale neighborhoods.

 Consequently, the natural surjection between the irreducible components of Weil divisor $D_{U}$ through $\overline{x}$ and $E$  through $x$ defined by $\pi: \Spec(\cO_{\overline{x},U})\to \Spec(\cO_{x,X})$ is a bijection. Thus the inverse image of an irreducible  divisorial component is locally an irreducible divisorial component.

Consequently, each such a component of $E$ itself  is a toroidal embedding  since it is strict toroidal in \'etale topology. (Its inverse image under \'etale morphism is a strict toroidal embedding). 
So, in particular, each component of $E$ is normal. This implies, by Lemma \ref{Mum2}, that $(X, D)$ is  strictly toroidal.  
\end{proof}

\begin{corollary} \label{extension2} Let $(X, D)$ be a toroidal embedding  embedding (which is not necessarily strict). Let $(V, D_V)$ be its  saturated toroidal subset. Then for any \'etale morphism $\phi: (U, D_U)\to (X, D)$ the inverse image $\phi^{-1}(V)$ is a saturated toroidal subset of $(U, D_U)$.
\end{corollary}

\begin{proof} Let $V':=\phi^{-1}(V)^\sat$ be the saturation of $\phi^{-1}(V)^\sat$. Then its image \\$(\phi(V'), D\cap\phi(V')) $ is a toroidal embedding such that all the strata on $\phi(V')$ intersect $V$. Then, by Lemma \ref{extension1}, $\phi(V')$ is a 
strict toroidal embedding, and thus it is contained in the saturation $V^\sat=V$. This implies  that $\phi(V')=V$, and $V'=\phi^{-1}(V)$.

\end{proof}

  \subsubsection{Relative SNC divisors on toroidal varieties}
  In the classical Hironaka desingularization the exceptional locus has SNC (simple normal crossings). In the relative desingularization the exceptional divisor has similar properties.
  \begin{definition} \label{free} Let $(X, D)$ be a toroidal embedding. By  a {\it free coordinate system} on $(X, D)$ at a point $x$ we mean a set of parameters $u_1,\ldots,u_k$ on an open subscheme $U$ of  $X$ %generating stratum $s$ through $x$, and 
  such that there is   an \'etale  neighborhood  $U^{\et}\to U$ of $x$, and  a monoid $P$, with 
an  injective chart  $$P\to \cM({U^{\et}})=\cO({U^{\et}})\cap \cO^*({{U^{\et}}\setminus D_{U^{\et}}}),$$  and an \'etale  morphism ${U^{\et}}\to \Spec{K[x_1,\ldots,x_k,P]}$ defined by $x_i\mapsto u_i $, and $P\to \cO({{U^{\et}}})$, where $x_1,\ldots,x_k$ are free variables independent of $P$, 
 and the ideal $\cI_s$ of the stratum $s$ through $\overline{x}$ over $x$ is generated by $P\setminus\{1\}$.
 
  \end{definition}
\begin{remark} Equivalently $u_1,\ldots,u_k$ is a free coordinate system if their restrictions to the stratum $s$ define a local parameter system on a smooth subvariety $s$.

\end{remark}

  \begin{definition} \label{nc} Let $(X, D)$ be a toroidal embedding  and $E$ be a Weil divisor. We say that $E$ has {\it normal crossings} (NC) (resp. {\it simple normal crossings} (SNC)) on $(X, D)$ % (or with $D$),
   if  it  is \'etale locally (resp. locally) defined by a part of a free coordinate system.  Equivalently we shall call $E$ a {\it relative NC divisor} (respectively a {\it relative SNC divisor}) on $(X, D)$.  
  
  \end{definition}
\begin{remark} It follows from Lemma \ref{2}, that if $E$ has  SNC on a  strict toroidal $(X, D)$ 
then then there is a part of a free coordinate system $u_1,\ldots,u_k$ locally on an open $U$, and an injective chart  $P\to \cM(U),$  with  \'etale  morphism ${U}\to \Spec{K[x_1,\ldots,x_k,P]}$, defined by $x_i\mapsto u_i $, and $P\to \cO({{U^{\et}}})$.
\end{remark}

One can rephrase Definition \ref{nc}
\begin{lemma} \label{1} A divisor $E$ has { normal crossings} (NC)  on $(X, D)$ iff 
$(X, D\cup E)$ is a toroidal embedding  and  any  point $x\in X$ admits an \'etale   neighborhood  $\alpha: U\to X$ of $x$ and \'etale morphism $$\phi: U\to X_{\sigma}\times X_{\tau}=X_{\sigma}\times \AA^n,$$ where $X_\tau=\AA^n$
such that, $\alpha^{-1}(D)_{|U}=\phi^{-1}(D_1\times X_{\tau})$,  
and $\alpha^{-1}(E)_{|U}=\phi^{-1}(X_{\sigma}\times E_1)$,
with a toric divisor $D_1$ on  $X_{\sigma}$, and $E_1$ is an SNC toric divisor on $X_{\tau}=\AA^n$. \qed \end{lemma}

\begin{lemma} \label{2} With the preceding notation and the assumptions. The divisor $E$ has {\it simple normal crossings} (SNC)  on $(X, D)$ if it has NC and its components are normal.
If, additionally, $(X, D)$ is a strict toroidal embedding then $(X, D\cup E)$ is such. Moreover, any point $x\in X$ admits a neighborhood $ U$  and an \'etale morphism $$\phi: U\to X_{\sigma}\times X_{\tau}=X_{\sigma}\times \AA^n,$$ where $X_\tau=\AA^n$
such that, $D_{|U}=\phi^{-1}(D_1\times X_{\tau})$,  
and $E_{|U}=\phi^{-1}(X_{\sigma}\times E_1)$,
with a toric divisor $D_1$ on  $X_{\sigma}$, and $E_1$ is an SNC toric divisor on $X_{\tau}=\AA^n$.
\end{lemma}
\begin{proof} If $(X, D)$ is a toroidal embedding and the components of $E$ are locally defined
by free parameters on a toroidal embedding, then, by definition, they are toroidal embeddings and hence normal.
Conversely, if $E$ has NC on $(X, D)$, and its components are normal then, by definition, the components of $E$ are \'etale locally defined by  local parameters. Denote by $E'$ the inverse image of $E$ on an \'etale neighborhood.
 Since the components of $E$ are normal on $X$  then, by Lemma \ref{Mum},  the irreducible components of $E'$ on \'etale neighborhood of $X$ are locally the inverse images of the components on $X$. By  definition, the components of $E'$ are defined by  a part of a  free coordinate system in the \'etale neighborhood.
 Thus, by the second part of Lemma \ref{Mum},  these components descend to  Cartier divisors defined locally by a part of free coordinate system  on $(X, D)$ describing components of $E$. Thus $E$ has  SNC on $(X, D)$.
 % (in the sense of Definition \ref{nc}). 

Now, if   $(X, D)$ is strict toroidal, then all the components of $D\cup E$ are normal, and since $(X, D\cup E)$ is toroidal, it is also strict toroidal by Lemma \ref{Mum2}. Thus  there is locally an \'etale morphism to $X_\sigma\times X_\tau$.  
%$(X, D\cup E)$ is a toroidal embedding then it is easy to see that 

\end{proof}

\begin{remark}

If $E$ is relative NC (respectively SNC) on $(X, D)$ then its 
% has  (NC) SNC with the strata of $(X, D)$.
 restriction  $E\cap (X\setminus D)$ to the logarithmically trivial locus $(X\setminus D)$ is an NC (resp. SNC) divisor in the usual sense on a smooth subset $(X\setminus D)$.
\end{remark}  

  \begin{lemma} Let $(X^0, D_{X^0})$ be a strict toroidal embedding, and $E$ has NC with $D_{X^0}$. Set $X:=X^0\setminus E$, $D_X:=D_{X^0}$. Then $(X^0, D_{X^0})$ is the  saturation of $(X, D_X)$.
  \end{lemma}
\begin{proof} We verify this property on a strict toroidal \'etale neighborhood, where it reduces to the obvious fact   for toric varieties that $X_{\sigma}\times \AA^n$ is the  saturation of $X_{\sigma}\times (\AA^n\setminus E)$. Then, by Lemma \ref{extension1}, $(X, D_X)$ is strict toroidal, and is the toroidal saturation of $(X^0, D_{X^0})$.

\end{proof}

%\begin{lemma} If $E$ is relative NC divisor on a toroidal $(X, D)$ then there exists a sequence of blow-ups at relatively smooth strata of $E$ such that the resulting disor $E'$ is relative SNC.

%\end{lemma}

\subsubsection{ Divisors with locally ordered components}

%\begin{definition} Let $X$ be a variety over $K$, and $D$ be a  Weil divisor.  Denote by $\overline{K}$ the algebraic closure of $K$.
%We say that $D$ is {\it geometrically irreducible} if the scheme $D\times_{\Spec(K)}{\Spec(\overline{K})}$ is irreducible. 
%\end{definition}

 \begin{definition} \label{order}Let $(X, D)$ be a strict toroidal embedding.
 We say that  a  Weil divisor $D$ on $X$ has {\it locally ordered components}  if 
there is given a partial order on the set of components, which is total for any subset of the components passing through a common point. 
%We say that $D$ has  {\it \'etale locally ordered components} if it has locally ordered components on a certain \'etale cover of $X$, such that for any two \'etale maps $U,V\to X$ in the \'etale cover the order induced on $U\times_XV$ by the projections to $U$ and $V$ coincide.

 \end{definition} 
 \begin{remark}
 %\begin{enumerate}
  %\item By Lemma \ref{Mum} on a strict toroidal embedding $(X, D)$  both notions coincide: $D$ has  locally ordered components if and only if it has \'etale locally ordered components as there is a bijective  correspondence between the components passing through a point $x$ on $X$ ,and those  on \'etale neighborhood $Y$ passing through $y\in Y$ over $x$.

 %\item If $D$ has  { \'etale locally ordered} components  then  it has locally { \'etale locally ordered}  components when passing to its separable field extension.

 The condition of ordering components  for the canonical desingularization except  toroidal subset is  unavoidable given  Example \ref{obstruction}!

%\item The condition of geometric irreducibility of the components of divisor $D_V$ in  the desingularization of $X$ with unmodified toroidal subset $(V, D_V)$  is needed for the functoriality with respect to the field extensions.
%\end{enumerate}
\end{remark}

%\subsubsection{Canonical desingularization of logarithmic varieties }
%\begin{definition} By a {\it locally toric} variety we mean a variriety over $K$ locally admitting an \'etale maps into toric varieties. By a {\it locally toric} divisor on a locally toric variety we ????

%\end{definition}

%\begin{theorem}(\cite{ATW-toroidal}) \label{Th:nonembedded 2}
 %Let $X$ be a fine and saturated logarithmic variety of  finite type over $K$
%and assume that $X$ is generically logarithmically smooth and locally equidimensional.

%Then there is a projective morphism  $X'\to  X$  which is an isomorphism over the logarithmically smooth locus of $X$ and such that $X$ is a quasi-toroidal (embedding).
 %The process is functorial for arbitrary log smooth (and, in particular, smooth) morphisms.
%\end{theorem}

\subsubsection{Desingularization of logarithmic varieties except for  log smooth locus}

 %\subsubsection{Logarithmic varieties}

 \begin{theorem} \label{th: resolution5} Let $(X,\cM)$
 be a logarithmic variety over a  field $K$ of characteristic zero. Let  $V$ be an open  toroidal  subset of  $X$ \footnote{Definition \ref{saturation}}. Assume that the divisor 
 $D_V$  on $V$  defining the  smooth logarithmic ( toroidal) structure  on $V$ has  locally ordered components.\footnote{Definition  \ref{order}} 
    
    There exists a canonical  resolution of singularities of $(X,\cM)$ except for  $V$ i.e. a birational projective morphism $f: Y\to X$ such that
 \begin{enumerate}
 \item $f$ is an isomorphism over the open set $V$.
 %\item The variety $(Y,V)$ is a toroidal embedding.
\item The variety $(Y, D_Y)$ is a strict toroidal embedding, where $D_Y:=\overline{D_{V,Y}}$ is the closure of the divisor $D_V$ in $Y$.
 Moreover, $(Y, D_Y)$  is the  saturation\footnote{Definition \ref{saturation}}
 of the toroidal subset $(V, D_V)$ in $Y$. (In particular, $(Y, D_Y)$  has the same singularities as $(V, D_V)$.)
 \item The complement  $E_{V,Y}:=Y\setminus V$ of $V$ in $Y$ is a divisor with  simple normal crossings (SNC)  with $D_Y$ \footnote{Definition \ref{nc}}.
 So is the exceptional divisor $E_{\exc}\subseteq E_{V, Y}$.
 % defined by the sets of valuations.
%\item If $V$ is the toroidal (i.e. log smooth) locus of $(X,\cM)$ then  $E=E_0$. 

%\item If  $(V, D_V)$  is a strict toroidal embedding then $(Y, D_Y)$ is such. 

\item In particular, if $V$ is smooth and $D_V$ is  an SNC divisor on  $V$ then 
$Y$ is smooth, and $D_Y\cup E_{V, Y}$ is an  SNC divisor.

 %\item $f$ is a composition of the normalization and the normalized blow-ups of the locally monomial filtered centers $\{\cJ_{in}\}_{n\in \NN}$.
 \item $f$ is a composition of a sequence of the blow-ups at  functorial centers.
 \item
 $f$ commutes with field extensions  and smooth morphisms  respecting  the logarithmic structure, the  subset $V$, and the order of the components of $D_V$.
 % in the sense that the centers of the blow-ups are transformed functorially, and the trivial blow-ups are omitted. 
 \item In particular, if $G$ is an algebraic group acting on $(X,\cM)$ and preserving the logarithmic structure $\cM$ and the subset $V$,  and the components of $D_V$ on (a $G$-stable) $V$ then the action of $G$ on $X$ lifts to $Y$, and $f:Y\to X$ is $G$-equivariant.
%\item The exceptional locus $E$ can be further modified to a relative SNC divisor $E'$ (so that $E'$ has SNC with $D_Y$) by a sequence of blow-ups at relatively smooth centers. 
 \end{enumerate}

%\begin{remark} Note that the singularities on $Y$ are the same as on $V$. It means there exists an equisingular stratification on a toroidal embedding $Y$ , and all its strata interesect open subset $V$.

%\end{remark}

\end{theorem}
\begin{proof} Let $X$ be a logarithmic variety.  We can assume that the  complement $X\setminus V$ is the support of a Cartier divisor $E_X$. To this end 
consider the blow-up of the ideal of the complement $X\setminus V$, and let $E_X$ be the exceptional divisor.  We can also assume that the logarithmic structure $\cM$ is not trivial anywhere outside of $V$, by replacing $\cM$ with the saturation of the log-structure generated by $\cM$ and $E_X$ (The push-out $\cM\leftarrow {\cO^*}\to \cM_{E_X}$.) Then, in particular, the components of $E_X$ are the closed irreducible strata of $\cM$. This is done as a preliminary step to ensure the condition (3) of the Theorem. The open subset $V$ is saturated on the newly created logarithmic variety $(X,\cM)$. Then the open subset $V_0:=V\setminus D_V$  coincides with the logarithmically trivial locus $X^{tr}$ of $X$.

By  \cite{ATW-principalization}, and \cite{ATW-toroidal} we can canonically desingularize $(X,\cM)$ that is transform birationally $X$ to a {\it quasi- toroidal embedding}  $\overline{X}$ (modifying its logaritmic struture on the way).  
The process is functorial for  logarithmically smooth morphisms, and thus  it preserves a toroidal subset
$V$.
% Moreover, the inverse image of $E_X$ is the union of the irreducible strata as it is in the complement of the set where the logarithmic structure is not trivial. So $V$ is saturated in $(X, D)$.
%Moreover, if $X$ is a strict logarithmic variety then by \cite{ATW-toroidal}, the {\it quasi- toroidal embedding}  $\overline{X}$ is also strict.

The quasi-toroidal embedding $\overline{X}$ is , in particular, \'etale locally toric, with the 
locally toric divisor $$E_{V_0,\overline{X}}:=\overline{X}\setminus V_0=\overline{D_{V,\overline{X}}}\cup E_{V,\overline{X}},$$ where $\overline{D_{V,\overline{X}}}$ is the closure of $D_V\subset V$ in $\overline{X}$, and $E_{V,\overline{X}}:=\overline{X}\setminus V$.
%This, by Theorem \ref{can}, defines on  $\overline{X}$ the structure of the stratified toroidal
%variety with a locally toric divisor $\overline{D}$. 
%Consider its  toroidal (log smooth) locus $W\supseteq V$. 

By Theorem  \ref{th: resolution4},  applied to the open saturated toroidal subset $(V, D_V)\subset (\overline{X},E_{V_0,\overline{X}})$, the variety $\overline{X}$ can be transformed into a strict toroidal embedding $(Y, D_Y)$, where  $D_Y$ is the closure of  $D_V$, and   the complement divisor $E_{V,Y}:=Y\setminus V$ of $V$ in $Y$ is an SNC divisor on $(Y, D_Y)$.
% and the union of strata so it is the toroidal divisor\footnote{Definition \ref{toroidd}}  on the toroidal embedding $(Y^0, D_Y^0)$. By applying desingularization except for $V$ from Theorem \ref{des toroidal2}  we finally  transform the toroidal embedding $(Y^0, D_Y^0)$ into a  toroidal embedding for which the complement of $V$ is a relative simple normal crossing divisor.
The conditions (1)-(7) follow from Theorem  \ref{th: resolution4}. 
\end{proof}

\begin{remark} Consider a divisor  $D$ with the ''pinch point'' singularity defined by $y^2=x^2z$ in $\AA^3$. 
The natural logarithmic structure on $\AA^3$ defined by the complement of $D$   is not 
fine and saturated (not even coherent). 
Its restriction to  $V:=\AA^3\setminus \{0\}$ defines a toroidal embedding with NC singularities.
 The ''pinch point'' singularity however, cannot be resolved without modifying  the NC locus $V$. 
The open subscheme $V$ is not a   toroidal subset of $(\AA^3, D)$ since it is not a strict toroidal embedding. (see also \cite{B-M-except-I}).
%The toroidal subsets on  the nonsingular logarithmic (fs) varieties are, in fact, SNC subsets in the \'etale topology, while the  ''pinch point'' is not.
\end{remark}

%\subsection{Relative Desingularization of NC divisors}\begin{theorem} \label{th: resolution2} Let $X$ be a logarithmic  variety over a  field $K$ of characteristic zero, with a divisor $D$.  Let  $V$ be an open of $X$ where $X$ is smooth and $D$ is an NC  divisor. Assume that the set of components of $D$ is partially ordered and the order is total for the subsets of components passing through a point.
 %There exists a canonical relative resolution of singularities of $(X, D)$ i.e. a birational projective $f: Y\to X$ such that
 %\begin{enumerate}\item $f$ is an isomorphism over the open set $V$ .
 %\item The variety $(Y,V)$ is a toroidal embedding.
 %\item The variety $(Y,Y\setminus D_Y)$ is a toroidal embedding, where $D_Y$ is the strict transform of $D$   in $Y$.
 
% \item The inverse image $f^{-1}(X\setminus V)=Y\setminus V$  has   normal crossing  with $D_Y$.
% \item $f$ is a composition of the normalization and the normalized blow-ups of the locally monomial filtered centers $\{\cJ_{in}\}_{n\in \NN}$.
 % defined by the sets of valuations.

 %\item
 %$f$ commutes with smooth morphisms and field extensions preserving the order of the components $D_V$, in the sense that the centers are transformed functorially, and the trivial blow-ups are omitted.

 %\end{enumerate}

%\end{theorem}

\subsection{Desingularization of varieties except for log smooth locus}
%define the stratification, with strata bei???\Jarek

\subsubsection{Extendable  toroidal embeddings}
\begin{definition} \label{trivia} Let $(X, D)$ be a strict toroidal embedding  and $S$ be the induced stratification. For for any stratum $s\in S$ consider the monoid of the effective  Cartier divisors $\Cart(s,S)^+$ on an open  neighborhood $$U_s:=Star(s,S)=\bigcup_{s'\geq  s }{s'}$$ supported on $D_s:=D\cap U_s$. Then  $(X, D)$
  will be called {\it  extendable}  if there exists a {\it  Cartier system } $\Phi=\Phi(X)$, that is, a collection $\Phi=\{\Phi_s\}_{s\in S}$ of finite nonempty subsets $\Phi_s\subset \Cart(s,S)^+$ of effective Cartier divisors on $U_s$ for  $s\in S$, such that
   \begin{enumerate} 
   
   %\item $1\in \Phi_s$\footnote{"1" stand for the trivial Cartier divisor. They naturally occur when restricting Cartier divisors to open subsets in the condition (3)}.
   \item  $\Phi_s\subset \Cart(s,S)^+$   generates $\Cart(s,S)^+$ (as monoid).
   \item If $s\leq s'$ then $U_{s'}\subset U_s$, and the induced restriction of the Cartier divisors defines a surjective map $\Phi_s\to \Phi_{s'}$ of the sets.

   \end{enumerate}
%If  $f: X\to Y$ is \'etale map of strict toroidal embeddings then the pull-back of  any Cartier system $\Phi_Y$ on
%$Y$ defines a Cartier system $f^*(\Phi_Y)$. We say that the Cartier systems $\Phi_X$, and $\Phi_Y$ are {\it compatible} if $\Phi_X=f^*(\Phi_Y)$.

%A toroidal embedding $(X, D)$ will be called {\it extendable}
%f there exist compatible Cartier systems  on a certain   \'etale strict toroidal cover.

\end{definition}

%Note  that,  if there exist compatible Cartier systems  on a certain strict  \'etale cover of a strict toroidal embedding $X$, then they descend to such a Cartier system  on $X$. This follows from the Lemma
%\begin{lemma} Let $(X, D_X)\to (Y, D_Y)$ be an \'etale morphisms of strict toroidal embeddings. Then for any stratum $s$ on  $(X, D_X)$ mapping to a stratum $s'$ on $(Y, D_Y)$, we have a map $\Star(s,S)\to \Star(s',S')$ inducing the isomorphism of the Cartier groups.

%\end{lemma}
%\begin{proof} A consequence of Lemma \ref{Mum}.
%\end{proof}

%The above Lemma can be  directly derived  from Lemma \ref{toroidal-cartier}.
%It shows that $\Phi$ can be descended from open \'etale cover.

%We extend the above definition to toroidal embeddings.
%\begin{definition} \label{trivia2} Let $(X, D)$ be a  toroidal embedding  %and $S$ be the induced stratification. %For for any stratum $s\in S$ consider the monoid of the effective  Cartier divisors $\Cart(s,U_s)$ on an open strict neighborhood neighborhood $$U_s:=Star(s,S)=\bigcup_{\{s'\mid s\subset \overline{s'}\} }{s'}$$ supported on $D_s:=D\cap U_s$. Then  $(X, D)$
 % will be called {\it  extendable}  if there exist compatible strict systems of Cartier divisor on a strict  \'etale cover.

  \begin{example} 
  The SNC divisors on smooth varieties are extendable, with the Cartier systems $\Phi_s$, defined by the  components of the Weil divisors through $s$. %(in the strict toroidal \'etale neighborhood).  
  %Likewise the SNC divisors are strict extendable.

\end{example}

\begin{lemma} Strict toroidal embeddings with quotient singularities are 
 extendable. Moreover the Cartier systems can be chosen functorial for smooth  morphisms dominating on the strata.
 \end{lemma} 
 \begin{proof} Define the  Cartier system $\Phi=\Phi(X):=\{\Phi_s\}$, with 
 $\Phi_s:=\Cart(s,U_s)^+_{\leq n}$ to be the set all the  effective Cartier divisors on $U_s\subset X$ which are linear combinations of their Weil components with the coefficients $\leq n$. 
 %This collection is compatible in the strict \'etale topology.

This defines a  Cartier system  for sufficiently large $n$.
Indeed, for any irreducible Weil  divisor $E$ on $U$, let $n_E$ denote the smallest integer for which $n_E\cdot E$ is Cartier.
Let $n$ be the integer $\geq n_E$, where $E$ ranges over all the components of $D$ and such that each 
$\Phi_s=\Cart(s,S)^+_{\leq n}$ generates $\Cart(s,S)^+$ for each stratum $s$.

If $s\leq {s'}$ then  any effective Cartier divisor $E_{s'}=\sum m_EE$ in $\Phi_{s'}$, defined on $U_{s'}$ extends to  an effective   Cartier divisor $$E_{s}=\sum m_EE+\sum s_{E'}E'$$ in $\Cart(s,S)^+$ by an elementary fact for the cones (See also Example \ref{tr2}), where $E'$ are the Weil component in $U_s\setminus U_{s'}$. Then for any such $E'$   any  multiple  $k\cdot (n_{E'}E')$ is a Cartier divisor and can be subtracted from the effective Cartier divisor $E_{s}$. Consequently, the  coefficient $s_{E'}$ can be adjusted to a new (nonnegative) coefficient $s'_{E'}$ with $s'_{E'}<n_{E'}\leq n$ in the  presentation of the modified Cartier divisor  $$E'_{s}=\sum m_dE+\sum s'_{E'}E'\in \Phi_{s}.$$  
This shows the surjectivity of the map $\Phi_s\to \Phi_{s'}$.
\end{proof}

\begin{example} \label{tr1}
If  $(X, D_X)$ is a strict toroidal embedding for which any effective Cartier divisor on any subset $U_s\subset X$  supported on $D_X\cap U_s$ extends to an effective Cartier divisor on $X$ supported on $D_X$, then $(X, D_X)$ is  extendable. We can simply set $\Psi_{s}$ to be the set of the extensions of generators of $\Cart(s,S)^+$ to $X$ and put $\Psi=\bigcup_{s\in S} \Psi_{s}$. Then we set $\Phi_s$ to be the set of restrictions of $\Psi$ to $U_s$.

 %It is very easy to show that
  \end{example}

 The conditions for strict extendable toroidal embeddings can  be translated into the language  of the associated conical complexes. (Section \ref{associated complex})

\begin{lemma} A strict toroidal embedding $(X, D)$ is extendable if the associated complex $\Sigma$ \footnote{Definition \ref{complex1}, Theorem \ref{cc}} satisfies the condition:
There exists  a functor $\cL$ associating with each  face $\sigma \in \Sigma$ a finite set of integral linear functions $\cL_\sigma$ from $\sigma \cap N_\sigma$ to $\ZZ_{\leq 0}$, such that 
\begin{enumerate}
%\item $0\in \cL_{\sigma}$

\item $\cL_{\sigma}$ generates the monoid  $\sigma^\vee\cap M_\sigma$ of the  integral linear functions which are nonnegative on $\sigma\cap N_\sigma$.
\item The restriction defined by a  face inclusion $i_{\tau\sigma}:\tau\to \sigma$  yields the surjective map of the sets $\cL_{\sigma}\to \cL_{\tau}$.
\end{enumerate}

\end{lemma}

We shall call such a complex {\it extendable}.

\begin{example} \label{tr2} If $\Sigma$ is a complex with a single
maximal face $\sigma$ then it is extendable. (See also Example \ref{tr1}).
in particular, any strict toroidal embedding $(X, D)$ with a single minimal (closed) stratum is extendable.
\end{example}
\begin{proof} Follows from a,  well-known fact,  that any nonnegative integral linear function  on a face $\tau$ of $\sigma$ extends to a nonnegative integral linear function  on $\sigma$. (See for instance \cite{Fulton}). So one can take $\cL_{\sigma}$ to be the set of extensions of the generators of  all $\tau^\vee\cap M_{\tau}$, where $\tau$ ranges over all faces of $\sigma$. Then we simply put
$\cL_{\tau}$ to be the set of the restrictions of the elements in $\cL_{\sigma}$  to $\tau$.

\end{proof}

\subsubsection{Extension of  log smooth logarithmic structures}
\begin{lemma}\label{extension} Let $(X, D)$ be a variety with a Weil divisor.
 Let $(V, D_V)$ be an extendable toroidal subset   of a variety $(X, D)$ with a Cartier system $\Phi$.
%Assume that $D_V$ has  locally ordered components.
Then there is a canonical projective birational  modification $\tilde{X}$ of $X$ which is an isomorphism on $V$, and a canonical extension $\cM$ of the logarithmic structure on $(V, D_V)$ such that 
 $(\tilde{X},\cM)$ is a  logarithmic variety.
Moreover,  the  logarithmic structure $\cM$ on $\tilde{X}$ is functorial  for smooth morphisms respecting  $V$ and $D$ and a Cartier system $\Phi$.

\end{lemma}

\begin{proof}

 Let $D_X:=\overline{D_V}$ be the closure of the Weil divisor $D_V$.
%Consider a strict \'etale neighborhood $\pi_U:U\to X$ such the induced toroidal subset $V_U\subset U$ is strict with the induced divisor $D_U$. 
Let $S$ be the  toroidal stratification  on the toroidal subset $(V, D_V)$. 
For any stratum $s\in S$ on $V$ consider the monoid $\Phi_s\subset\Cart^+(s,S)$ of effective Cartier divisors on $$V_s=\Star(s,S)\subset V.$$ The closures of the divisors in $\Phi_s$ can be thought as Weil divisors on $X$ . 
%To get the functorial properties we can choose a  Cartier system $\Phi=\{\Phi_s\}$  to be the minimal with respect to, some a priori chosen, lexicographic order on the components. Such a Cartier system will be uniquely determined.

For any point $x\in X$ let $\overline{s_{1x}},\ldots,\overline{s_{n_xx}}$ be the minimal sets (with respect to inclusion) which are the closures of the strata of $S$ through $x$. Let  $n_x$ be their number. Observe that $n_x=1$ when $x\in V$, as any such point belongs to a unique minimal stratum of $S$. 
Set $$N:=\max \{n_x \mid x\in X\}$$ and denote  by $\cI_{\overline{s_{ix}}}$ the ideals of the subschemes $\overline{s_{ix}}$ and consider the blow-ups at $$\cI_{\overline{s_{1x}}}+\ldots+\cI_{\overline{s_{n_x}}}$$ cosupported on the disjoint (by maximality of $N$) subsets $\overline{s_{1x}}\cap \ldots\cap \overline{s_{n_xx}}$, with $n_x=N$.
In particular, as long as $N>1$ the centers of the blow-ups are disjoint from $V$.
 After each such a blow-up the strict transforms of $\overline{s_{1x}}, \ldots,\overline{s_{n_xx}}$ will have no common intersection, and 
the number $N$ drops. By continuing this process we arrive at the situation with $N=1$ so that each 
point belongs to a unique minimal closure of a single stratum. 

The process does not affect the open subset $V$. 
Moreover, for any two closures of strata $\overline{s}_1$  and $\overline{s}_2$ the intersection $\overline{s}_1\cap\overline{s}_2$ is the union of the closures of the strata, as any point $x\in \overline{s}_1\cap\overline{s}_2$ belong to a unique minimal $\overline{s}\subset \overline{s}_i$, for $i=1,2$.

Now for any $s\in S$ consider  the corresponding locally closed subset  $t=t(s)\in T$ defined as $$t:=\overline{s}\setminus (\bigcup_{s\not\subseteq \overline{s'}\subseteq \overline{s}} \overline{s'})$$

Then $t\cap V=s$,  $\overline{t}=\overline{s}$, and  all subsets $t$ are disjoint. Thus  $T$ is a stratification 
 . Moreover, the  order $\leq$ on $S$ defines the order
on $T$:
$t(s)\leq t(s')$ iff $\overline{s}\subset \overline{s'}$ iff    $s\leq s'$.
 
 So we can write $$t:=\overline{t}\setminus (\bigcup_{t'<t} \overline{t'}).$$ 

For any $t=t(s)\in T$, let $$U_t:=\Star(t,T)=\bigcup_{t\leq t'}{t'}$$ be an open subset of $X$. Then  $U_t\cap V=V_s$.

%$$t:=\overline{s}\setminus (\bigcup_{s'<s} \overline{s'})= \overline{s}\setminus (\bigcup_{s\not\subset \overline{s'}} \overline{s'})$$

Denote by $\Phi_t$, with $t=t(s)\in T$  the set of closures in $U_t$ of divisors in $\Phi_s$  in $V_s$, which are Weil divisors on $U_t$. 

 If $t_1\leq t_2$ then, $U_{t_1}\supset U_{t_2}$, and,  by the assumption on $\Phi$,  there is a natural surjective map of the sets of divisors on $U_{t_2}$: $\Phi_{t_1|U_{t_2}}\to \Phi_{t_2}$. 
 %Thus the blow-ups of $I_{t_2|U_{t_1}}$ and $I_{t_1}$ are the same on $U_{t_1}$.

Consider  the ideal sheaf
 $$\cI_t:= \prod_{D\in \Phi_t}\cO_X(-D)$$ on the open subset $U_t$.
 We define a modification $\overline{X}\to X$ to be the  blow-up of the ideal sheaf $I_t$ on $U_t$. Note that the blow-up of $I_t$ on $U_t$ is equivalent to the composition of the blow-ups at $\cO(-D)$, with $D\in \Phi_t$. If $t\leq t'$ then,  by condition (2) of Definition \ref{trivia},  the set of  of the restrictions   $\cO(-D_{|U_{t'}})$, where $D\in \Phi_{t}$,
coincides with the set of the $\cO(-D')$ with
$ D'\in \Phi_{t'}$.

%This determines the canonicalzbiratiozal transformations $\sigma_t: \tilde{U}_t\to U_t$, which glue coherently to 
%$\sigma: \tilde{X}\to X$. 
 Thus the blow ups of the ideals $I_t$ on $U_t$ glue to define a unique transformation $\sigma: \tilde{X}\to X$  which is an isomorphism on $V$.  

 %which is the blow-up of the ideal  $I_t$ on $U_s\subset X$ defined by the descent of $I(U)$. Moreover, the blow-ups agree when passing from an allowable neighborhood for a given  stratum to an allowable neighborhood for its generization by the assumption on surjectivity of $\Phi_{s'}\to \Phi_{s}$.

%By definition, the modification  determines an isomorphism over each $V_U\subset U$.  
Moreover, any divisor $\sigma^{-1}(D)$, where $D\in \Phi_t$ is Cartier on $\widetilde{U}_t$. For any such $\widetilde{U}_t$ consider the monoid $P(\widetilde{U}_t)$ of the effective Cartier divisors on $\widetilde{U}_t$ generated by the inverse image $$\sigma^{-1}(\Phi_t):=\{\sigma^{-1}(D), D\in \Phi_t\}.$$ This induces the  canonical coherent logarithmic structure $$\cM_{\widetilde{U}_t}=P(\widetilde{U}_t)\cdot \cO_{\widetilde{U}_t}^*\subset \cO_{\widetilde{U}_t}$$ on each $\widetilde{U}_t$ which glues to the coherent logarithmic structure $\cM_{\widetilde{X}}\subset \cO_{\widetilde{X}}$ on $\widetilde{X}$.

 By condition (1) of Definition \ref{trivia}, 
the restriction $\sigma^{-1}(\Phi_t)\to \Phi_s$ of  Cartier diviors 
%in $P(\widetilde{U}_t)$  to $V_s=V\cap \widetilde{U}_t$ 
defines a surjective map to the set of generators $\Phi_s$ of the monoid $\Cart(V_s, D_V\cap V_s)$ of the Cartier divisors on $U_s$ supported on $D_V\cap U_s$. Thus the logarithmic structure  $\cM_{\widetilde{X}}$ extends the toroidal logarithmic structure on $V$ induced by $D_V$.

%Moreover, the groupification $\cM_{\widetilde{X}}^{gp}$ is a subsheaf of monoids of  the constant sheaf $ K(X)^*$. 
%We obtain a logarithmic scheme $(\widetilde{X},\cM_{\widetilde{X}})$ with a coherent logarithmic structure.
Finally the saturation functor {\it sat} transforms functorially  the logarithmic variety  $(\widetilde{X},\cM_{\widetilde{X}})$ with a  coherent logarithmic structure  into the variety $(\widetilde{X}^{\sat},\cM^{\sat}_{\widetilde{X}})$ with a  fine and saturated logarithmic structure. %(which is not necessarily strict). 
%Since ${\widetilde{X}}$ is normal the saturation of $\cM_{\widetilde{X}}$ in $K(X)^*$,   defined locally as $$\cM_{\widetilde{X},x}^{\sat}:=\{ f\in  K(X)^* \mid f^n\in
%\cM_{\widetilde{X},x}\},$$ is a subsheaf of monoids of $\cO_{\widetilde{X}}$. It defines   a  fine and saturated logarithmic structure on a variety $\widetilde{X}$ extending the logarithmic structure on the subset $(V, D_V)$.

%clude
%Consider %the integral closure followed by 
%its saturation $({\widetilde{X}},\cM_{\widetilde{X}})^\sat$.  The variety  $({\widetilde{X}},\cM_{\widetilde{X}})^\sat$ has   a  fine and saturated logarithmic struture.
%( but maybe nonreduced, so we apply the reduction.

%The fact that the toroidal subset $V$ is saturated  in $(\tilde{X},\cM)$ (see Definition \ref{saturation}) follows from the property that the support of $E=\tilde{X}\setminus V$ is a closed stratum (thus the union  of strata) on $\tilde{X}$.
\end{proof}

\subsubsection{Desingularization of varieties except for a toroidal subset}

\begin{theorem} \label{th: resolution6} Let $X$
 be a  variety over a  field $K$ of characteristic zero, and $D_X$ be any Weil divisor on $X$.  Let  $(V, D_V)$ be an open   extendable toroidal subset of $(X, D_X)$\footnote{Definitions \ref{saturation}, \ref{trivia}}. Assume that $D_V$ has  locally ordered components \footnote{Definition  \ref{order}} .
  
   There exists a canonical  resolution of singularities of $(X, D_X)$  except for the toroidal subset $V$ i.e. a  projective birational morphism $f: Y\to X$ such that
 \begin{enumerate}
 \item $f$ is an isomorphism over the open set $V$ .
 %\item The variety $(Y,V)$ is a toroidal embedding.
 
 \item The variety $(Y, D_Y)$ is a strict toroidal embedding, where $D_Y:=\overline{D_{V,Y}}$ is the closure of the divisor $D_V$ in $Y$.
 \item  $(Y, D_Y)$ is the  saturation of $(V, D_V)$\footnote{Definition \ref{saturation}}. In particular, $(Y, D_Y)$ is an extendable toroidal embedding.
 \item The complement $E_{V,Y}:=Y\setminus V$ is a divisor which has simple  normal crossings   with $D_Y$ \footnote{Definition \ref{nc}}. So is the exceptional divisor $E_{\exc}\subseteq E_{V,Y}$.
 %\item If $D$ is the closure of $D_V$ in $X$ and $(V, D_V)$ is saturated in $(X, D)$  then $V$ is the set
 %where $f$ is an isomorphism.
% \item If $V$ is the toroidal (log smooth) locus of $(X, D_X)$  then the complement $E_0=Y\setminus V=E$ is the exceptional divisor. 

 %\item  $(Y, D_Y\cup E)$ is  a strict toroidal embedding.  % defined by the sets of valuations.
\item In particular, if $V$ is smooth and $D\cap V$ is an  SNC divisor on a smooth subset $V\subseteq X$ then 
$Y$ is smooth and  $D_Y\cup E_{V,Y}$ is an SNC  divisor.
 
 %\item $f$ is a composition of the normalization and the normalized blow-ups of the locally monomial nonsingular centers $\{\cJ_{in}\}_{n\in \NN}$.
 \item
 $f$ commutes with field extensions  and smooth morphisms respecting the saturated toroidal subset, the order of the components $D_V$, and the Cartier system on $(V, D_V)$.
 \item If $G$ is an algebraic group acting on $(X, D)$ and preserving the open subset $V$, and the components of the divisor $D_V$ on $V$ then the action of $G$ on $X$ lifts to $Y$, and $f:Y\to X$ is $G$-equivariant.
%\item The exceptional locus $E$ can be further modify to a relative SNC divisor $E'$ (so that $E'$ has SNC with $D_Y$) by a sequence of blow-ups at relatively smooth centers. 

 \end{enumerate}

\end{theorem}
\begin{proof} By Lemma \ref{extension}, we extend canonically the logarithmic structure on $(V, D_V)$ to the structure $\cM_{\overline{X}}$ of logarithmic variety on a certain birational modification $\overline{X}$ of $X$ preserving $V$. %so that $V$ is saturated in $(X,\cM)$. 
To finish the proof we apply Theorem \ref{th: resolution5} to $(\overline{X},\cM_{\overline{X}})$.

\end{proof}

\begin{remark} The theorem extends  the nonembedded Hironaka desingularization. We consider the zero divisor $D=0$ on a variety $X$, and   the nonsingular locus $V=X^{ns}$. Then  $V=(X,0)^{tor}$ is the toroidal  locus of the  variety $(X,0)$  with a  single nonsingular stratum $V$ on $V$. Let $Y\to X$  be the desingularization except for $V$. 
 Since $Y$ is the saturation of $V$ it is nonsingular with one big smooth stratum $Y$ which is an extension of the unique smooth stratum $V$ on $(V,0)$. Consequently, the resolution $Y\to X$  except for $V$ is simply a canonical Hironaka desingularization which keeps nonsingular locus singular locus $V=X^{ns}$ untouched.
Moreover, the exceptional locus  $E=Y\setminus X^{ns}$ is a simple normal crossing divisor.

\end{remark}

The above theorem,   in very particular,  is closely related to a variant of the Bierstone-Milman desingularization theorem except for SNC locus. (\cite{B-M-except-I}, Theorem 3.4) Consider a reducible divisor $D$ on a
smooth variety $X$. Let $V=(X, D)^{snc}$ will be the locus of point in $X$ where $(X, D)$ is SNC.
 By  Theorem \ref{th: resolution6}, there exists a resolution $(Y, D_Y)$  of $(X, D)$ except for $V$,   which is an isomorphism on $V$ and such that  $D_Y\cup E_{V,Y}$ is SNC.
 
\subsubsection{Toroidal compactifications of toroidal embeddings }

As an immediate corollary from the above we obtain the following result.
\begin{theorem}\label{compactification}
Let $(X, D)$ be a  extendable toroidal embedding over a field of characteristic zero \footnote{Definitions \ref{saturation}, \ref{trivia}}.  There exists a 
toroidal compactification $(\overline{X},\overline{D})$ of $(X, D)$ such that 

\begin{enumerate}
\item $(\overline{X},\overline{D})$ is a complete strict toroidal embedding,  where $\overline{D}$ is the closure of $D$ in $\overline{X}$.
\item $(\overline{X},\overline{D})$ is the  saturation of $(X, D)$ in $(\overline{X},\overline{D})$\footnote{Definition \ref{saturation}}. In
particular it is  extendable (and admits the same singularities as $(X, D)$). 
\item The complement $E:=\overline{X}\setminus X$ of $X$ is a divisor, which has SNC with $\overline{D}$. \footnote{Definition \ref{nc}}
%\item If $(X, D)$ is strict toroidal then $(\overline{X},\overline{D})$ is such.
\item If $X$ is quasi-projective then $\overline{X}$ is projective.
\item Moreover, if an algebraic group $G$ acts on $(X, D)$ preserving the components then there exists a $G$-equivariant compactification $(\overline{X},\overline{D})$ satisfying the above properties.\end{enumerate}
\end{theorem}

\begin{proof} We  consider the Nagata (\cite{Nagata}) or a projective completion ${X}_0$ of $X$, and let ${D}_0$ be the closure of $D$ in ${X}_0$. %Let $(V, D_V)\subseteq ({X}_0,{D}_0)$ be the  saturation of $X_0$. 
To finish the proof we apply Theorem \ref{th: resolution6} to the toroidal subset $(X, D)\subset ({X}_0,{D}_0)$. If $G$ acts on $X$ we use the Sumihiro compactification theorem (\cite{Sumihiro}) to construct $X_0$. 
%The action preserves  the saturation $V$ and lifts to $\overline{X}$.

%By Lemma \ref{extension},  we can extend the fine and saturated log structure on $V$ to a birational projective modification $\overline{X}_1$ of $\overline{X}_0$.
%By Theorem \ref{th: resolution5} we can  transform the logarithmic variety $\overline{X}_1$ into  a quasi-toroidal embedding $\overline{X}_1$. The latter can be resolved 
%by Theorem \ref{th: resolution4}. 
\end{proof}

The above theorem is a natural  extension of a much simpler fact of the compactification of toric varieties due to Sumihiro,(\cite{Sumihiro}) while controlling singularities. 
We give here a strenghtening of the Sumihiro result
\begin{theorem} Let $X$ be a (normal) toric variety containing a torus $T$, and let $D_X:=X\setminus T$ be a toric divisor.
The toric variety $(X, D_X)$, admits an equivariant toric compactification $\overline{X}$ having the same singularities as $X$, and such that
$E:=\overline{X}\setminus X$  is a (toric)  divisor having SNC with the closure $\overline{D_X}\subset \overline{X}$ of $D_X$.
%In particular, the singularities on $\overline{X}$ are  the same as on $X$.
\end{theorem}
\begin{proof} We find  a Sumihiro compactification $X_0\supseteq X$ first.  Note that since the set $X$ is $T$-stable it is the union of the orbits (strata) and thus it is  saturated in $X_0$.
Then  we apply a  $T$-equivariant toroidal desingularization from Theorem \ref{des toroidal2} except for $X$, to the toric compactification $X_0$, and $D=X_0\setminus T$.
\end{proof}

\subsubsection{Desingularization except for locally toric singularities}
%As a consequence we show a local version of the previous theorem

\begin{theorem} \label{compactification2}
Let $X$ be a variety over a field 
of characteristic zero. Let $x\in X$ be a point where $X$ has a locally toric singularity \footnote{Definition \ref{locally}}. There exists a projective birational transformation $Y\to X$, which is an isomorphism over a certain neighborhood $U$
of $x\in X$,  and a Weil divisor $D_Y$ on $Y$ such that $(Y, D_Y)$ is a strict toroidal embedding which has  a single closed
stratum, and  all its strata pass through $x\in X$.
%Moreover, if a torus $T$ acts  on $X$ with fixed point $x$, then one can choose $Y$ so that the action lifts to $(Y, D_Y)$.
%\begin{enumerate} 

%\item$(Y, D_Y)$ is a strict toroidal embedding defined by a certain Weil divisor $D_Y$ on $Y$.\item $(U, D_U)$ is a strict toroidal embedding with a single closed stratum, where $D_U:=D_Y\cap U$.\item  $(Y, D_Y)$ is the  saturation of $(U, D_U)$. In particular, $(Y, D_Y)$ has a single closed stratum.\end{enumerate}
\end{theorem}
\begin{proof} Consider the structure of a toroidal embedding on a neighborhood $(U, D_U)$ of $x$. We can assume (by shrinking $U$ if necessary) that the there exists a unique closed  stratum on $U$ which passes through $x$. This implies that $(U, D_U)$ is extendable, by Example  \ref{tr2}. Let $D_X=\overline{D_U}$ be the closure of $D_U$, and $(V, D_V)\subset (X, D_X)$ be the  saturation. 
 It suffices  to apply Theorem \ref{th: resolution6} to $(V, D_V)\subset (X, D_X)$.
 
 %If ,additionally, $T$ acts $X$ then one can choose a $T$-stable neighborhood $U$ and $T$-stable divisor
 %$D_U$. Indeed $T$ acts on the     completion $K[[u_

\end{proof}

\section{Toroidal embeddings } \label{complexes}
\subsection{Toric varieties}
Let $K$ be a field and $T=\Spec (K[x_1,x_1^{-1},\ldots,x_n, x_n^{-1}])$ be an $n$-dimensional torus over $K$. 
Denote by $$M:={\rm Hom}_{alg.gr.}(T,K^*)=\langle x_1,x_1^{-1},\ldots,x_n, x_n^{-1}\rangle=\{x^\alpha\mid \alpha\in \ZZ^n\} $$ \noindent the lattice of the
group homomorphisms to $K^*$, i.e.  characters of $T$. Then
the dual lattice $$N=
{\rm Hom}_{alg.gr.}(K^*,T)\simeq M^\vee$$ can be 
identified with the lattice of $1$-parameter subgroups of $T$, that is the homomorphisms $$K^*\to T,\quad  t\mapsto (t^{a_1},\quad \ldots, t^{a_n}).$$ 
The
vector space $M^{\bf Q}:=M\otimes{\bf Q}$ is dual to $N^{\bf Q}:=
N\otimes{\bf Q}$. Denote by $(v,w)$  the relevant pairing for $v\in N^{\bf Q}, w\in M^{\bf Q}$.

%Let $N\simeq
%{\bf Z}^k$ be a lattice contained in 
%the vector space $N^{\bf Q}:=N\otimes {\bf
%Q}\supset N$.  
 By a {\it cone} in this paper 
we  mean a convex set $$\sigma = {\bf Q}_{\geq 0}\cdot
v_1+\ldots+{\bf Q}_{\geq 0}\cdot v_k\subseteq N^{\bf Q}$$ in $N^{\bf Q}$, for some vectors $v_i\in N^{\bf Q}$.  
A cone is {\it strictly convex} if it contains no line. 
For any cone $\sigma\subset N^{\bf Q}$ we denote by $\sigma^\vee$ the dual cone, 
$$\sigma^\vee:=\{m\in M^{\bf Q} \mid (v,m) \geq 0 \,\, {\rm
for\,\, 
any} \,\,\,    v\in \sigma\}\subseteq  M^{\bf Q},$$ and  by $(\sigma^\vee)^{\integ}:=\sigma^\vee\cap M$ the monoid of the integral vectors in  $\sigma^\vee$.
%$$(\sigma^\vee)^{\integ}:=\{m\in M \mid (v,m) \geq 0 \,\, {\rmfor\,\, any} \,\,\,    v\in \sigma\}=\sigma^\vee\cap M ,$$ 

This associates with a cone $\sigma$ a (normal) affine toric variety
 $$X_\sigma:=\Spec(K[(\sigma^\vee)^\integ])\supseteq T.$$

%All cones in the papae

%By abuse of language we shall speak of a {\it cone $\sigma$ in alattice $N$}. To avoid confusion we 
%shall sometimes write $(\sigma, N)$ for the cone $\sigma$
%in $N$. For a cone $\sigma$ in $N$ denote by $N_\sigma:=N\cap\lin(\sigma)$ the
%sublattice generated by $\sigma$.

\begin{definition} (see \cite{Danilov},
\cite{Oda}).\label{de: fan} By a {\it fan}
$\Sigma $ in
$N^{\bf Q}$ we mean a finite collection of finitely 
generated strictly convex cones $\sigma$ in $N^{\bf Q}\supset N$ such 
that 

$\bullet$ any face of a cone in $\Sigma $ belongs to $\Sigma$,

$\bullet$ any two cones of $\Sigma $ intersect in a common face. 

By the {\it support} of the fan we mean the union of all
its faces, 
$|\Sigma|=\bigcup_{\sigma\in \Sigma}\sigma$.

If $\sigma$ is a face of $\sigma'$ we shall write $\sigma\preceq\sigma'$.

%For any set $\Sigma$ of cones  in $N$ we denote by $\overline{\Sigma}$
 %the set $\{\tau \mid \tau\prec \tau'\mbox{ for some } \tau'\in\Sigma\}$.
\end{definition}

%For any set $\Sigma$ of cones  in $N$ we denote by $\overline{\Sigma}$
% the set $$\overline{\Sigma}:=\{\tau \mid \tau\preceq \tau'\mbox{ for some } \tau'\in\Sigma\}$$

To a fan $\Sigma $ there is associated a {\it toric variety}
$X_{\Sigma}\supset {\bf }T$, obtained by gluing $X_\sigma$, where $\sigma\in \Sigma$.

It is  a normal variety on which a torus
$T$ acts effectively with an open dense orbit 
(see \cite{KKMS}, \cite{Danilov}, \cite{Oda}, \cite{Fulton}).
\subsection{The orbit stratification of toric varieties}
\label{orbits}
 To 
each cone $\sigma\in \Sigma$
corresponds an  open affine invariant subset
$X_{\sigma}$ and its unique closed orbit $O_{\sigma}$. 

The orbits form a locally closed smooth stratification,
and $\tau\preceq \sigma$ if and only if $O_\sigma\subset\overline{O_\tau}$ (or in our notation $O_\sigma\leq {O_\tau}$).

\begin{definition}\label{de: star2} Let $\Sigma$ be a fan
and $\tau \in \Sigma$. The {\it star} of the cone $\tau$ is
defined as follows:
$${\rm Star}(\tau ,\Sigma):=\{\sigma \in \Sigma\mid 
\tau\preceq \sigma\},$$ 
%$$\overline{{\rm Star}}(\tau ,\Sigma):=\{\sigma \in
%\Sigma\mid  \sigma'\preceq \sigma  \mbox{ for some }  \sigma'\in
%{\rm Star}(\tau ,\Sigma)\}.$$ 
\end{definition}

The
orbits in the
closure $\overline{O_\tau}$ of the  orbit $O_\tau$ on the toric variety $X_\Sigma$ correspond to the cones of 
${\rm Star}(\tau ,\Sigma)$. 

Denote by $N^{\bf Q}_\tau$ the subspace of $N^{\bf Q}$ spanned by the cone $\tau$. It defines the lattice  $N_\tau:=N\cap N^{\bf Q}_\tau$. The dual space  to $N^{\bf Q}_ \tau $  is isomorphic to $M^{\bf Q}_ \tau := M^{\bf Q}/ \tau ^\perp,$ where $ \tau ^\perp: =\{v\in M^{\bf Q} \mid (v,w)=0 \mid w\in  \tau \}$, and
 the dual cone to $\overline{ \tau }:=( \tau ,N^{\bf Q}_ \tau )$ is isomorphic to $\overline{ \tau }^\vee:= \tau ^\vee/ \tau ^\perp$ in  $M^{\bf Q}_{\tau}=M^{\bf Q}/\tau ^\perp$.

The subtorus $T_ \tau $ corresponding to the sublattice $N_ \tau :=N\cap N^{\bf Q}_ \tau $ is isomorphic to the stabilizer of the points in $O_ \tau $.

The quotient torus $\overline{T}_ \tau :=T/T_ \tau $ corresponds to the quotient space $\overline{N}^{\bf Q}_ \tau :={N}^{\bf Q}/N^{\bf Q}_ \tau $.
It acts transitively on the big orbit $O_ \tau \subset\overline{O_ \tau }$  making  $\overline{O_ \tau }$ into a toric variety.

Denote by $\pi_\tau: N^{\bf Q} \to \overline{N}^{\bf Q}_\tau$ the projection map.
The toric  subvariety $\overline{O_\tau}\supseteq O_\tau$ corresponds to the quotient fan 
$${\rm Star}(\tau ,\Sigma)/\tau:=\{ \pi_\tau(\sigma) \mid \sigma\in {\rm Star}(\tau ,\Sigma)\}.$$
(see \cite{KKMS}, \cite{Danilov}, \cite{Oda}, \cite{Fulton}).

If  $\sigma\in {\rm Star}(\tau ,\Sigma)$ then 
$$\pi_\tau(\sigma)^\vee=\sigma^\vee\cap \tau^\perp.$$ The natural embedding $\overline{O_\tau}\cap X_\sigma\to X_\sigma$ is given by the projection $$\sigma^\vee\cap M\to \sigma^\vee\cap \tau^\perp\cap M,$$ such that $m\mapsto m$ if $m\in \sigma^\vee\cap \tau^\perp\cap M$, and $m\mapsto 0$ otherwise. This  defines an epimorphism: $$K[(\sigma^\vee)^\integ]\to K[(\pi_\tau(\sigma)^\vee)^\integ],$$ 
identifying the closure of the orbit $\overline{O}_\tau$ with a toric variety.
In particular, we have that

\begin{lemma} \label{normalt} The closures of the toric orbits are normal (toric) varieties.
\end{lemma}

%\begin{remark} The quotient complexes play important role in the algorithm. (Definition \ref{quotient}).

%\end{remark}

\subsection{Birational morphisms of toric varieties}
\begin{definition}(see \cite{KKMS}, \cite{Oda},
\cite{Danilov}, \cite{Fulton}). A
{\it birational toric morphism}  of toric
varieties $X_\Sigma \to X_{\Delta}$ is a 
morphism identical on $T\subset X_\Sigma, X_{\Delta}$. 
\end{definition}

\begin{definition} (see \cite{KKMS}, \cite{Oda},
\cite{Danilov}, \cite{Fulton}). 
A {\it subdivision} of a fan
$\Sigma$ is a fan $\Delta$ such that $|\Delta|=|\Sigma|$
and any cone $\sigma\in
\Sigma $ is the union of cones $\delta\in
\Delta$. 
\end{definition}

\begin{theorem}(see \cite{KKMS}, \cite{Oda},
\cite{Danilov}, \cite{Fulton})
There exists a bijective correspondence between  proper toric birational morphisms $X_\Sigma \to X_{\Delta}$ and the subdivisions $\Sigma$ of the fan $\Delta$. \end{theorem}
The theorem was originally stated over algebraically closed field but remains valid without this assumption with unchanged proof.

\subsection{Toroidal embeddings} \label{toroidal embeddings}

Toroidal embeddings were introduced in \cite{KKMS} over an algebraically closed field.
The following definition over an arbitrary nonclosed field is essentially due to Mumford. It is equivalent to the definition of Kato (\cite{Kato-log}), who 
considered toroidal embeddings  in a more general context of logarithmic geometry (and refer to them as logarithmically smooth varieties). It is also equivalent to  another Mumford's  definition over a base field which is algebraically closed.

\begin{definition} \label{strict}  A {\it 
strict toroidal embedding} (respectively toroidal embedding)  is  a variety $X$ with an open subset
$U$ such that any point  $x\in X$ admits an open neighborhood $V\subset X$ (respectively an \'etale neighborhood $f: V\to X$), and an \'etale morphism $\phi: (V,U_V)\to (X_\sigma,T)$,  where $U_V=U\cap V$ (respectively $U_V=f^{-1}(U)$, and $\phi^{-1}(T)=U_Y$. 
Such a morphism is called an \'etale {\it chart}.
(In the sequel, and prequel we  often represent a toroidal embedding $(X,U)$ as $(X, D)$ for the reduced divisor $D=X\setminus U$.)
\end{definition}

The   definition and Lemma \ref{normalt} imply immediately that
\begin{lemma}
\label{normal0}  The irreducible components of the divisor $D=X\setminus U$ on a strict toroidal embedding are strict
toroidal embeddings. In particular, they are  normal.
\end{lemma}
\begin{proof}
 By Lemma \ref{normal}, the components of  a  toric divisor are normal. On the other hand, the inverse images of the components of a  toric divisor under \'etale morphism, are also normal and thus, in particular, locally irreducible. Hence the irreducible components of $D$ are locally the inverse images of  irreducible components of a toric divisor, and thus are normal. Hence they are \'etale equivalent to the components of the toric divisor, which are, by Lemma \ref{normal}, isomorphic to toric varieties.
\end{proof}

\begin{remark}
 The irreducible components of the divisor $D$ on toroidal embeddings are not necessarily normal, as those components may   admit self-intersections.
\end{remark}

\begin{remark} \label{equ}
\begin{enumerate}
\item Equivalently, one rephrases  Definition \ref{strict} of toroidal embeddings in the language of the completion of local rings at $K$-points: %where $K=\overline{K}$ is algebraically closed  field : 
A variety $X$ over an algebraically closed field $K=\overline{K}$ with an open subset
$U$  is 
 a {\it toroidal embedding} if  
 for any point $x\in X$ there is an isomorphism
$$\widehat{\phi}: \widehat{\cO_{X,x}}\to \widehat{\cO(X_{\sigma})_y},$$ where $X_{\sigma}$ is  a toric variety containing a torus $T$ and corresponding to the cone $\sigma$ of the maximal dimension and $y\in O_\sigma$ is a  closed  point in the orbit $\cO_\sigma$, and  $\widehat{\phi}$ takes  the ideal of $X\setminus U$ to that of $X_\sigma\setminus T$. 
The above definition is also due to Mumford \cite{KKMS}. 

\item A toroidal embedding is {\it strict} (or without self-intersections) if,
additionally,  the irreducible components of the divisor $D=X\setminus U$ are normal (so they do not have self-intersections). (see \cite{KKMS}, and  Lemma \ref{Mum2}).
%\Jarek{alg closed field}
This characterization  of strict toroidal embeddings is due to Mumford \cite{KKMS}. 
 Later it was proven over any field not necessarily algebraically closed, in particular, in \cite{Denef} (see also Lemma \ref{Mum2}). 
 
 \item The theory of toric, and toroidal varieties was initially considered over algebraically closed fields, and mostly in the language of the completions of the local rings. As was observed by Kato in \cite{Kato-log}, most of the results can be extended to the case of nonclosed fields in a  more convenient language of  logarithmic geometry which uses  charts in the Zariski  or \'etale topology and does not require assumption on the algebraically closed base field. 
 %In fact toroidal embeddings are logarithmically smooth varieties over a field.
 \end{enumerate}
 \end{remark}

%Alternatively a variety $(X,U)$  is toroidal if any point admits $x\in X$ an \'etale neighborhood $f: Y\to X$, and an \'etale morphism $\phi: (Y,U_Y)\to (X_\sigma,T)$,  where $U_Y=f^{-1}(U)$, and $\phi^{-1}(T)=U_Y$. 

\subsection{Conical complexes}\label{conical}

The notion of the {\it conical complex associated with  a strict toroidal embedding} is a  natural extension of the fan associated with a toric variety.

The following definition of the conical complex is equivalent to the original one  from \cite{KKMS}. We use this formalism since in the later sections we are going to consider a variation of this notion in a more general setting of semicomplexes. (Definition \ref{semii})

\begin{definition} \label{complex1}

 By a {\it conical  complex}
$\Sigma $  we mean a finite partially ordered  set of finitely 
generated strictly convex cones $\sigma$ of maximal dimension in $N^{\QQ}_\sigma\supset N_\sigma$ such that 

\begin{enumerate}

%\item For any cone $\sigma$ of dimension $k$  each $N^{\QQ}_\sigma\simeq \QQ^k$ contains the lattice $N_\sigma\simeq \ZZ^k$, so that $N^{\QQ}_\sigma=N_\sigma\otimes \QQ$.

\item For any  $\tau\preceq \sigma$ there is a linear injective  map $i_{\tau,\sigma}:\tau \to \sigma$ such that
$i_{\tau,\sigma}(\tau)$ is a face of $\sigma$, with the lattice  $i_{\tau,\sigma}(N_\tau)$ saturated in $N_\sigma$.  Moreover, each face of $\sigma$ can be presented in such a form.
 %with the face relation $\preceq$ and face inclusion $i_{\tau,\sigma}$ preserving the lattice structure.
%that 

\item If $\tau\preceq\sigma\preceq \delta$ then $i_{\tau\delta}=i_{\sigma\delta}i_{\tau\sigma}$.
\end{enumerate}

\end{definition}

The definition implies that the intersection of 
two cones is 
the union of common faces.

%\begin{definition} We shall call a conical complex {\it separated} if any two faces intersect along a single face.\end{definition}

\subsection{Support of a complex}
By the {\it support} of a complex $\Sigma$ we mean the topological space
 
$$|\Sigma|:=\coprod_{\sigma\in \Sigma}\sigma/\sim $$ 
where $\sim$ is the equivalence relation generated by the inclusions $i_{\tau\sigma}: \tau\to \sigma$.

There is an inclusion $\phi_\sigma: \sigma\to |\Sigma|$, onto the closed subset  $|\sigma|\subset |\Sigma|$ homeomorphic to $\sigma$. 

Denote by $\inte(\sigma)$ the relative interior of the cone $\sigma$. This means the interior of $\sigma$ in $N^{\QQ}_\sigma$.
There is an inclusion 
$\inte(\sigma)\to |\Sigma|$ onto a locally subset $|\inte(\sigma)|$ which allows to write the support of $ |\Sigma|$ as the disjoint  union 
$$|\Sigma|:=\bigcup_{\sigma\in \Sigma}\inte(\sigma)$$
%where  $\tau\ni x\simeq i_{\tau,\sigma}(x) \in \sigma $ 
In general,  by the {\it support} of  any subset $\Sigma_0$ of a complex $\Sigma$ is defined as $$|\Sigma_0|:=\bigcup_{\sigma\in \Sigma_0} \inte(\sigma)$$

\subsection{Mumford's definition of complexes}
Using the above we see that the conical complexes  define  topological  spaces that are covered by the closed cones. This allows to define conical complexes  as topological spaces with a local cone structure. 
%(This approach is no longer valid in the case of semicomplexes, which, in general) do not glue to form a topological space.)

\begin{definition}(\cite{KKMS} \cite{Payne}) A conical  complex $\Sigma$ is a topological space $|\Sigma|$  together with a finite collection $\Sigma$ of closed subsets $|\sigma|\in \Sigma$ of $|\Sigma|$ 
such that
\begin{enumerate}
\item For each $|\sigma|$ there is  a finitely generated lattice $M_\sigma$ of continuous functions on $|\sigma|$, and the dual lattice  $N_\sigma = Hom(M,\ZZ)$ in the vector space $N^{\bf Q}_\sigma=N_\sigma\otimes Q$.

\item The natural map $\phi_\sigma: |\sigma| \to N^{\bf Q}_\sigma$  given by $x\mapsto (u\to u(x))$ maps $\sigma$ homeomorphically onto a rational convex  cone $\sigma:=\phi_\sigma(|\sigma|)$. %By abuse of notattion we shall 

\item  The inverse image  %$\phi_\sigma^{-1}(\tau)$ 
under $\phi_\sigma$
 of each face   of $\sigma$ is some $|\tau| \in \Sigma$, with $|\tau|\subset |\sigma|$ and $M_\tau = \{u_{|\tau}  : u \in M_\sigma\}$.
\item  The topological space $|\Sigma|$ is the disjoint union of the relative interiors of the $|\sigma|\in \Sigma$.

\end{enumerate}

\end{definition}

Identifying $\sigma$ with $|\sigma|$ we obtain the natural maps $i_{\tau\sigma}:=\phi_\sigma\phi_\tau^{-1} : \tau \to \sigma$ satisfying the conditions from Definition \ref{complex1}.

\subsection{Conical complexes associated with strict toroidal embeddings}

\subsubsection{Mumford's lemma}

The following useful results are essentially due to Mumford.
\begin{lemma} \label{Mum}\cite{KKMS}, \cite{Denef} Let $f: (X, D_X)\to (Y, D_Y)$ be an  \'etale morphism of normal varieties, 
mapping a point $x\in X$ to a point $y=f(x)\in Y$, and such that
\begin{enumerate}
\item The components of $D_Y$ are normal.
\item  $f^{-1}(D_Y)=D_X$.
%\item The components of $D_X$ intersect at $x_0\in X$, and the components of $D_Y$ intersect at $y_0\in Y$.
\end{enumerate}
Then $f$ defines a bijective correspondence between
the components of $D_X$ through $x$ 
and the components of $D_Y$ through $y$. 
The above correspondence extends to a bijective correspondence  between the effective Cartier  divisors on $\Spec(\cO_{X,x})$ supported on $D_X$  and  
and the effective Cartier  divisors on $\Spec(\cO_{Y,y})$ supported on $D_Y$. %through $y$.
\end{lemma}
\begin{proof} 
%First, note that such a correspondence exists locally  for some neighborhoods $U_x$ of any point $x\in X$, and $U_y$ of $y=f(x)\in Y$. 
The correspondence between Weil or Cartier diviors can be deduced by  adapting the arguments of the proof of \cite[Lemma 1, p.60]{KKMS}, or more precisely the proof of \cite[Lemma 2.3]{Denef}. 

 First, the bijective correspondence between the irreducible components of Weil divisors through $x$ and $y$ follows from the fact that the inverse image of any irreducible component is normal so is locally an irreducible component.  In other words, any component of $D_Y$ through $y$ induces a unique component of $D_X$ through $x$, which is locally its inverse image.
  
 Moreover, since the morphism $\Spec(\cO_{X,x})\to \Spec(\cO_{Y,y})$ is  flat and even \'etale the schematic pull-back of a divisor is a divisor. Thus there is 
 a bijective correspondence between the Weil divisors of $\Spec(\cO_{X,x})$ supported on $D_X$ and those of $\Spec(\cO_{Y,y})$ supported on $D_Y$.

%the normality of the components and thus their local irreducibility is preserved. 
%Since all the components intersect no two components of $D_X$ can map to the same component of $D_Y$. 
%This defines the correspondence between the components, and the Weil divisors supported on $D_X$ and on $D_Y$. Moreover, the image of Cartier divisor at a point $x$ is a  Cartier divisor at a point $y$.
%So if $D_X$ is an irreducible Weil divisor then $\cI_{D_X}$ is prime of height one. Then $\cI_{D_Y}$ is prime and  $\widehat{\cI}_{D_Y}$ is prime.

We need to show that the image $E_Y$ of an effective Cartier divisor $E_X$  on $\Spec(\cO_{X,x})$ supported on  $D_X$ is  Cartier on $\Spec(\cO_{Y,y})$. 
%It is  a local question  on $Y$. 
We will show that $\cI_{E_Y}$ is principal at a point $y$.

%Note that $\cI_{E_Y}\cdot \cO_X$ is the ideal of the Weil divisor

%First note, that by the previous considerations $E_X=f^{-1}(E_Y)$ is the inverse image of $E_Y$ so $\cI_{E_X}=\cI_{E_Y}\cdot \cO_X$. 

By the assumption, $\cI_{E_X}=\cI_{E_Y}\cdot \cO_X$  is principal so $\widehat{\cI_{E_X}}=\cI_{E_Y}\cdot \widehat{\cO_{X,x}}$  is  principal in  $\widehat{\cO_{X,x}}=\widehat{\cO_{Y,y}}\otimes_{k(y)}k(x)$ as well. This, in turn, implies that the ideal of the initial forms in the graded ring $$\inn(\cI_{E_X})=   
\inn(\widehat{\cI}_{E_X})
\subset \cO_{X}/m_x\oplus\ldots\oplus m_x^n/m_x^{n+1}\oplus \ldots$$ is principal and generated by any homogenous form of the lowest degree. Such a generating initial form can be chosen as $\inn(f)$, where $f\in \cI_{E_Y}$. Then  for any  $g\in \cI_{Y,y}$, with
$\inn(g)\in \inn(\cI_{Y,y})$ one can write as $\inn(g)=H\cdot \inn(f)$, with $H\in \inn(\cO_{X,x})=\inn(\cO_{Y,y})\otimes_{k(y)}k(x)$. This implies  that all the coefficients of $H$ are in $k(y)$, and $H\in \inn(\cO_{Y,y})$. So $\inn_y(\cI_{E_Y})\subset \inn(\cO_{Y,y})$ is principal and generated by $\inn(f)$. 
Thus  $\cI_{E_Y}\cdot \widehat{\cO_{Y,y}}$  is principal in $\widehat{\cO_{Y,y}}$ and generated by $f\in \cI_{E_Y}$. 
%This implies that $\cI_{E_Y}\cdot \widehat{\cO_{Y,y}}$ is generated by $f$
 Now by faithful flatness of the completion of local ring the ideal $\cI_{E_Y}=f\cdot \cO_{Y,y}$ is principal in $\cO_{Y,y}$.

 %This establishes a bijection between Weil or Cartier divisors on $X$ supported on 
 %Thus  for any points $x\in X$, and $y\in Y$, the bijection between the Weil divisors components of $E_X$ and $E_Y$ through $x$ , and $y$  defines locally the bijection between Cartier divisors 
 %supported on $D_X$ and on $D_Y$ passing through $x$ and $y$x. Since the morphism is surjective the local correspondence extends to the bijection between the Cartier divisors supported on $D_X$ and on $D_Y$.
\end{proof}

The following result shows that the Mumford  condition on the normality of the components can be used to detect strict toroidal embeddings.

\begin{corollary}\label{Mum2}\cite[page 195 footnote]{KKMS}

Let $(Y, D)$ be a toroidal embedding.  Then $(Y, D)$ is a strict toroidal embedding iff  the divisor $D$ has normal components. 
\end{corollary}
\begin{proof} The "only if" part follows from Lemma \ref{normal0}.
We need to prove " if " part of the corollary.
The question is local.  
 Consider an \'etale strict toroidal neighborhood $(U, D_U)\to (Y, D_Y)$ of $y\in Y$ with a point $x\in X$ over $y$. Let $s_x$ be  the stratum through $x\in X$, and $s_y$ be the stratum through $y\in Y$.

  One can assume that there is an  \'etale morphism $$\phi: (U, D_U)\to (X_\sigma\times T, D_\sigma\times T)$$ with closed orbit $O_\sigma\times T$, where $\sigma$ is a cone of maximal dimension in $N_\sigma$. Note that the induced map $$s_x\to O_\sigma\times T \simeq T$$ is an \'etale morphism of smooth subvarieties.
  
  Moreover,  any  morphism $$\phi: (U, D_U)\to (X_\sigma\times T, D_\sigma\times T)$$  which is \'etale at $x$, is defined by
  a smooth morphism $(U, D_U)\to (X_\sigma, D_\sigma)$, and a morphism $U\to T$, such that its restriction $$s_x\to O_\sigma\times T \simeq T$$ to the stratum $s_x$   is \'etale.
  In such a case the completion of a local ring at $x$ is given by 
$$\widehat{\cO}_{x,X}\simeq \widehat{\cO}_{t,X_\sigma}\otimes_{K(t)}K(x)=K(x)[[u_1,\ldots,u_k,(\overline{\sigma}^\vee)^\integ]],$$
where $(\overline{\sigma}^\vee)^\integ$, generate  $\cI_{s_x}$, and $u_1,\ldots,u_k$  determine the \'etale morphism $s_x\to T$.

By the above, given an \'etale morphism $\phi$, a morphism $\psi: (U, D_U)\to (X_\sigma, D_\sigma)$ is \'etale if the pull-backs in $\psi^*(\overline{\sigma}^\vee)^\integ$  differ from  $\phi^*(\overline{\sigma}^\vee)^\integ$ by units, and the restriction  morphism $s_x\to T$ is \'etale, that is given by a set of local parameters $v_1,\ldots,v_k$.

   We need to show that there is such a morphism for a
  certain Zariski neighborhood  $V$ of $y$. 
  First, the existence of the \'etale morphism $s_y\to T$ is clear as $s_y$ is smooth \'etale isomorphic to $s_x$. The morphism $s_y\to T$ is given by some parameters and extends to a morphism $\psi_V: V\to T$ from a certain neighborhood $V$ of $y$. 
  
  On the other hand, the morphism $(U, D_U)\to (X_\sigma, D_\sigma)$ is given by 
  the pull-backs of the $T$-invariant generators of $K[X_\sigma]=K[(\sigma^\vee)^\integ]$.
  By the previous lemma, the Cartier divisors on $U$, which are also pull-backs of the principal  toric divisors  defined by characters in $M_\sigma$
on $X_{\sigma}$ descend to  Cartier divisors  
 in a neighborhood of $y\in Y$. This means that such descent exists up to units, defining the smooth morphism $\phi_V: (V, D_V)\to (X_\sigma, D_\sigma)$.
The pull-backs of these functions define an \'etale morphism $\psi: (U, D_U)\to (X_\sigma, D_\sigma)$,  which factors through 
 a morphism
$$\phi_U\times \psi_V: (V, D_V)\to (X_\sigma\times T, D_\sigma\times T).$$

It follows that $$\widehat{\cO}_{y,Y}\simeq \widehat{\cO}_{t,X_\sigma}\otimes_{K(t)}K(y)=K(y)[[v_1,\ldots,v_k,(\overline{\sigma}^\vee)^\integ]],$$ and the morphism $\phi_U\times \psi_V$ is \'etale. This shows that $(Y, D)$ is a strict toroidal embedding.

\end{proof}

\subsubsection{The monoids of Cartier divisors associated with strata}
Let  $(X, D)$ be a strict toroidal embedding with the stratification $S=S_D$. 
Following \cite{KKMS} we associate with a stratum $s\in S$ the canonical monoids (commutative semigroups)  and groups (lattices):
\begin{enumerate}
\item $M^+_s=\Cart(s,S)^+$ is the monoid of the Cartier divisors on $$\Star(s,S)=\bigcup_{s\leq s'} s'$$ supported on $D\cap \Star(s,S)$
\item $M_s=\Cart(s,S)$ is the free abelian group of the Cartier divisors on $\Star(s,S)$ supported on $D\cap \Star(s,S)$. So $M_s$ is a lattice which is the groupification of $M^+_s$
\item $N_s:=Hom(M_s,\ZZ)$ is the dual lattice with the vector space $N_s^{\bf Q}:=N_s\otimes_{\ZZ}\QQ$.
\item $\sigma_s=\{v\in N_s^{\bf Q} \mid F(v)\geq 0,\quad F\in M^+_s\}$ is the associated strictly convex cone of maximal dimension in
$N_s^{\bf Q}$.
\end{enumerate}

\begin{lemma} \cite{KKMS} \label{toroidal-cartier} Let $Y\subset \Star(s,S)$ be an open subset intersecting $s\in S$.
Let $\phi: (Y,U)\to (X_\sigma,T)$ be an \'etale map mapping $ s$ into a  closed orbit  $O_\sigma\subset X_\sigma$.
Then the group $\Cart(Y,U)$ (resp. $\Cart(Y,U)^+)$ of Cartier  divisors (resp. effective Cartier divisor) supported  on $Y\setminus U$ is the pull-back of the group of toric Cartier divisors supported on $X_\sigma\setminus T$. 
 
 In particular, $$\Cart(Y,U)\simeq M_s\simeq M_{\overline{\sigma}}=M_\sigma/{(\sigma^\perp)}^\integ
, $$ $$ \Cart(Y,U)^+)\simeq M^+_s\simeq {(\overline{\sigma}^\vee)}^\integ={(\sigma^\vee)}^\integ/{(\sigma^\perp)}^\integ,$$
  where $\overline{\sigma}:=(\sigma,N^{\bf Q}_\sigma)$ is the associated cone of maximal dimension in $N^{\bf Q}_\sigma\subset N^{\bf Q}$.
 
 %The strata on $Y$ are the pull-backs of the orbits on $X_\sigma$. Moreover, the relation preserves the order $$s(\tau)\leq s(\tau')\quad \mbox{ iff}\quad O_\tau\leq O_\tau' \quad \mbox{ iff}\quad \tau'\preceq  \tau.$$
%in particular, $\cM_{x,Y}= \phi^*(\cM_{t,X_\sigma})\cdot \cO_{x,Y}^*$.
\end{lemma}
\begin{proof} The Lemma is a consequence  of Lemma \ref{Mum}. 
% and the properties of the orbits on toric varieties. 
%In fact, we get the correspondence between the Weil divisors supported on $X\setminus U$ and Cartier divisors
%$\Cart(Y,U)$ (resp. $\Cart(Y,U)^+)$ and the restrictions of the relevant $T$-invariant Weil and Cartier divisors  $M_{\overline{\sigma}}$ (resp. ${(\overline{\sigma}^\vee)}^\integ$) to  $\phi(Y)\subset X_\sigma$, where $\phi(Y)$ is an open neighborhood of
%$O_\sigma$. It extends to the correspondence $M_s\simeq M_{\overline{\sigma}}$, (resp. $M^+_s\simeq {(\overline{\sigma}^\vee)}^\integ$.
\end{proof}
%\simeq {(\sigma^\vee)}^\integ/{(\sigma^\perp)}^\integ
\begin{corollary} \label{inclusion2} With the preceding notation,
the cone $\overline{\sigma}=(\sigma, N_{\sigma}^{Q})$ (defined by the chart)  is dual to $M_s^+$. 
%in particular,  we have 
% $$\overline{\sigma}\simeq (\Cart(X_\sigma,T)^\vee)^\integ\simeq (\Cart(Y,Y\cap U)^\vee)^\integ\simeq (\Cart(U_s,U)^\vee)^\integ.$$ 
In particular, the cone $\overline{\sigma}$ with the lattice $N_{\sigma}$  is independent of chart and uniquely defined for the stratum $s\in S$ : $\sigma_s=\overline{\sigma}$. \qed\end{corollary}

\begin{corollary}\cite{KKMS} \label{inclusion} (see also \cite{Oda},\cite{Fulton}) 
With the preceding notation,
if $s\leq s'$ then $\Star(s,S)\supset \Star(s',S)$, is an open subset. 

Let  $Y\subset \Star(s,S)$ be an open subset with \'etale chart $\phi: (Y,U)\to (X_\sigma,T)$ and $Y':=Y\cap \Star(s',S)$  be its  open subset. The restriction $\phi_{|Y'}$ of $\phi$  defines an \'etale morphism $$\phi_{|Y'}: (Y',U)\to (X_{\sigma'},T)$$ into the open subset $X_{\sigma'}\subset X_\sigma$, with $\sigma'\preceq \sigma$ corresponding to $s'$.
%, and a smooth morphism $Y'\to  X_{\overline{\sigma'}}$. 

There is a  natural  surjection $$\Cart^+(Y,U)\simeq M^+_s\simeq (\overline{\sigma}^\vee)^\integ \quad  \longrightarrow \quad   \Cart^+(Y',U) \simeq M_{s'}\simeq (\overline{\sigma'}^\vee)^\integ $$   induced by the restriction of the Cartier divisors. Its dual map  corresponds
%corresponds to the injection $X_{\sigma',N_{\sigma'}}\to X_{\sigma,N_\sigma}$ corresponding 
to the face inclusion $\overline{\sigma'}\hookrightarrow \overline{\sigma}$.
\end{corollary}
\begin{proof}
This translates into a a well known fact of  toric varieties and cones  \cite{KKMS}, \cite{Oda}, \cite{Fulton}.
If $\sigma'\preceq \sigma$, then the open immersion $X_{\sigma'}\hookrightarrow X_{\sigma}$ correspond to the localization
$\cO(X_{\sigma'})=\cO(X_{\sigma})_m$ by a monomial $m\in (\sigma^\vee)^\integ \subset \cO(X_{\sigma})$. Hence 
$$((\sigma')^\vee)^\integ=(\sigma^\vee)^\integ+\ZZ\cdot m= (\sigma^\vee)^\integ+((\sigma')^\perp)^\integ.$$ Consequently, there is a surjection
$$(\overline{\sigma}^\vee )^\integ= ({\sigma}^\vee)^\integ/({\sigma}^\perp )^\integ \twoheadrightarrow  ({\sigma'}^\vee)^\integ/(\overline{\sigma}^\perp )^\integ=(\overline{\sigma'}^\vee )^\integ.$$
and the dual injective map of cones and lattices:
$i_{{\sigma'}{\sigma}}: \overline{\sigma'}\to \overline{\sigma},$  
where $i_{{\sigma'}{\sigma}}(\overline{\sigma'})$ is described as the face $\{m\}^\perp \subset \overline{\sigma}$.
  %Let  $s=O_\sigma$ be in the closure of the stratum $s'=O_{\sigma'}$. 

  %Then an open embedding $$X_{\sigma'}=\Star(O_{\sigma'},S_{\sigma})\subset X_{\sigma}=\Star(O_{\sigma},S_{\sigma}),$$ and 
%$\Cart(X_\sigma,T)=(\overline{\sigma}^\vee)^\integ$ restricts surjectively to $\Cart(X_{\sigma'},T)=(\overline{\sigma'}^\vee)^\integ$ corresponding to the face inclusion ???
\end{proof}
\subsubsection{Toroidal embeddings and logarithmic smoothness}
\begin{theorem} (Kato-Mumford) \label{KM}, \cite{Kato-log}
Let $(X, D)$ be a strict toroidal embedding, and $s$ be the stratum through $x\in X$. Then 
\begin{enumerate}
\item The stratum $s$ is locally the intersection of the components of the divisor $D$

\item $\widehat{\cO}_{x,X}\simeq K(x)[[u_1,\ldots,u_k,M^+_s]]$,
 where 
 %$M_x\subset \widehat{\cO}_{x,X}$ is isomorphic to the group of Cartier divisors supported on $D$, and  
 $u_1,\ldots,u_k$ generate $\widehat{\cO}_{x,X}/\cI_s\simeq\widehat{\cO}_{x,s}$.
\item  $\dim(s)+rank(M_s)=\dim(X)$. 
\end{enumerate}
\end{theorem}
\begin{proof} Let $\phi: (X,U)\to (X_\sigma,T)$ be an \'etale map mapping  a certain $x\in s$ to a point $t$ in the closed orbit  $O_\sigma\subset X_\sigma$.

The statements are valid for the toric varieties. 
Moreover,  by corollary \ref{inclusion},
we have $M^+_s\simeq (\overline{\sigma}^\vee)^\integ $, so that
 $$\widehat{\cO}_{x,X}\simeq \widehat{\cO}_{t,X_\sigma}\otimes_{K(t)}K(x)=K(x)[[u_1,\ldots,u_k,(\overline{\sigma}^\vee)^\integ]],$$
where $(\overline{\sigma}^\vee)^\integ=(\sigma^\vee)^\integ/(\sigma^\perp)^\integ$ is 
isomorphic to the semigroup  $M_s^+$.

\end{proof}

\subsubsection{Conical complexes associated with strict toroidal embeddings} \label{associated complex}

\begin{theorem} (\cite{KKMS})\label{cc} A strict toroidal  embedding $(X,U)$ determines a unique  associated   conical complex $\Sigma$. 
Moreover, there is a bijective correspondence between the strata on a toroidal embedding $X$ and  the cones of the complex. 
  $$\tau\to s_{\tau}.$$
Furthermore $\tau\preceq \sigma$ iff $s_\sigma\preceq s_\tau$
(i.e $s_\sigma$ is contained in the closure $\overline{s_\tau}$ of $s_\tau$).

\end{theorem} 
\begin{proof}
The theorem was initially proven over algebraically closed field but it extends to non-closed field using the results above: Lemma \ref{toroidal-cartier}, Corollaries \ref{inclusion2}, \ref{inclusion}. The conical complex $\Sigma:=\{\sigma_s:  s\in S\}$ is  obtained glueing of $\sigma_s$ along the  natural inclusion maps $\sigma_s\hookrightarrow \sigma_{s'}$ for $s\leq s'$ as in Corollary \ref{inclusion}. The verification is straightforward. The gluing is essentially the same as  for the cones in fans corresponding to toric varieties $X_\sigma$.
\end{proof}
\begin{definition} By a  piecewise linear functions $f: |\Sigma|\to \ZZ$ we mean a function $f$, such that the restriction $f_{|\sigma\cap N_\sigma}$ is defined by the functional $m\in (\sigma^\vee)^\integ$.

\end{definition}

Another consequence of Corollary \ref{inclusion} is the following result.

\begin{corollary}(\cite{KKMS}) \label{PL} The piecewise linear functions $|\Sigma|\to \ZZ$ are in bijective correspondence with Cartier divisors on $X$ supported on $D=X\setminus U$.
\end{corollary}
 \begin{proof} The Cartier divisors on  $\Star(s,S)$ supported on $D$ correspond to the integral linear functions defined by $m_\sigma\in M_\sigma$ on $\sigma=\sigma_s$. Moreover, by Corollary \ref{inclusion}, if $s\leq s'$ and $s'$ corresponds to a face $\sigma'$ of $\sigma$ then $m_{\sigma'}$ is the restriction $m_{\sigma|\sigma'}$.
 
 \end{proof}

\subsubsection{Saturated subsets}\label{saturated}

%The correspondence extends to the subsets.
Recall  that a subset of a strict toroidal embedding is called {\it saturated} if it is the union of  strata.
(see also Definition \ref{saturation}.)

Let $\Sigma$ be the  conical complex associated with a strict toroidal variety $X$.
Similarly with  any subset $\Sigma_0$ of $\Sigma$ one can associate the constructible saturated subset $$X(\Sigma_0):=\bigcup_{\tau \in \Sigma_0} s_{\tau}.$$ Then it follows immediately that 
\begin{lemma} \label{sat2}
$X(\Sigma_0)$ is open (closed under generization) iff $\Sigma_0$ is a subcomplex.
\end{lemma}

In particular,  with a cone $\sigma$ one can associate the open subset $$X(\sigma)=\bigcup_{\tau \preceq \sigma} s_{\tau}=\Star(s,S),$$   
where $s$ is the stratum corresponding to $\sigma$.

%The notion of the support allows to interpret the topology of the subsets $|\Sigma_0|$ of $|\Sigma|$ , and the corresponding subsets $X(\Sigma_0)$ of $X$
\begin{lemma}
The following are equivalent for $\Sigma_0\subset \Sigma$:
\begin{enumerate}
\item $\Sigma_0$ is a subcomplex
\item $X(\Sigma_0)\subset X$ is open
\item $|\Sigma_0|$ is a closed subset of $|\Sigma|$.
\end{enumerate}
\end{lemma}
\begin{proof} The property is local and can be verified for the the closed cover $|\sigma|$ of $|\Sigma|$.

\end{proof}

%\begin{lemma} There is a bijective correspondence between the relative fans $(\Sigma,\Omega)$, and open embeddings of  toric varieties  $U=X_\Omega\subset X=X_\Sigma$. 

%Moreover, any such embedding defines a unique divisor $D=\overline{D_\Omega}$ which is the closure of $D_\Omega:=X_\Omega\setminus T$.
%The strata in $S_\Omega$ on $X_\Sigma$ extend the strata on the toric variety $X_\Omega$.
%Any toric divisor $D$ on $X$ defines the maximal (saturated) toroidal subset (the toroidal locus) of $(X, D)$. 

%So the toric divisors $D$ on $X=X_\Sigma$ are in bijective correspondence with the saturated relative fans $(\Sigma,\Omega)$.
%\end{lemma}
%\begin{proof} The first part follows from Definition. For the second part note that if $X_\Omega\subset X_\Sigma$ is a maximal saturated subset,  where $D$ is toroidal then 
%$X_\Omega$ intersects all the nonempty intersections of the components of $D$. Indeed the generic orbit of such an intersection defines an open affine toric  neighborhood. Since $O_\sigma$ is a generic point, by Lemma we get that $D$ on $X_\sigma$ is equal to $D_\sigma$ so 
%D???\Jarek{}

%\end{proof}

\subsection{Maps of conical  complexes}

\begin{definition} (\cite{KKMS}) \label{map}
A map of conical complexes $f: \Sigma \to \Sigma'$, is a function which assigns to a cone $\sigma\in \Sigma$ a unique cone $\sigma'\in \Sigma'$, together with the linear map $f_{\sigma,\sigma'}: (\sigma,N^{\bf Q}_\sigma) \to (\sigma',N^{\bf Q}_{\sigma'})$, such that,
\begin{enumerate}
\item $f_{\sigma,\sigma'}(N_{\sigma})\subseteq (N_{\sigma'})$. \item $f_{\sigma,\sigma'}(\inte(\sigma))\subset \inte(\sigma')$.
\item If $\tau\preceq \sigma$ then $\tau'\preceq \sigma'$ and  $f_{\sigma,\sigma'}i_{\tau,\sigma}=i_{\tau',\sigma'}f_{\tau,\tau'}$.
\end{enumerate}

\end{definition}
Note  that the map $f$  induces a unique continuous map of topological spaces $|f|: |\Sigma| \to |\Sigma'|$. Equivalently
 
 \begin{definition} (\cite{KKMS}, \cite{Payne}) 
 A map of conical  complexes $f: \Sigma \to \Sigma'$ is a continuous map of topological spaces $|\Sigma | \to |\Sigma '|$ such that, for each cone $\sigma\in \Sigma$  there is some $\sigma' \in \Sigma$ with $f(\sigma)\subseteq \sigma'$ and $f^*M_{\sigma'} \subseteq M_\sigma$.
\end{definition}

\begin{definition} \label{locali}
A {\it subdivision} of a  complex
$\Sigma$ is a map $\Delta\to \Sigma$ such that $|f|$ is a homeomorphism identifying 
$|\Delta|=|\Sigma|$, so that
 any cone $\sigma\in
\Sigma $ is a union of cones $\delta\in
\Delta$.  A subdivision $\Delta$ of $\Sigma$ which is regular is called {\it desingularization} of $\Sigma$.

A map $f: \Sigma \to \Sigma'$ will be called  a {\it local isomorphism} %(respectively {\it local linear isomorphism}
if each $f_{\sigma,\sigma'}$ is an isomorphism.%(respectively a linear isomorphism injective on lattices). 

If $f$ is bijective and is a local isomorphism then $f$   is called  an {\it isomorphism}

If  $f$ is injective and is a local isomorphism then $\Sigma$   is isomorphic to   a {\it subcomplex} of $\Sigma'$.

A map $f: \Sigma \to \Sigma'$ is called  {\it a regular local projection} if 
 for each $\sigma\in \Sigma$ there is a decomposition  $\sigma=\sigma'\times \tau$, where $\tau$ is regular and $f_{\sigma,\sigma'}: \sigma=\sigma'\times \tau \to \sigma'$ is the projection on the first component.
\end{definition}
\begin{lemma} Let $f: \Delta \to \Sigma$, be a subdivision, and $|f|:  |\Delta| \to |\Sigma|$ be the induced homeomorphism of the topological spaces. Then $$\Delta^\sigma:=\{\tau\in \Delta: |\tau|\subseteq|f|^{-1}(|\sigma|)\}$$  defines a fan in $N_\sigma^{\bf Q}$, which is a  subdivision  of the cone $\sigma$.\qed 
\end{lemma}

\subsection{Toroidal morphisms of toroidal embeddings}

The following definitions are equivalent to the definitions of log smooth morphisms in characteristic zero.
\begin{definition} \label{toro mor} A morphism of strict toroidal embeddings $f:(X,U)\to (Y,V)$ is {\it strictly toroidal} if there exists the induced map of open neighborhoods $f':(X',U')\to (Y',V')$, and a commutative diagram of 

\[\begin{array}{rcc}
 (X',U') &  \rightarrow & (X_{\sigma'},T')\\

\downarrow {\scriptstyle f'} & & \downarrow \\
 (Y',V')
 & \rightarrow &
(X_{\sigma} , T)
\end{array},\] 
with horizontal  morphisms \'etale and
a vertical toric map  $(X_{\sigma'},T')\to (X_{\sigma} , T)
$.

\end{definition}
\begin{definition} A morphism of  toroidal embeddings $f:(X,U)\to (Y,V)$ is {\it  toroidal} if there exists the induced map of open \'etale neighborhoods $f':(X',U')\to (Y',V')$, and a commutative diagram as above.

\end{definition}

\subsection{Canonical birational toroidal maps}

The following definition is equivalent to  \cite[Definition 3 p.87, Definition 1 p.73]{KKMS} in view of Theorem \ref{sub}.
\begin{definition}
 A birational  morphism of strict toroidal
embeddings $f: (Y,U) \to (X, U)$ will be called {\it  canonical toroidal} if
for any $x\in s\subset X$ there exists an open neighborhood  $U_x$ of $x$,  an \'etale morphism  $ U_x\to
X_{\sigma}$  and a fan $\Delta^\sigma$ mapping  to $\sigma=\sigma_s$ and the  
 fiber square   of morphisms of  varieties 

\[\begin{array}{rcccccccc}
U_x \times_{X_{\sigma_s}}
X_{\Delta^\sigma} &&\simeq& (f^{-1}(U_x),f^{-1}(U_x)\cap U
 & \rightarrow &
(X_{\Delta^\sigma} , T)&&&\\
&&&\downarrow {\scriptstyle f} & & \downarrow &&& \\
&& & (U_x,U_x\cap U) &  \rightarrow & (X_{\sigma_s},T)&&&

\end{array}\] 

Here $f^{-1}(U_x):=U_x\times_{X}Y$.
\end{definition}
We will prove later that such a local description of 
canonical toroidal morphisms does  not depend upon
the  choice of  \'etale charts. 
This fact can be described nicely using
the following Hironaka condition:

For any geometric $\overline{K}$-points $ x, y$ which are in the same stratum    every
isomorphism $ \alpha : \widehat{X}^{\overline{K}}_x\to 
\widehat{X}^{\overline{K}}_y $ preserving stratification 
  can be lifted
to an isomorphism $\alpha':Y\times_X\widehat{X}^{\overline{K}}_x
\to Y\times_X\widehat{X}^{\overline{K}}_y$ preserving stratification.

(Here $\overline{K}$ is the algebraic closure of $K$, ${X}^{\overline{K}}:=X\times_{\Spec(K)}\Spec(\overline{K})$, and  $\widehat{X}^{\overline{K}}_x:=\Spec(\cO_{{X}^{\overline{K}},x}$)).

The Hironaka condition is extremely important when considering stratified toroidal varieties.  As we can see the condition is satisfied for the canonical birational maps of strict toroidal embeddings. %when using since it is satisfied for $\Delta^\sigma\to X_\sigma$. 
(Lemma \ref{le: tembeddings}).

%The following theorem is due to \cite{KKMS}. 
%(see also \cite{WF} for a slight generalization).
%\begin{theorem} The canonical proper birational toroidal morphismsare in bijective correspondence with subdivisions of the complex. In general, the toroidal maps correspondend to the maps of the  conical complexes. \end{theorem}

%The subcomplexes of $\Sigma$ correspond to open subsets of the corresponding toroidal embedding $X$.
%One can associate with any complex some canonical subcomplexes.

\begin{theorem} (\cite[Theorem 6 p.90]{KKMS}) \label{sub} Let $(X,U)$ be a strict toroidal embedding, and $\Sigma$ be the associated semicomplex.
Then there is a bijective correspondence between the subdivisions of $\Sigma$  and the canonical proper birational toroidal maps.

\end{theorem}
\begin{proof} Again,  the theorem was originally proven over an algebraically closed field but it can be extended to arbitrary fields (with our definition). (It also can be further generalized to the case of the stratified toroidal varieties in Theorem \ref{th: modifications}).

If $Y\to X$ is any morphism of strict toroidal embeddings then
the strata are mapped into strata. Moreover, if $t\in T$ maps to $s\in S$ then its face $t'\preceq t$ maps to $s'\preceq s$.
These and other properties of Definition \ref{map} can be verified locally in the charts where they follow from the properties of the toric maps. Moreover, we get the map of cones between the cones associated with the strata. Because of  Lemma \ref{toroidal-cartier}, and Corollaries \ref{inclusion2}, \ref{inclusion2}, the maps between the cones are independent of charts.
% as they correspond to the maps $M_t^+\to M_s^+$ between the relevant groups of Cartier divisors.
The restrictions of $f$ to the open stars $\Star(s,S_Y) \to \Star(t,S_X)$, induce the maps $M^+_{s,X}\to  M_{t,Y}^+$ and the dual maps of cones $\sigma_s\to \sigma_t$ commuting with face inclusions (by Lemma \ref{inclusion}) and defining the maps of the conical complexes $f:\Sigma_Y\to \Sigma_X$.

Moreover, if $Y\to X$ is a canonical birational toroidal morphism then 
the correspondence between cones defined by the charts shows that the set $\Delta^\sigma:=\{\tau\in \Delta_Y: |\tau|\subseteq|f|^{-1}(|\sigma|)\}$ is a fan inside of $\sigma$. Since the map is proper by using valuative criterion of properness we see that $|\Delta^\sigma|=\sigma$, so $f$ is, in fact a subdivision.

Conversely, a subdivision $\Delta$ of the conical complex $\Sigma$ defines locally for any chart $U\to X_\sigma$ the variety $\widetilde{U}:=U\times_{X_{\sigma}}{X_{\Delta^\sigma}}$ over $U$. The subset $\widetilde{U}$ is the union of the canonical open subsets $\widetilde{U}=\bigcup_{\tau\in \Delta} \widetilde{U}_\tau$, where 

$$\widetilde{U}_\tau:= U\times_{X_{\sigma}}{X_{\tau}}=\Spec \cO(U)(\tau^\vee)^\integ=
\Spec(\sum_{D\in (\tau^\vee)^\integ} \cO(U)(-D))$$
(see also \cite{KKMS}, page 74)
Note that the Cartier divisors in $(\tau^\vee)^\integ\subset M_s$ are naturally contained in $M_s\subset \cK(X)/\cO_X^*$, where $K(X)$ is a constant sheaf of the rational functions on $X$, so $\sum_{D\in (\tau^\vee)^\integ} \cO(U)(-D))$ is a subsheaf of $\cK(X)$ over $U$.

Consequently, the subsets $\widetilde{U}_\tau$
are independent of charts, and the canonically determined morphisms $\widetilde{U}_\tau\to U$ are birational.
This allows to represent $Y$ canonically  by glueing the open subsets $\widetilde{U}_\tau$ over $U$ along the subsets corresponding to their faces so that 
$$Y=\bigcup_{\sigma\in \Sigma}\bigcup_{\tau \in \Delta^{\sigma}} \Spec(\sum_{D\in (\tau^\vee)^\integ} \cO(U_{s_\sigma}(-D))$$
The natural projection $Y\to X$  is proper and separated as it is locally represented  by the morphism $\widetilde{U}_\tau:= U\times_{X_{\sigma}}{X_{\tau}}\to U$.
\end{proof}

As a corollary from the proof  we get: 

\begin{corollary} The strictly  toroidal maps determine  the maps of the associated  conical complexes.\end{corollary}
%\begin{proof} Note that the strictly toroidal morphism $(X,U)\to (Y,V)$ map strata on the canonical stratification $T$ on $(X,U)$ to the strata in $S$ on $(Y,V)$. Moreover, if $t\in T$ maps to $s\in S$ then its face $t'\preceq t$ maps to $s'\preceq s$.
%These and other properties of Definition \ref{map} can be verified locally in the charts where they follow from the property of the toric maps. Moreover, we get the map of cones between the cones associated with the strata. In view of  Lemma \ref{toroidal-cartier}, and Corollaries \ref{inclusion2}, \ref{inclusion2}. The maps between the cones are independent of charts, as the correspond to the maps between the relevant groups of Cartier divisors.

%\end{proof}
\subsubsection{Smooth maps}

\begin{definition} \label{toro mor2} A  morphism of strictly toroidal varieties $f: (Y, D_Y)\to (X, D_X)$ will be called 
{\it smooth } (respectively {\it \'etale }) if it is a smooth ( resp. \'etale)  and $f^{-1}(D_X)=D_Y$
\end{definition}

%\begin{definition} \label{toro mor3} A  morphism of strictly toroidal varieties $(Y, D_Y)\to (X, D_X)$ will be called 
%{\it smooth toroidal} (respectively {\it \'etale toroidal}) if it is a smooth ( resp. \'etale) and strictly toroidal. (see also Definition \ref{toro mor}).
%\end{definition}

%\begin{remark} The smooth morphisms are locally \'etale equivalent to projections along tori, while
%the smooth toroidal morphisms are those corresponding to toric morphisms which are at the same time smooth. 
%\end{remark}

%The following Lemma illustrates some  differences between two notions

\begin{lemma} 
%\begin{enumerate}

%\item 
If $f: (Y, D_Y)\to (X, D_X)$ is a strictly toroidal morphism which is  smooth  then the corresponding  map of the associated conical complexes $\Sigma_X\to \Sigma_Y$ is  a local isomorphism.

%\item The smooth toroidal   morphisms determine  the regular local projections (resp. local isomorphisms) of the associated conical complexes. Conversely, if a strict toroidal map determines a 
%local projection then it is a smooth toroidal morphism.
%\end{enumerate}
\end{lemma}

\begin{proof}   Consider a local chart $\alpha: U\to X_\sigma$, and  the induced  smooth morphism $$\beta:=\alpha\circ f: V:=f^{-1}(U)\to X_\sigma$$ preserving strata. Then one can locally find a coordinate system $u_1,\ldots,u_s$ on the stratum $\beta^{-1}(O_\sigma)\subset V$ defining an \'etale chart $V\to X_\sigma\times \AA^s$.
We can additional assume that $u_1,\ldots,u_s$ do not vanish on $V$, and they induce an \'etale chart
 of the form  $V\to X_\sigma\times T^s$, by shrinking $V$ if necessary. Here $T^s\subseteq \AA^s$ is an $s$-dimensional torus.
  Thus the morphism $f$ in the \'etale charts $\alpha$, $\beta$ is represented by the toric map $X_\sigma\times T^s\to X_\sigma$, defining an isomorphism of the associated cones.

\section{Functorial desingularization of complexes}
\label{Desi}

\subsection{Regular and singular subcomplexes}
In the sequel, we shall say that  cones of a conical complex are {\it disjoint} if their intersection is the zero cone.

\begin{definition}
We say that any nonzero  integral vector $v\in N_\sigma$ is {\it primitive} if it generates the monoid $\QQ_{\geq 0}\cdot v\cap N_\sigma$.

Any strongly convex, finitely generated cone can be written \underline{uniquely} as $$\sigma =\langle v_1,\ldots,v_k\rangle:=\QQ_{\geq 0}\cdot v_1+\ldots+ \QQ_{\geq 0}\cdot v_k,$$ such that $v_i$ are primitive vectors, and $k$ is minimal. 
We shall call vectors $v_i$ {\it vertices} of $\sigma$.
\end{definition}
\begin{definition}\label{singul} (\cite{KKMS})
We say that a cone
$\sigma$ in $N^{\QQ}$ is {\it regular} or {\it nonsingular} if it is generated by a part of a basis of the lattice
$e_1,\ldots,e_k\in N$, written $\sigma=\langle e_1,\ldots,e_k\rangle$. If the cone is not regular it will be called {\it singular}. A complex $\Sigma$ is {\it regular} or {\it nonsingular} if all cones $\sigma\in \Sigma$ are regular.
\end{definition}
\begin{lemma} (\cite{KKMS})
Let $(X,U)$ be a strict toroidal embedding, and $\Sigma$ be the associated semicomplex. 
 Then a cone $\sigma\in \Sigma$ is regular iff the corresponding open subset $X(\sigma)$ is nonsingular. In particular, $\Sigma$ is regular if $X=X(\Sigma)$ is nonsingular.

\end{lemma}

\begin{definition} \label{irreducible} A cone $\sigma$ is called {\it irreducible singular } or simply {\it irreducible} if it cannot be written as $\sigma=\tau\times \sigma_1$, with $\sigma_1$ being regular.
%If $\sigma=\tau\times \sigma_1$, then we shall call the pair $(\sigma,\tau)$ {\it regular}.
\end{definition}
Any singular cone $\sigma$ contains a unique maximal irreducible singular face denoted by $\sing(\sigma)$, so we can write
$$\sigma=\sing(\sigma)\times \reg(\sigma),$$
where $\reg(\sigma)$ is the maximal regular face of $\sigma$ disjoint with $\sing(\sigma)$.
This follows from a simple observation that  an irreducible face of $\tau\times \reg(\sigma)$ is contained in $\tau$.

Denote by $\sing(\Sigma)$ the subset of all irreducible singular faces of $\Sigma$, and let $\Sing(\Sigma)$ denote the minimal subcomplex of $\Sigma$ containing $\sing(\Sigma)$. As we will see later the subset $\sing(\Sigma)$ describes the maximal components of the singular set on $X$.

On the other hand, let   $\Reg(\Sigma)$ the set 
 of all the regular cones in $\Sigma$. Then $\Reg(\Sigma)$ corresponds to the open subset of nonsingular points on $X$.
We see immediately from the definition that

\begin{lemma}\label{local} The  restriction of a regular local projection $f:\Sigma\to \Sigma'$ 
is  a local  isomorphism  $$\Sing(f): \Sing(\Sigma)\to \Sing(\Sigma')$$ on the subcomplexes.

\end{lemma}
\begin{proof} By definition \ref{locali}, 
$\sigma\simeq (f(\sigma))\times \tau$, where $\tau$ is regular. So if $\sigma\in \sing(\Sigma)$ then $f(\sigma)\simeq (\sigma)$. In particular, the cone $f(\sigma)$ is irreducible, and  thus it is in $\Sing(\Sigma')$. Consequently $f$ defines  isomorphisms between the corresponding cones in $\sing(\Sigma)$ and in $\sing(\Sigma')$, as well as  between their faces in $\Sing(\Sigma)$ and in $\Sing(\Sigma')$.

\end{proof}

\subsubsection{Toric divisors} \label{s: toric divisors}

\begin{definition} \label{toric divisors} Any $T$-stable divisor on a toric variety $(X,T)$ will be called a {\it toric divisor}.

\end{definition}
\begin{remark} The maximal toric divisor on on a toric variety $(X,T)$ id given by $D:=X\setminus T$.
	
\end{remark}
\begin{remark} Note that any toric divisor $D$ on a smooth toric variety $X$, induces the  structure of a strict toroidal embedding $(X, D)$.
\end{remark}

\begin{lemma} \label{reg1} Let $D=\bigcup D_i$ be a toric divisor on $X_\sigma$  with the components $D_i$ such that the closed orbit $O_\sigma$ is the  the intersection of all $D_i$.
Then 
$(X_\sigma, D)$ is a strict toroidal embedding if and only if 
$$D=D_\sigma:=X_\sigma\setminus T$$
(So $D$ contains all the irreducible toric divisors on $X_\sigma$ and is the maximal toric divisor on $X_\sigma$.)
\end{lemma}

\begin{proof}
Let  $x\in O_{\sigma}$ be a closed point. By Theorem \ref{KM}, we get that  the group  of the Cartier divisors supported on $D_\sigma$ on $X_\sigma$ is given by $$\Cart(X_\sigma, D_\sigma)=M/(\sigma^\perp)^\integ=M_\sigma.$$ 
%(\sigma^\vee)^\integ
Moreover, $$\rank(\Cart(X_\sigma, D_\sigma))+\dim(O_\sigma)=\dim(X)$$ 
Since $(X, D)$ is strict toroidal and $O_\sigma$ is a  stratum  defined by $D$ we have the same equalities for our $D$ at the point $x$.
Hence $$\rank(\Cart(X_\sigma, D_\sigma))=\rank(\Cart(X_\sigma, D)),$$ as $D$ and $D_\sigma$.
% define the same stratum $O_\sigma$. 
The  group  $\Cart(X_\sigma, D_\sigma)$ can be interpreted as a group of  integral functionals on $N_\sigma\subset N$, so $$\rank(\Cart(D_\sigma))=\rank (M_\sigma)=\dim(\sigma).$$

If $D\neq D_\sigma$, say there exists a component $E_i \in D_\sigma\setminus D$ corresponding to the one dimensional faces (rays)  $\rho_i$ of $\sigma$. Then $\Cart(X_\sigma, D)$ is a proper subgroup of $\Cart(X_\sigma, D_\sigma)$ which corresponds to a subgroup of the integral functionals vanishing on $\rho_i$. This implies that
$\rank(\Cart(s_\sigma, D))<\rank(\Cart(s_\sigma, D_\sigma))$, which is a cotradiction. 
So  $D=D_\sigma$ on $X_\sigma$.
\end{proof}

\begin{lemma} \label{reg2}
Let $D$ be a toric divisor on $X_\sigma$. Then 
$(X_\sigma, D)$ is a strict toroidal embedding if and only if  
$\sigma=\tau\times \sigma_1$, where $\sigma_1$ is regular, and $D=D_{\tau}\times X_{\sigma_1}$.
\end{lemma}
\begin{proof}

Let $O_{\tau}$ be  the generic orbit in the intersection of the divisor components $D_i$ of $D$ corresponding to the rays $\rho_i$. Passing to $X_{\tau}$ we see, by  Lemma \ref{reg1}, that $\rho_i$ generate   $\tau=\langle \rho_1,\ldots, \rho_k \rangle$, and $D_{X_\tau}={D_\tau}$.

Now let $x\in O_\sigma$. If $(X_\sigma, D)$ is a strict toroidal embedding at $x$ then $s_\tau=\overline{O_\tau}=\bigcap D_i$ is the unique smooth toroidal stratum through $x$.  Then, by the characterization of the group of Cartier divisors around a stratum in Lemma \ref{toroidal-cartier}, we have that 
$\Cart(X_\sigma, D)\simeq \Cart(X_\tau,{D_\tau})$ is 
the subgroup of $\Cart(X_\sigma, D_\sigma)$ consisting of  the toric Cartier divisors on $X_\sigma$ supported on $D$.
In other words, the natural map defined by the restriction of Cartier divisors
$$\Cart(X_\sigma, D)\to \Cart(X_\tau, D_\tau)$$ 
is an isomorphism.

But this map is the restriction to the subgroup $\Cart(X_\sigma, D)\subset \Cart(X_\sigma, D_\sigma)$, of the another natural surjection map $$\Cart(X_\sigma, D_\sigma)\to \Cart(X_\tau, D_\tau).$$ 

 Consequently, any nonnegative integral functional  $F$ on $\tau$ extends uniquely to an integral  functional $\overline{F}$ on $\sigma\supset \tau$, such that $\overline{F}_{\rho}=0$ for all one dimensional rays $\rho$ in $\sigma\setminus \tau$. In particular, those rays form  a face $\sigma_1$. We have the exact sequence of the maps of monoids 
$$0\to (\tau^\vee)^\integ \to (\sigma^\vee)^\integ \to  ((\sigma_1)^\vee)^\integ\to 0,$$
where $ (\sigma^\vee)^\integ \to  ((\sigma_1)^\vee)^\integ$ is defined  by the restrictions, and the exact sequence of the corresponding lattices:
$$0\to M_\tau\to M_\sigma\to M_{\sigma_1}\to 0.$$

Both exact sequences split as we have the natural restriction map $(\sigma^\vee)^\integ\to ((\tau)^\vee)^\integ$, and $M_\sigma\to M_{\tau}$. So 
$$M_\sigma\simeq M_\tau\times M_{\sigma_1},$$
$$(\sigma^\vee)^\integ \simeq  (\tau^\vee)^\integ \times  (\sigma_1^\vee)^\integ$$  
Dualizing
$$\sigma\simeq  \tau \times \sigma_1.$$
Moreover, the orbit $O_\tau$ on $X_\sigma=X_{\tau \times \sigma_1}$ corresponds to $T_{\sigma_1}\times O_\tau$, and its closure is equal to $\overline{O_\tau}= O_\tau \times X_{\sigma_1}$ and since it is a smooth stratum we conclude that $X_{\sigma_1}$ is smooth, and $\sigma_1$ is regular.

\end{proof}

\subsection{The determinants of simplicial cones}

\begin{definition} \label{simplicial}
A cone $\sigma$ is {\it simplicial} if it 
is generated over $\QQ$ by linearly
independent primitive vectors $v_1,\ldots,v_k$, written $\sigma =\langle v_1,\ldots,v_k\rangle $.
\end{definition}

To control the singularities of the simplicial cones  $\sigma=\langle v_1,\ldots, v_k\rangle$ one introduces the {\it multiplicity} or {\it determinant} $\det(\sigma)$ of the cone $\sigma$ to be the absolute value of $\det(v_1,\ldots, v_k)$, where the determinant is computed with respect to any basis of the lattice $N_\sigma$.
Set $\Ver(\sigma): =\{v_1,\ldots, v_k\}$ for the set of vertices of $\sigma$.

Let $N_{\Ver(\sigma)}:=\quad \bigoplus\,\, \ZZ v_i\subset N_\sigma$ be the sublattice generated by $v_i$. 
\begin{lemma} Let $\sigma=\langle v_1,\ldots, v_k\rangle$ be a simplicial cone. Then

%The {\it determinant} or {\it multiplicity}  $\det(\sigma)$ of a cone $\sigma$ is the index of the lattices $[N_\sigma :N_{\Ver(\sigma)}]$. In particular,
\begin{enumerate}

%\item If 
\item The order of the quotient group $\frac{N_\sigma} {N_{\Ver(\sigma)}}$ is equal to $n:=\det(\sigma)$, and the cosets in the quotient group  $\frac{N_\sigma}{N_{\Ver(\sigma)}}$ have representatives which are integral vectors of $N_\sigma$ of the form $\sum a_iv_i$, where $0\leq a_i<1$, $a_i\in \frac{1}{n}\cdot \ZZ_{\geq 0}$.
  \item $\sigma$ is regular iff $\det(\sigma)=1$.

\end{enumerate}

\end{lemma}
\begin{proof} It is a well known fact.

(1) By a triangular linear modification (which does not change determinant) one can transform $\Ver(\sigma)$ into a set $\{n_1e_1,\ldots,n_ke_k\}$ where
 $\{e_1,\ldots,e_k\}$  is a basis of $N_\sigma$ so that $$\det(\sigma)=\det(n_1e_1,\ldots,n_ke_k)=n_1\cdot\ldots\cdot n_k=[N_\sigma :N_{\Ver(\sigma)}].$$

If $\sum a_iv_i \in N_\sigma$ and $n=[N_\sigma :N_{\Ver(\sigma)}]$ then $n\sum a_iv_i\in N_{\Ver(\sigma)}$, so $na_i\in \ZZ$.

(2) The condition $\det(\sigma)=1$ means that the set $\Ver(\sigma)$ is a basis of $N_\sigma$. 

\end{proof}
\subsubsection{Minimal vectors}

\begin{definition} \label{minimall} By  a {\it minimal point} or a {\it minimal  vector} of a  cone  $\sigma$ we shall mean  an integral vector $v\in \sigma$, which is not a vertex, and which cannot be decomposed as $v=x+y$, where $x,y\in \sigma^{\integ}\setminus \{0\}$.

\end{definition}

\begin{definition} \label{small} 
More generally a {\it small vector} $v$ is a nonzero integral vector of  a  cone $\sigma=\langle v_1,\ldots, v_k\rangle$  which cannot be written as $v=v_i+y$, where $y\in \sigma$. 
%A vector with coordintes satisfying $0< a_i\leq 1$ will be called the {\it minimal internal.}
\end{definition}
By definition any minimal  vector is small. Moreover,
\begin{lemma} 
A  {\it small  vector} of a simplicial cone  $\sigma=\langle v_1,\ldots, v_k\rangle$ is  a nonzero integral vector of the form
$v=a_1v_1+\ldots+a_kv_k\in N_\sigma$ with $0\leq a_i<1$. \qed
\end{lemma}

Now, let $v\in \sigma$ be a small vector, which is not minimal. Then  there is a  decomposition $v=x+y$, where both $x$ and $y$ are small vectors again. By continuing this decomposition process we show the following:
\begin{lemma} 
\label{minimal5}
\begin{enumerate}
\item Any small vector in a  cone is a nonnegative integral combination of the minimal vectors.

\item Any integral vector of a cone is  a nonnegative integral combination of  minimal vectors and  vertices. 
\end{enumerate}
\qed
\end{lemma}

%Let $\sigma=\langle v_1,\ldots,v_l\rangle$ be a simplicial cone, and $\tau$ be its face. We shall call a pair of cones $(\sigma,\tau)$ regular if $\sing(\tau)

%Denote by $\sing(\Sigma)$ the subset of all irreducible singular faces of $\Sigma$.

 \begin{lemma} 
\label{minimal4} All the small and the minimal points in $\sigma$ are necessarily in $\sing(\sigma)$, and conversely,  $\sing(\sigma)$ is the smallest face of $\sigma$ containing all the minimal or small points.

\end{lemma}
\begin{proof} Any nonzero vector of $\sigma=\sing(\sigma)+\reg(\sigma)$  which is not in $\sing(\sigma)$ is the sum of  a nonzero vector in $\reg(\sigma)$ and a vector in $\sing(\sigma)$, and 
thus cannot be minimal. So the minimal vectors are in $\sing(\sigma)$, and thus their combinations are as well.

\end{proof}

% It may happen However, that the relative interior of an irreducible singular cone contains no minimal points.

\begin{example}(D. Abramovich) \label{Abramovich} Let $\sigma
= \QQ_{\geq 0}^3$, with $N^{\QQ}_\sigma=\QQ^3$, and the lattice 
$$ N_\sigma=\{(a_1,a_2,a_3) :a_i\in \ZZ, a_1+a_2+a_3 \in 2\ZZ\}\subset N^{\QQ}_\sigma$$
Then $\sigma$ is generated by $\{(1,1,0),(1,0,1),(0,1,1),(2,0,0),(0,2,0),(0,0,2)\}$, with the vertices $\{(2,0,0),(0,2,0),(0,0,2)\}$. Thus $\sigma$ contains no minimal points in its relative interior.
\end{example}
Note that the semigroup (or monoid)  of a simplicial cone $\tau$ containing no minimal vectors is   necessarily generated by  $v_i$, and thus $\tau$ is regular. 
Summarizing 
\begin{lemma} \begin{enumerate} \label{min}

\item A simplicial cone $\sigma$ is regular if it contains no minimal  points.
\item If $\det(\sigma)=n>1$ then there exists a minimal or a small point of the form $\sigma$, $v=a_1v_1+\ldots+a_kv_k\in N_\sigma$ with $0\leq a_i<1$, where $a_i\in \frac{1}{n}\cdot N$.
%\item The set of all the minimal points of $\sigma=\sing(\sigma)\times \reg(\sigma)$ is conatined in $\sing(\sigma)$.
%\item If $\det(\sigma)=n\det(\tau)>1$ then there exists a minimal point of the form $\sigma$, $v=a_1v_1+\ldots+a_kv_k\in N_\sigma$ with $0\leq a_i<1$, where $a_i\in \frac{1}{n}\cdot N$.

\end{enumerate}

\end{lemma}

\end{proof}

\subsection{Star subdivisions}
%%%
\begin{definition}\label{de: star} Let $\Sigma$ be a conical complex and $\tau \in \Sigma$. The {\it star} of the
cone $\tau$, the {\it closed star}, and the {\it link} of $\tau$ are
defined as follows:
 $${\rm Star}(\tau ,\Sigma):=\{\sigma \in \Sigma\mid 
\tau\preceq \sigma\},$$ 
$$\overline{{\rm Star}}(\tau ,\Sigma):=\overline{{\rm Star}(\tau ,\Sigma)}$$ 
$${\rm Link}(\tau ,\Sigma)=\overline{{\rm Star}}(\tau ,\Sigma)\setminus {\rm Star}(\tau ,\Sigma)$$

%Assuming that $\Sigma$ is simplicial we define 
%$${\rm Nerve}(\tau ,\Sigma)=\{\sigma\in \overline{ {\rm Star}}(\tau ,\Sigma) \mid \tau\cap \sigma=\{0\}\}$$

%A face $\tau\in \Sigma$  will be called {\it separated} if any two faces in ${\rm Star}(\tau ,\Sigma)$ intersect along a unique face.
%If $\tau$ is a face of a simplicial cone $\sigma$ then by ${\rm Nerve}(\tau ,\sigma)$ we mean  the maximal face of $\sigma$ which is disjoint from $\tau$. 

\end{definition} 

%Immediately from the definition we get
%\begin{lemma} \label{Nerve} Assume that $\Sigma$ is  simplicial.
% and  $\tau$ in $\Sigma$, and $\tau$ is separated. 
%Then  for any $\sigma \in {\rm Star}(\tau ,\Sigma)$, we can write
%$\sigma=\tau+\delta$, where $\delta={\rm Nerve}(\tau ,\sigma)\in {\rm Nerve}(\tau ,\Sigma)$ 
%\qed
%???Moreover, if $\Sigma$ is separated then

%$$ {\rm Star}(\tau ,\Sigma)=\{\tau+\delta\mid \delta\in {\rm Nerve}(\tau ,\Sigma)\}$$
%separated and\qed
%\end{lemma}
%%\begin{proof}% If $\sigma\in {\rm Star}(\tau ,\Sigma)$ then $\sigma=\tau+\tau_0$, where $\tau_0\in {\rm Nerve}(\tau ,\Sigma)$. We need to show that the correspondence is unique. But if two faces $\sigma, \sigma'$ in ${\rm Star}(\tau ,\Sigma)$ correspond to the same $\tau_0\in {\rm Nerve}(\tau ,\Sigma)$, so they both contain  $\tau$, and $\tau_0$. Thus its intersection also contains both, and since it is a single face (so it is convex) it contains $\tau+\tau_0=\sigma=\sigma'$.????

%\end{proof}

%\begin{definition} For any $\sigma\in  {\rm Star}(\tau ,\Sigma)$, by ${\rm Nerve}(\tau ,\sigma)$ we mean the unique face $\delta\in {\rm Nerve}(\tau ,\Sigma)$, such that $\sigma=\tau+\delta$.

%\end{definition}

 %\begin{lemma} A face $\tau\in \Sigma$  is separated if all its vertices (one-dimensional faces) are separated.A face $\tau\in \Sigma$ is separated if  all the vaerices in its $\Nerve(\tau,\Sigma)$ are separated. \end{lemma}

For any $\sigma \in \Link(\tau
,\Sigma),$  there is a unique minimal face of $\Star(\tau,\Sigma)$, which contains $\sigma$. We denote it, by the abuse of notations by $\tau+\sigma$.

\begin{definition}\label{de: star subdivision} Let $\Sigma$ be a conical complex and $v$ be a primitive vector
 in the
relative interior of  $\tau\in\Sigma$. Then the {\it star
subdivision}
%$\varrho$ be a ray passing
 $v\cdot\Sigma$ of $\Sigma$ at
$v$ is defined to be
$$v\cdot\Sigma=(\Sigma\setminus {\rm Star}(\tau ,\Sigma) )\cup
\{\langle \sigma_v\mid   \sigma\in \Link(\tau
,\Sigma)\},$$ 
where $$\sigma_v:=\langle v\rangle +\sigma\subset \tau+\sigma,$$ 
with  $i_{\sigma,\sigma_v}$ is the restriction of the unique map $i_{\sigma, \tau+\sigma}$. 
%Uniqueness follows  
The vector $v$ will be called the {\it center} of the star subdivision.
\end{definition} 

One can extend the definition of the center for the purpose of functoriality: 
\begin{definition}\label{de: star subdivision22} Let $\Sigma$ be a complex and $V=\{v_1,\ldots,v_k\}$ be a set of the primitive vectors $v_i\in \inte(\tau_i)$
 in the
relative interiors of  the cones $\tau_i\in\Sigma$ for $i=1,\ldots,k$ defining  disjoint stars $\Star(\tau_i ,\Sigma)$.

The {\it star
subdivision}
%$\varrho$ be a ray passing
 $V\cdot\Sigma$ of $\Sigma$ at
$V$ is defined to be
$$V\cdot\Sigma=v_1\cdot\ldots\cdot v_k\cdot\Sigma=
(\Sigma\setminus \bigcup_{i=1,\ldots,k} {\rm Star}(\tau_i ,\Sigma) )\cup
\bigcup_{\sigma\in \Link(\tau_i
,\Sigma)}\langle v_i \rangle+\sigma    .$$ 

The set of vectors $V$ is the {\it center} of the star subdivision.
\end{definition} 

\begin{remark} Let $V=\{v_1,\ldots,v_k\}$ be the center of the star subdivision of $\Sigma$. Then $V\cdot \Sigma=
	v_1\cdot\ldots\cdot v_k\cdot\Sigma$ is independent of the order of the vectors $v_i$ in $V$.
\end{remark}

%%%%%%%

\begin{definition}
A {\it multiple  star subdivision} of $\Sigma$ is  a subdivision obtained as a sequence of the star subdivisions at the consecutive centers $V_1,\ldots,V_k$. A regular  subdivision is called a {\it  desingularization}.

%We shall assume the natural face relation on the faces in $v\cdot\Sigma$. That is 
%If  $\sigma_1\preceq \sigma_2$ , with $\sigma_2$ in $\Link(\tau
%,\Sigma)\}$ then  $\sigma_1\preceq \sigma_2(v):=\langle v\rangle +\sigma_2$, with the natural face inclusion $i_{\sigma_1, \sigma_2(v)}$.

\end{definition}

The star subdivisions at minimal points allow to resolve singularities of simplicial faces.
\begin{lemma} (\cite{KKMS})\label{des} Let $v\in \inte(\tau)$ be  a minimal point of $\tau\in \Sigma$.  Then for any cone $\sigma\in \Star(\tau,\Sigma)$ the resulting cones in $v\cdot \sigma$ in the the star subdivision $v\cdot \Sigma$ of the complex $\Sigma$ have smaller determinants then  $\det(\sigma)$.
\end{lemma}

\begin{proof}
Let  $\tau=\langle v_1,\ldots,v_k\rangle$, and $\sigma =\langle  v_1,\ldots,v_s\rangle$, and write $v=a_1v_1+\ldots+a_kv_k$ with $0\leq a_i<1$.
Then  for the cone $$\sigma_i=\langle v_1,\ldots,\check{v_i},\ldots,v_s \rangle $$ we have
$$\det(\sigma_i)=|\det(v,v_1,\ldots,\check{v_i},\ldots, v_s\rangle)|=a_i|\det(v_1,\ldots,v_s)|=a_i\det(\sigma).$$

\end{proof}

\subsection{Barycenters and irreducible barycentric subdivisions}
In our  desingularization algorithm, one  considers the canonical centers of the star subdivisions  in the relative interiors of the irreducible cones. There are several ways of doing this. One could associate with any cone $\sigma=\langle v_1,\ldots,v_k\rangle $ the canonical center $v_1+\ldots+v_k$.  However, for the applications in Section \ref{stratified}, we need another choice of the centers in the desingularization of semicomplexes.

\begin{definition} \label{mi} \label{bar} By a {\it minimal internal vector} of $\sigma$ we mean an integral vector $v\in\inte(\sigma)$ which  cannot be represented as the sum of two nonzero  integral vectors in $\sigma$, such that at least one of them is in the relative interior $\inte(\sigma)$. Then the sum of all its minimal internal vectors  $v_\sigma$ will be called the {\it canonical barycenter} of $\sigma$.

\end{definition}
%For any (rational ) cone $\sigma$ the semigroup $\sigma\cap N_\sigma$ is finitely generated.
%Thus for any strictly convex cone $\sigma$ the set of its minimal generators is finite and generates together with all the vertices the semigroup $\sigma\cap N_\sigma$. %The  minimal generators are minimal points, and each minimal point is the sum of minimal generators (not necessarily distinct).

%Immediately from the definition we have:
%\begin{lemma} \label{bar} Let $\sigma$ be an irreducible cone. Then the sum of all its minimal internal vectors  $$v_\sigma:=w_1+\ldots+w_r$$ is in the relative interior of $\sigma$ so can be chosen as the canonical barycenter.  
%\end{lemma}

%\begin{proof} Otherwise all the minimal vectors are in a proper face  $\tau$, so $N_\sigma/N_{\Ver({\sigma})}$ is generated by $N_\tau$. Let $\{v_1,\ldots,v_k\}={\Ver({\sigma})}\setminus N_\tau$.  Then write any integral vector $v$ in $\sigma$ can be written as $v=a_1v_1+\ldots a_kv_k+b_1w_+\ldots+b_sw_s$, where $w_i$ are either vertices of $\tau$, or
%the minimal vectors of $\sigma$
	
%\end{proof}

\begin{definition} \label{barycenter} Let $\Sigma$ be a conical complex. By the {\it canonical irreducible barycentric subdivision} of $\Sigma$, we mean  a sequence of the star subdivisions at the sets of all barycenters  of all the irreducible faces in $\Sigma\setminus \Omega$ of the same  dimension in order of decreasing dimension.
\end{definition}

\begin{lemma} \label{bar2} If  $\Delta$ is a canonical irreducible barycentric subdivision of $\Sigma$ then $\Delta$ is simplicial.

\end{lemma}
\begin{proof} If  $\delta$ is a face of $\Delta$ then all its new rays ( or vertices) are linearly independent of the other rays. So $\delta$ has a unique maximal face $\tau$ which is in $\Sigma$, and this face is regular, as it contains no irreducible face and  thus $\sing(\tau)=\{0\}$. Consequently, $\Ver(\delta)\setminus\Ver(\tau)$ are linearly independent from $\Ver(\tau)$, and $\Ver(\tau)$ are linearly independent.

\end{proof}

\subsection{Marked complexes}
\label{marking}
The procedure described in Lemma \ref{des} allows to resolve singularities of simplicial complexes by applying the star subdivisions at the minimal points.
Unfortunately, the choice of such minimal points is highly noncanonical. To eliminate choices, we introduce here the concept of \emph{marking}.

\begin{definition} \label{marki} A {\it marking} on a complex $\Sigma$ is a partially ordered subset $V$ of the set of all vertices $\Ver(\Sigma)$ of $\Sigma$ such that the following conditions are satisfied.
\begin{enumerate}
\item For any cone $\sigma$ in $\Sigma$ the set $V(\sigma):=V\cap \Ver(\sigma)$ is linearly independent of the remaining vertices in $\Ver(\sigma)\setminus V$. It means $\sum_{v_i\in \Ver(\sigma)} c_iv_i=0$ implies that $c_i=0$ for each $v_i\in V(\sigma)$. 
\item The order on $V$ is total on each subset $V(\sigma)$
\end{enumerate}

A complex with a marking will be called {\it marked}.
 We say that a face $\tau$ of a complex $\Sigma$,  is {\it completely marked} if $V(\sigma)=\Ver(\sigma)$. A subcomplex $\Sigma_0$  is {\it completely marked} if all its faces are completely marked. 
 A face is {\it unmarked} if $V(\sigma)=\emptyset$. 
 The set of all unmarked faces of $\Sigma$ forms the {\it maximal  unmarked subcomplex} $U(\Sigma)$.  %A subcomplex $\Sigma_0$ will be called {\it completely marked} if $V(\Sigma_0):=V\cap \Ver(\Sigma_0)$ is equal to $\Ver(\Sigma_0)$.

\end{definition}
\begin{definition} A complex with marking will be called {\it regularly marked} if all the unmarked faces are regular.
	
\end{definition}

\begin{remark}
	
The empty set $V=\emptyset$ defines the trivial marking on a complex.
\end{remark}

%\begin{remark} \label{mark2} Given  any completely marked simplicial subcomplex $\Sigma_0$ of a marked complex, one can define the canonical order $\leq_V$ on the set of the vectors in  $|\Sigma_0|$.
%We order the set of vertices of any face $\sigma=\langle v_1,\ldots, c_k\rangle$  of $\Sigma_0$ according to the order on $V$, which is total on each face of $\Sigma$. For any $v\in \sigma$ we define the lexicographic order on the  the coefficients $c_i$ in the presentation $v=\sum c_iv_i$. (Note that the face $\sigma$ is simplicial, and vectors $v_i$ are linearly independent.)
%\end{remark}

\subsubsection{Natural order on marked complexes}
Let $\Sigma$ be a  complex with a marking $V\subset Ver(\Sigma)$. One can introduce the natural partial order on vectors in $|\Sigma|$.
Each such vector $v$ is in the relative interior of a single cone $\sigma\in \Sigma$. Then it can be written
as $$v =\sum_{v_i\in V(\sigma)} c_iv_i+\sum_{v_i\in \Ver(\sigma)\setminus V}c_iv_i$$ 
By Condition (1) of Definition \ref{marki}, we shall associate with it uniquely the linear combination $$\pi(v):=\sum_{v_i\in V(\sigma)} c_iv_i$$
Then  we say that   $\pi(v)> \pi(v')$  if in the difference
$$\pi(v)-\pi(v')=\sum (c_i-c_i')v_i$$ for any vertex $v_i$ with negative coefficient there exists at least one vertex $v_j>v_i$ with a positive coefficient.

Then we set $v> v'$ if $\pi(v)> \pi(v')$. 
 
 This defines a partial order which is total for the points in any completely marked cone. The order in such a face is lexicographic, when the coefficients form a sequence determined by the descending order on the vertices.

The following simple observation is critical for the algorithm.
\begin{lemma} \label{uni} In any regularly marked face there exists a unique small vector which is minimal with respect to
the marking order.
\end{lemma}
\begin{proof} Note that any regularly marked face $\sigma=\langle v_1,\ldots,v_k\rangle$ is simplicial, and its maximal unmarked face $\sigma_0$ is regular and spanned by $v_{i}\in \Ver(\sigma)\setminus V$.

For $j=0,1$ we can write two minimal vectors $w_1, w_2$ of $\sigma$ as $$w_j =\sum_{v_{i}\in V\cap \Ver(\sigma)} a_{ji}v_{i}+\sum_{v_{i}\in \Ver(\sigma_0)}b_{ji}v_{ji}.$$ 
Since both vectors are minimal for the marking order,  the coefficients $a_{ji}$ are the same for $j=0,1$. Since the vectors $w_j$ are small and the cone $\sigma_0$ is regular we have that $0\leq b_{ji}<1$.
The difference  $$w_0-w_1=\sum_{v_{ji}\in \Ver(\sigma_0)}(b_{0i}-b_{1i})v_{i}$$
is thus an integral vector with the coefficients satisfying inequalities 
 $$-1<(b_{0i}-b_{1i})<1$$ So $(b_{0i}-b_{1i})=0$, since the cone $\sigma_0$ is regular, and  the coefficients $(b_{0i}-b_{1i})$ are integral.
\end{proof}
%\begin{remark} Note that, although this lemma is used only in the simplicial case, it is proven without this assumption.

%\end{remark}

\subsubsection{Marking defined by star subdivisions}

Marking naturally occurs when applying  star subdivisions of complexes.

\begin{lemma} \label{mark} Let $\Sigma$ be a complex with marking $V\subset \Ver(\Sigma)$.
Let $\Sigma'$ be obtained by a sequence of star subdivisions of a complex $\Sigma$ at the consecutive centers $V_1,\ldots, V_k$. (see Definition \ref{de: star subdivision22}). Then there exists a natural  marking $V':=V  \cup V_1\cup\ldots\cup V_k$, which extends the order on $V\setminus (V_1\cup\ldots\cup V_k)$ and such that
\begin{enumerate}
 \item All the vectors in $V_1\cup\ldots\cup V_k$ are greater than vertices in $V\setminus (V_1\cup\ldots\cup V_k)$ 
\item $v_i<v_j$ if $i<j$, and $v_i\in V_i$, $v_j\in V_j$ are in a common face of $\Sigma$.
\end{enumerate}
\end{lemma}
\begin{proof} Note that all the new vertices of $\Sigma'$ are  the centers of the star subdivisions. On the other hand, each center of a star subdivision is linearly independent from other vertices in the newly constructed faces. Moreover, this property is preserved under the consecutive star subdivisions. 
Thus, for any face $\sigma'$ of $\Sigma'$, the new vertices   (the centers of the star subdivisions) are linearly independent and marked. So Condition (1) of Definition \ref{marki} is satisfied.

Any cone has vertices that are either some vertices of a cone in $\Sigma$ or are elements of some $V_i$. By the construction, each cone may have at most one vertex in any set $V_i$.  This implies Condition (2) of Definition \ref{marki}.
%This extends to a unique order described in the lemma.

\end{proof}

The following is the auxiliary result used in the proof of Theorem \ref{can des}
\begin{lemma}\label{can mar} Let $\Sigma$ be a \underline{regularly marked complex}. 
There exists a  canonical multiple  star desingularization $V_1\cdot\ldots\cdot V_k\cdot\Sigma$ of $\Sigma$ such that
\begin{enumerate}
\item The centers $V_i$ are sets of points which lie 
 in $|\sing(\Sigma)|$, and no regular faces are affected.

\item The centers of the consecutive  star subdivisions are  sets of  minimal points  in the interior of  singular irreducible faces.

The algorithm is functorial for regular local projections and local isomorphisms of complexes,
in the sense that the centers transform functorially with the trivial subdivisions removed. 
\end{enumerate}

\end{lemma}
\begin{proof} 
By definition, a regularly marked  complex is simplicial.
By Lemma \ref{des}, any star subdivision at minimal points improves singularities. To
 control this process we shall consider here a polynomial invariant $$P_{\det}(\Sigma)(t):=\sum a_it^i$$ whose coefficients  $a_i$ of  $t^i$ are  the numbers of the maximal faces with  determinant equal to $i$. We  can order the polynomials lexicographically. Equivalently:  $$P_{\det}> P_{\det}'\quad\quad \rm{ if}\quad\quad  \lim_{t\to \infty}(P_{\det}-P_{\det}')>0.$$

Then, by Lemma \ref{des}, after any star subdivision $\Sigma'$ of $\Sigma$
at minimal points the polynomial invariant drops $P_{\det}(\Sigma')<P_{\det}(\Sigma)$.
 We  choose canonically the center of the star subdivision to be the set of the  integral vectors of $|\Sigma|$ which are minimal for the canonical partial order defined by the marking. Then, by Lemma \ref{uni},
  \begin{enumerate}
 \item Any face $\sigma$ of the complex contains at most one point $P_\sigma\in \inte(\sigma)$ which is in  the center of a star subdivision.
 	\item The point $P_\sigma$ is a minimal  vector of a cone $\sigma$ (in the sense of Definition \ref{minimall}).
 	\item The choice of the center point $P_\sigma$ is independent of the other faces of the complex.
 \end{enumerate}

% We can use the following invariant to determine the points of the center.
 
 %For any $v\in |\Delta|=|\Sigma|$:
%$$\inv(v)=(\dim(\tau),\tau, -d_1,\ldots,-d_k)$$ 
%where $v\in \inte(\tau), \tau\in \Sigma$, $\tau=\langle w_1,\ldots,w_k\rangle$, $w_1<\ldots<w_k$ and
%$$v=d_1w_1+\ldots+d_kw_k, \quad 0\leq d_i<1,$$  

%Note that there is a partial order on the set of cones in $\Sigma$, defined by the comparison of the vertices ordered lexicographically. Any two cones of the same dimension which are contained  in a certain  face can be compared.

%We apply the star subdivision at the sets of the minimal points of the consecutive subdivisions $D(\Delta)$ of $\Delta$ for which the invariant $\inv$ is  maximal. If there are several points for which the invariant is minimal  then any cone in $D(\Delta)$ contains at most one such a point. This is because we always take centers in the interior of the irreducible faces of maximal dimension of $\Sigma$, and in each such a face there is exactly one such a minimal point.
%So the corresponding stars in $D(\Delta)$ (which are contained in the stars in $\Sigma$) have  disjoint the relative interia and the star subdivisions commute.
%We run the algorithm until the faces of the resulting subdivision have no minimal points and thus  are regular. By Lemma \ref{des}, during each subdivision the maximal determinant of the subdivided faces 
The procedure of the star subdivisions at such centers clearly terminates since  the set of the polynomials with  nonnegative integral coefficients has a dcc property so any descending sequence stabilizes.
\end{proof}

\subsection{Canonical desingularization of conical complexes}

\begin{theorem} \label{can des} There exists a canonical desingularization $\Sigma'$ of a  conical complex $\Sigma$, that is a 
sequence of star subdivisions $(\Sigma_i)_{i=0}^n$ of $\Sigma$ such that $\Sigma_0=\Sigma$ and $\Sigma_n=\Sigma'$ is regular. Moreover,
\begin{enumerate}
\item  All the centers  are in $|\sing(\Sigma)|$, where  $\sing(\Sigma)$ is the subset of all the irreducible singular faces of $\Sigma$ \footnote{Definition \ref{irreducible}}. The centers are in the relative interiors of irreducible faces of the intermediate subdivisions.

\item The centers of the consecutive star subdivisions are either  sets of
  minimal points, %\footnote{Definition \ref{minimall}} in   singular simplicial \footnote{Definition \ref{simplicial}}
%irreducible faces of the induced subdivisions, or the sets of the 
or sets of the canonical barycenters of irreducible faces. 
%\footnote{Lemma \ref{bar}}that is is the sums of the minimal internal vectors of irreducible faces.
\item  The subdivision does  not affect the set  $\Reg(\Sigma)$ of  all the regular cones. %\footnote{Definition \ref{singul}}

\item  The algorithm is functorial for regular local projections and local isomorphisms \footnote{Definition \ref{locali}} of complexes, in the sense that the centers transform functorially with the trivial subdivisions removed.  In the case of surjective local isomorphisms, the centers of the star subdivisions transform functorially.

\end{enumerate}

\end{theorem}
\begin{proof} %It is possible to run this algorithm without any changes directly on an arbitary nonesparated  coniacal complexes but for the purpose of the simplicity of the arguments we shall assume first that the complexes are separated. Then by the functoriality argument we extend the algorithm  to arbitrary nonseparated complexes.

 %Recall that  $\Sing(\Sigma)$ denotes the closure of $\sing(\Sigma)$. While $\sing(\Sigma)$ consists of irreducible singular cones its closure $\Sing(\Sigma)$ contains some regular cones which shall not be subdivided.
 
 %By Lemma \ref{subdivision}, any (regular) star subdivision of  $\Sing(\Sigma)$
%extends uniquely to  a   ( regular) star subdivision of $\Sigma$. Let  $\Reg(\Sing(\Sigma))$ denote the subcomplex of $\Sing(\Sigma)$ consisting of all regular cones in $\Sing(\Sigma)$.

We take the trivial ( the empty) marking on  $\Sigma$. 
Consider  the  canonical irreducible barycentric subdivision  $B(\Sigma)$ of $\Sigma$ (as in Definition \ref{barycenter}). By Lemma \ref{mark}, we create some  marking on the set of new vertices, defined by the dimensions of the  subdivided cones.  The subdivision will not affect any regular cones.
 %in %$\Reg(\Sing(\Sigma))$, 
%and all the  cones from $\sing(\Sigma)$ will be subdivided. 
It induces a simplicial subdivision of $\Sigma$, with the induced marking $V$ which is regularly marked, as
all untouched faces must be regular. Then the resulting regularly marked complex  can be resolved
by Lemma \ref{can mar}.

%and such that its maximal unmarked subcomplex $U(B(\Sing(\Sigma))$ of $B(\Sing(\Sigma))$ consists of $\Reg(\Sing(\Sigma))$- the regular part of $\Sing(\Sigma)$. Indeed, the untouched faces of $\Sing(\Sigma)$ are exactly the regular cones.

Commutativity with regular local projections follows immediately from Lemmas \ref{local}, and \ref{can mar}.

\end{proof}
\section{Desingularization of relative complexes} \label{relative complexes}
In this Chapter we extend the combinatorial desingularization algorithm to a relative situation.
\subsection{Relative complexes} \label{reli}
\begin{definition} By a {\it relative conical complex} $(\Sigma,\Omega)$  we mean a pair of a conical complex $\Sigma$ and its subcomplex $\Omega\subseteq \Sigma$.

We say that a cone $\sigma\in \Sigma$ is {\it balanced} if it contains a unique maximal face $\omega$ which is in $\Omega$. In such a case we shall call $(\sigma,\omega)$ a {\it pair} in $(\Sigma,\Omega)$, and write 
$(\sigma,\omega)\in (\Sigma,\Omega)$. The relative complex $(\Sigma,\Omega)$ is balanced if any $\sigma\in \Sigma$ is balanced.

Let  $(\sigma,\omega)\in (\Sigma,\Omega)$, where $\sigma=\langle v_1,\ldots,v_k,w_1,\ldots,w_s\rangle\in \Sigma$ be a (balanced) cone with the maximal face $\omega=\langle w_1,\ldots,w_s\rangle\in \Omega$. 
We say that a pair $(\sigma,\omega)$ is {\it regular} if $\sing(\sigma)\preceq \omega$.  We say that $(\sigma,\omega)$   
is {\it simplicial} if $v_1,\ldots,v_k$  are linearly independent from $w_1,\ldots,w_s$. Equivalently we say  that $\sigma$ is {\it relatively regular} (respectively  {\it relatively simplicial}) if $\sigma$ is a   a balanced cone with a maximal face $\omega \in \Omega$, and $(\sigma,\omega)$ is { regular} (resp. simplicial). If $\sigma$ is not { relatively regular} (In particular,  if it is not balanced or simplicial) it will be called {\it relatively singular}.

The relative complex $(\Sigma,\Omega)$ is balanced  (respectively regular or simplicial) if any cone $\sigma\in \Sigma$ is  is balanced (resp. relatively regular or relatively simplicial).
\end{definition}
\subsection{Relative irreducibility}\label{reli2}

As in the previous case of $\Omega=\{0\}$ we can introduce a relative version of irreducibility.
Let  $(\Sigma,\Omega)$  be a  relative complex. 
We say that $\sigma\in \Sigma$  is {\it relatively irreducible} if
it contains no proper face $\tau$ which contains both $\sing(\sigma)$ and $\sigma\cap |\Omega|$. 
In particular, any balanced relatively irreducible face $\sigma$ which is not in $\Omega$ is singular since $\sing(\sigma)\neq \{0\}$.

 Any  cone $\sigma\in \Sigma$ contains a unique maximal relatively irreducible  face denoted by $\sing_\Omega(\sigma)$. Equivalently $\sing_\Omega(\sigma)$
 is a unique minimal face containing $\sing(\sigma)$ and $\sigma\cap |\Omega|$.
 
 By definition $\sing(\sigma)\subseteq \sing_\Omega(\sigma)$, and thus as in Definition \ref{irreducible}, we can write
$$\sigma=\sing_\Omega(\sigma)\times \reg_\Omega(\sigma),$$
where $\reg_\Omega(\sigma)$ is the maximal regular face of $\sigma$ disjoint from $\sing_\Omega(\sigma)$ (inersecting it at the zero cone).
We shall need the following
\begin{lemma} \label{relative regular} $\sing_\Omega(\sigma)\in \Omega$ iff
$\sigma$  is relatively regular. 
\end{lemma}
\begin{proof} If $\omega:=\sing_\Omega(\sigma)\in \Omega$ then $\omega=\sigma\cap |\Omega|$ is a unique maximal face of $\sigma$ in $\Omega$. Moreover, $\sing(\sigma)\subset \omega$.

Conversely, suppose  that $\sigma$ is relatively regular. Then it is balanced and $\sing(\sigma)\subset \omega$, where $\omega$ is a unique maximal face of $\sigma$ in $\Omega$. Then  $\sing_\Omega(\sigma)=\omega\in \Omega$.

\end{proof}

%A face $\sigma'$ of $\sigma$ is relative
%\begin{corollary} \label{rel} Let $(\sigma,\omega)\in (\Sigma,\Omega)$ be a regular pair such that
 %$\sigma=\tau+\omega$, where $\tau$ is a face of $\sigma$ 
 %which is disjoint from $\omega\in \Omega$. Then $\tau$ is regular.
	%\end{corollary}

%\begin{proof} Write $\sigma=\sing(\sigma)+\reg(\sigma)$.
%Then since $\sing(\sigma)$ is a face  of $\omega$ we infer
%that
%$\tau$ is a face of $\reg(\sigma)$, and is regular.
	
%\end{proof}

\subsection{Singular faces}
Denote by $\sing(\Sigma,\Omega)$ the subset of all relatively irreducible  faces of $\Sigma$, and let $\Sing(\Sigma,\Omega)$ be its the smallest subcomplex of $\Sigma$ containing $\sing(\Sigma,\Omega)$. %The subset $\sing(\Sigma,\Omega)$ describes the maximal components of the singular set on $X$.
Denote by    $\Reg(\Sigma,\Omega)$ the set 
 of all the relatively regular cones in $\Sigma$. 
 
 \begin{lemma} \begin{enumerate}
 \item $\sing(\Sigma,\Omega)$ contains $\Omega$.
 \item $\Reg(\Sigma,\Omega)$ is a subcomplex of $\Sigma$ containing $\Omega$.
 \item $\sing(\Sigma,\Omega)\cap \Reg(\Sigma,\Omega)=\Omega.$
 %\item $\Sing(\Sigma,\Omega)\setminus \sing(\Sigma,\Omega)$ consists of relatively regular  cones???
 \end{enumerate}
 
 \end{lemma}
 
 \begin{proof} (1) follows from definition.
 
 (2) if $\sigma$ is relatively regular with maximal cone $\omega\preceq \sigma$ which is in $\Omega$.
 and $\tau$ is its face then $\omega\cap \tau\in \Omega$ is maximal in $\tau$. Moreover, $\sing(\tau)\subseteq \sing(\sigma)\subseteq \omega$, hence  $\sing(\tau)\subseteq  \omega\cap \tau.$ So  $\Reg(\Sigma,\Omega)$ is a subcomplex of $\Sigma$.
 
 (3) If $\sigma$ is relatively regular  and irreducible then $\sigma=\sing_{\Omega}(\sigma)=\omega$.
 \end{proof}

\begin{remark} These  results are interpreted geometrically in Lemma \ref{strat2} in the toric situation. In the toroidal case, they have  the same meaning.
The set $\sing(\Sigma,\Omega)$ describes the strata of the singular locus defined by both: the singularity type and the divisor  $D=\overline{D_\Omega}$.
%(as in Definition \ref{omega}).  
On the other hand, the open subset $X_{\Reg(\Sigma,\Omega)}$ is the saturation of the open set $(X_\Omega, D_\Omega)$ defined by $\Omega$ in $(X_\sigma, D)$.
(See Lemmas \ref{cr}, and \ref{strat2} (in the toric situation)).  \end{remark}

\subsection{Regular relative complexes and SNC divisors}\label{reli4}

\begin{definition} \label{toroidd} Let $(X, D_X)$ be a strict toroidal embedding. By a {\it toroidal divisor} on $X$ we mean a  divisor $D\subset D_X$.
Usually $(X, D)$  is not a strict toroidal embedding.
 By the {\it closed strata} of $D=\bigcup_{i\in J} D_i$ with irreducible Weil components $D_i$ we  mean  the irreducible components of the intersections  $\bigcap_{i\in I} D_i$ of $D_i$. Denote by $S_D$ the induced stratification with strata
the components of  $$\bigcap_{i\in I} D_i\setminus (\bigcup_{i\in (J\setminus I)} D_i),$$ and by $\overline{S}_D$ the set of the induced closed strata.

\end{definition}
 Note that the strata of $S_D$ are the unions of those on the toroidal embedding $(X, D_X)$ (defined by $D_X$). Thus $S_D$ is coarser than the canonical stratification on $(X, D_X)$. Moreover,  the strata in $S_D$ are not necessarily smooth.

\begin{lemma} \label{rr} Let $(X, D_X)$ be a strict toroidal embedding with the associated complex $\Sigma$. There exists a bijective correspondence between the saturated open toroidal subsets $V=X(\Omega)$ of $(X, D_X)$ and the subcomplexes $\Omega\subset \Sigma$. (see notation in Section \ref{saturated}).
\end{lemma}
 \begin{proof} The Lemma is a rephrasing of Lemma \ref{sat2}.
  %The open saturated subsets are the union of strata and closed under generization. They correspond to a collection of cones $\Omega$ of $\Sigma$ closed under the face relation: If $\sigma\in \Omega$ then all its faces are in $\Omega$. So they form a complex. 
 
 \end{proof}

\begin{definition}\label{omega2} A subcomplex $\Omega\subset \Sigma$ will be called {\it saturated} if 
any cone of $\sigma$ with rays (one dimensional faces) in $\Omega$ is in $\Omega$.
\end{definition}

Immediately from the definition we obtain the following characterization of saturated subcomplexes:
\begin{lemma} If $(\Sigma,\Omega)$ is balanced then the subcomplex $\Omega\subset \Sigma$ is saturated.\qed
	\end{lemma}

The saturated subcomplexes have the following geometric description:
\begin{lemma}\label{comp}  The subcomplex $\Omega\subset \Sigma$ is saturated in $\Sigma$ iff
the complement $$X\setminus X(\Omega)=X(\Sigma\setminus\Omega))$$ is a divisor.

\end{lemma}
\begin{proof} Assume $E:=X\setminus X(\Omega)$ is a divisor.
%Let $E$ denote  the closure of $X(\Sigma\setminus\Omega)$. 
%Then $X\setminus E$ is an open saturated subset corresponding to a  subcomplex $\Omega'\subseteq \Omega$ of $\Sigma$.  
%Then  $\Omega$ and $\Omega'$ have the same vertices. 
If the vertices of $\sigma=\langle v_1,\ldots,v_k\rangle \in \Sigma$ are in $\Omega$ then
the closure of the corresponding stratum $\overline{s_\sigma}$ is the intersection of the irreducible components $D_i$ corresponding to $v_i$. But the stratum $s_\sigma$ is not in $E$, since the only divisors which contain it are $D_i$, and they are all intersecting $X(\Omega)$ and thus are not contained in $E$. So $s_\sigma\subset X\setminus E=X(\Omega)$, and $\sigma \in \Omega$. We use the property that any closed stratum is the intersection of the irreducible divisors containing it.

Conversely, if $\Omega$ is saturated, then any stratum $s$ in $E=X\setminus X(\Omega)$, corresponds to a cone $\sigma\in \Sigma\setminus \Omega$. Since $\Omega$ is saturated,  at least one of  the vertices of $\sigma$ is not in $\Omega$ and corresponds to a divisor $D\subset E$. Thus $s$ is contained in an irreducible Weil divisor $D\subset E$. This implies that $E$ is a divisor.

\end{proof}
%\begin{definition} \label{omega} Let $(X, D_X)$ be a strict toroidal embedding with $U=X\setminus D_X$ and  with the associated complex $\Sigma$.
%Then for any complex $\Omega\subset \Sigma$ we define $D_\Omega:=X(\Omega)\setminus U=X(\Omega)\cap D_X$, which is a divisor on an open subset $X(\Omega)$.\end{definition}

%Recall that by  Theorem \ref{cc}, there is a bijective correspondence between the strata in the canonical stratification $S$ on  $(X, D_X)$ and faces $\sigma\in \Sigma$: $\sigma\mapsto s_\sigma$

\begin{lemma} \label{cr} 
%Let $(X, D_X)$ be a strict toroidal embedding, with associated complex $\Sigma$. 
Let $(X, D_X)$ be a strict toroidal embedding,  with the associated complex $\Sigma$.  
There is a bijective correspondence between the toroidal divisors $D\subset D_X$ and the  saturated subcomplexes $\Omega$ of $\Sigma$.
Given a  toroidal divisor $D$  on $X$, 
 there is a unique saturated subcomplex $\Omega\subset \Sigma$ which is defined by the set of vertices of $\Sigma$ corresponding to the components of $D$. Conversely, $D=\overline{D_\Omega}$, where $$D_\Omega:=X(\Omega)\setminus U=X(\Omega)\cap D_X.$$ 
 \qed
 %which%\item The strata in $S_D$ correspond to cones in $\Omega$. Any closed stratum in $\overline{s_D}$ is of the form $\overline{s_\sigma}$, where $\sigma\in \Omega$.
%In particular, the
%The  cones $\sigma\in\Omega$ are defined by the closed strata $\overline{s_\sigma}$ which are in $\overline{S}_D$.
%,  and for which $(X, D)$ is a strict  toroidal embedding at a neighborhood of $s_\sigma$.
 %\item The open  subset $X(\Omega)$ is toroidal in $(X, D)$, and intersects all strata of $D=\overline{D_\Omega}$. 

%Moreover, $V=X_{\Reg(\Sigma,\Omega)}$. 

\end{lemma}

%\begin{proof}
 %If  $\sigma\in \Sigma$ has   its vertices in $\Omega$ then the corresponding  components of $D$ intersect, and is saturated.

 %(2) If $\sigma\in \Sigma$ defines the closed stratum $\overline{s_\sigma}$ in $S_D$, and for which $(X, D)$ is a toroidal at $s_\sigma$ then, by Lemma \ref{reg1}, $D$ locally corresponds to $D_\sigma$ in a neighborhood of $s_\sigma$, so $\sigma$ and all the rays and all the faces of $\sigma$  are in  $\Omega$.
%Conversely,  if $\sigma$ is in $\Omega$, then $(X, D)$ is  strict  toroidal embedding on $X(\Omega)$, and in particular, at $s_\sigma$. 
%in particular,  $\Omega$ is a complex saturated in $\Sigma$. 

%(3) By (2) the open  subset $V=X(\Omega)$ is toroidal in  $(X, D)$.  Moreover, $V=X(\Omega)$ intersects all the divisorial strata of $D$, so  all the strata on $V$ are toroidal, and $S_{D|V}=S_{|V}$, and $(V, D\cap V)$
 %is  a strict toroidal embedding, with the stratification defined by
%$D_{|V}=D_\Omega$, corresponding to the subcomplex $\Omega$, and $D$ is the closure of $D_\Omega$.  

%Part (2) follows immediately from Lemma
%\ref{regular} below.
 %\end{proof}

 \begin{lemma} \label{regular}Let $(X, D_X)$ be  a strict toroidal embedding with the associated complex $\Sigma$. Let
 %There is a bijective correspondence between the saturated toroidal subsets $(X', D)$, of $(X, D)$ defined by the toroidal divisor $D\subseteq D_X$, and  relative complexes $(\Sigma,\Omega)$.
 $\Omega$  be a saturated subcomplex in $\Sigma$, and let $D=\overline{D_\Omega}\subseteq D_X$.
 Then
 \begin{enumerate}

   \item $(X, D)$ is a strict toroidal embedding  if  $(\Sigma,\Omega)$ is   a regular relative complex. In such a case $(X, D)$ is the saturation of $(X(\Omega), D_{\Omega})\subset (X, D)$, and   $E:=\overline{D_X\setminus D}=X\setminus X(\Omega))$ is a relative SNC divisor on $(X, D)$.
   \item Conversely, if $(X, D)$ is a strict toroidal embedding, where $D\subseteq D_X$ then $D=\overline{D_\Omega},$ where $(\Sigma,\Omega)$ is regular relative complex.
 \item The toroidal locus $(X, D)^{\tor}$ of $(X, D)$ is defined by the saturated  subcomplex $\Reg(\Sigma,\Omega)$ of $\Sigma$. It is the toroidal saturation 
of  $(X(\Omega), D_\Omega)$ in $(X, D)$.

%(Definition \ref{locus})

  \end{enumerate}
 \end{lemma}

\begin{proof} %With the preceding  notation, w 
(1) Using a chart we can reduce the situation to the toric case.
Let $(\sigma,\omega)\in (\Sigma,\Omega)$ be a regular pair. 
 Then  $X_\sigma=X_\omega\times X_\tau=X_\omega\times \AA^k$. By  construction, the divisor  $$D\cap X_\sigma=D_\omega\times \AA^k,$$ corresponds to the vertices of $\omega=\Omega\cap \sigma$.
  The complement  $E\cap X_\sigma=X_\omega\times D_\tau$ corresponds to the remaining vertices of $\sigma$ and has SNC with $D_\omega\times \AA^k$. Moreover the set $(X(\Omega),D_\Omega)$ corresponds in a local chart to  $(X_\omega\times T, D_\omega\times T)$ and $(X_\sigma,D_\omega\times \AA^k)$ is its saturation.

(2) Conversely, assume that  $(X, D)$ is a strict toroidal embedding. Then, by Lemma \ref{reg2}, for any $\sigma\in \Sigma$, we have that $\sigma=\omega\times \tau$, where $\tau$ is regular, and $D\cap X_\sigma=D_\omega\times X_\tau$, whence $(\Sigma,\Omega)$ is regular.

(3) Follows from (2).

%if $(\Sigma,\Omega)$ is regular relative complex (so $\Reg(\Sigma,\Omega)=\Sigma)$).

 %defining the toroidal structure $(X_\sigma, D_\omega\times \AA^k)=(X_\omega\times \AA^k, D_\omega\times \AA^k)$ on $X_\sigma$. The  

 %the complement   $E=X\setminus X(\Omega)$
%is  a divisor. It locally corresponds to $X_\omega\times D_\tau$ and has SNC crossings with $D_\omega\times \AA^k$. 
%By Lemma \ref{reg2}?, 
%By Lemma \ref{comp}?

%By Lemma \ref{cr}, 

\end{proof}

 %Then $\Reg(\Sigma)$ corresponds to the open subset of nonsingular points on $X$.

\subsection{Subdivisions of relative complexes}
\begin{definition}
A map $f: (\Sigma,\Omega) \to (\Sigma',\Omega')$ is a map of complexes $\Sigma\to \Sigma'$ which induces a map
of the subcomplexes $\Omega\to \Omega'$.

A {\it subdivision} of a  relative complex
$(\Sigma,\Omega)$ is a subdivision $\Delta\to \Sigma$ which is identical on $\Omega\subset \Delta$, and thus defines a  relative complex $(\Delta,\Omega)$.  A subdivision $(\Delta,\Omega)$ of $(\Sigma,\Omega)$, which is  a regular relative complex is called {\it desingularization} of $(\Sigma,\Omega)$.

A map $f: (\Sigma,\Omega) \to (\Sigma',\Omega')$ is a {\it local isomorphism} if each $f_{\sigma,\sigma'}$ is an isomorphism  mapping $\Omega\cap \sigma$ to $\Omega'\cap \sigma'$.  

A map $f: (\Sigma,\Omega) \to (\Sigma',\Omega')$ is an isomorphism if  $f$ is a bijection of the sets  and  a local isomorphism. 
%If  $f$ is injective and is local isomorphism then $\Sigma$   is called  a {\it subcomplex} of $\Sigma'$.

A map $f: (\Sigma,\Omega) \to (\Sigma',\Omega')$  is called  {\it a regular local projection} if 
 for each $\sigma\in \Sigma$ there is an isomorphism  $\sigma\simeq \sigma'\times \tau$, where $\tau$ is regular, which takes $\Omega\cap \sigma$ to $(\Omega'\cap \sigma')\times \tau$, and   
 such that $$f_{\sigma,\sigma'}: \sigma\simeq\sigma'\times \tau \to \sigma'$$ is the projection on the first component. \end{definition}

%We see immediately that
\begin{lemma}\label{local2} A regular local projection $f: (\Sigma,\Omega) \to (\Sigma',\Omega')$
induces  a local  isomorphism  of the relative subcomplexes: $$\Sing(f): \Sing(\Sigma,\Omega)\to \Sing(\Sigma',\Omega').$$

\qed
\end{lemma}

\subsection{Minimal vectors in relative complexes} 

\begin{definition} \label{minrel0}
Let $\sigma \in \Sigma$ be a  cone with a unique  maximal face $\omega\in \Omega$. %We say that $(\sigma,\omega)$ is regular or $\sigma$ is relatively regular if $\sing(\sigma)\in  \Omega$. We say that $(\sigma,\omega)$ is simplicial if $v_1,\ldots,v_k$  are linearly independent from $w_1,\ldots,w_s$.

 %A {\it minimal relative generator}  of   $\sigma$ is a vector $v\in \sigma\setminus |\Omega|$, such that $v$ is not a vertex of $\sigma$ and is not the sum of the two nonzero integral vectors in $\sigma$.
A {\it minimal vector} (respectively {\it small vector}) of the  pair $(\sigma,\omega)$ is a minimal  (respectively {\it small}) vector of the cone $\sigma$ which is not in $\omega$.

\end{definition}

\begin{lemma} \label{minrel}
Let $\sigma=\langle v_1,\ldots,v_k,w_1,\ldots,w_s\rangle \in \Sigma$ be a relatively simplicial cone with a unique  maximal face $\omega=\langle w_1,\ldots,w_s\rangle\in \Omega$. 
Then a {\it minimal vector} of  $(\sigma,\omega)$ can be  written down in the form $\sum c_iv_i+\sum d_jw_j$, where 
 $0\leq c_i< 1$, and 
 at least one $c_i\neq 0$.
%\item The set $\{v_i\}\cup \{w_j\mid  d_j\neq 0\}$ is linearly independent.

%and which is not a sum of a vector in $\omega$ and some other nonzero vetcor in 
\qed

\end{lemma}

%A face $\sigma'$ of $\sigma$ is relative

%\begin{lemma} \label{minimal}
 %A minimal generator of a simplicial $(\sigma,\omega)$ is  a minimal vector. Any minimal vector is the sum of minimal generators and possibly some vertices of $\omega$. \end{lemma}
 %\begin{proof} If $v=\sum c_iv_i+\sum d_jw_j$ is not a minimal vector then $c_j\geq 1$ for some $i$ do
 %$w:=v-v_j\in \sigma$, and $v=v_j+w$ is not a minimal generator.
 
 %By the inductive argument $v$ is the sum of vectors  with $c_i<0$ which cannot be written as the sum of others vectors. 
 %\end{proof}

\begin{lemma} \label{minrel2} Let $\sigma=\langle v_1,\ldots,v_k,w_1,\ldots,w_s\rangle \in \Sigma$ be a relatively simplicial cone with a unique  maximal face $\omega=\langle w_1,\ldots,w_s\rangle\in \Omega$. Then  $(\sigma,\omega)$  is regular if there is no minimal vector of $(\sigma,\omega)$.
\end{lemma}
\begin{proof}  If  $(\sigma,\omega)$ is regular. Then $\sigma=\omega\times \reg_{\Omega}(\sigma)$, and all the minimal vectors of $\sigma$ are in $\omega$.
Conversely, suppose all the minimal vectors of $\sigma$ are in $\omega$. Then, by Lemma \ref{minimal5}(2), 
 any integral vector $v\in \sigma$ can be written  as a nonnegative combination of the vertices $v_1,\ldots,v_k$, and vectors in $\omega^\integ$.
Since $v_1,\ldots,v_k$ are linearly independent from vectors in $\omega$ we infer that $N_\sigma$ is generated by $v_1,\ldots,v_k$, and a basis of $N_\omega$. Hence $v_1,\ldots,v_k$ is a part of a basis of $N_\sigma$ and generate  a regular face $\reg_{\Omega}(\sigma)$ of $\sigma$, whence $\sigma^\integ=\omega^{\integ}\times \reg_{\Omega}(\sigma)^{\integ}$.

 \end{proof}
 
 \subsection{Star subdivisions}
 Let $(\Sigma,\Omega)$ be a relative simplicial complex.  \begin{definition}\label{de: star subdivision2} Let $(\Sigma,\Omega)$ be a relative conical complex and $v$  be a primitive vector
 in the
relative interior of  $\tau\in\Sigma\setminus \Omega$. Then the {\it star
subdivision}
%$\varrho$ be a ray passing
 $v\cdot(\Sigma,\Omega)$ of $(\Sigma,\Omega)$ at
$v$ is defined to be
$$v\cdot(\Sigma,\Omega):=(v\cdot\Sigma,\Omega)$$
Analogously if  $V=\{v_1,\ldots,v_k\}$ is  a set of the primitive vectors $v_i$
 in the
relative interior of  the cones $\tau_i\in \Sigma\setminus \Omega$ for $i=1,\ldots,k$ defining the disjoint stars $\Star(\tau_i ,\Sigma)$ then 
the {\it star
subdivision}
%$\varrho$ be a ray passing
 $V\cdot(\Sigma,\Omega)$ of $(\Sigma,\Omega)$ at
$V$ is defined to be
$$V\cdot(\Sigma,\Omega)=(V\cdot\Sigma,\Omega)$$ 

A {\it multiple  star subdivision} of $(\Sigma,\Omega)$ is  a subdivision obtained as a sequence of star subdivisions at the consecutive centers $V_1,\ldots,V_k$. A {\it subdivision} of $(\Sigma,\Omega)$ which is a regular relative complex is called  {\it  desingularization}.

\end{definition}
%\begin{lemma} \label{subdivision2}%\begin{enumerate}
 
%Any multiple star subdivision or star desingualarization of $\Sing(\Sigma,\Omega)$ at centers contained in  $\sing(\Sigma,\Omega),$  extends canonically to  the multiple star subdivision or the desingularization   of $\Sigma$.

%\item
% Let $\Sigma_0$ be a subcomplex of $\Sigma$ which contains  $\Sing(\Sigma,\Omega)$. Assume that any cone in $\Sigma\setminus\Sigma_0$ intersects $\Sigma_0$ along it face. Then  any subdivision or  desingularization of $\Sigma_0$ extends canonically to the subdivision or desingularization of $\Sigma$
 %\end{enumerate}
% \end{lemma}
%\begin{proof} (1) If $\sigma\in \Sigma$ then $\sigma=\sing_\Omega(\sigma)\times \reg_\Omega(\sigma)$, where $\sing_\Omega(\sigma)\in \Sing(\Sigma,\Omega).$

%Any  subdivision of $\sing_\Omega(\sigma)$ extends naturally to the subdivision of $\sigma$ . The extension commutes with faces so it defines a subdivision of the complex.
%(2) If $\sigma$ is a cone in $\Sigma\setminus\Sigma_0$ then it intersects $\Sigma_0$ at the face $\tau$ containing $\sing(\sigma)$. Thus $\sigma=\tau \times \sigma_1$, where $\sigma_1$ is regular.

%\end{proof}

 \subsection{Determinants of  relative  subdivision}
\label{determ} 
 We introduce the measure for the singularity of the relative simplicial cones:
 
  If $(\sigma,\omega)$ is simplicial with $\omega=\langle w_1,\ldots,w_s \rangle$, and $\sigma=\langle v_1,\ldots,v_k,w_1,\ldots,w_s\rangle$ then we put  $$\det(\sigma,\omega): =N_\sigma/(N_{\Ver(\sigma)}+N_\omega)=\det(\overline{v}_1,\ldots,\overline{v}_k),$$
 where $\overline{v}_1,\ldots,\overline{v}_k$ are the images of $v_1,\ldots,v_k$ are in $N_\sigma/N_\omega$.
The following lemma is a consequence of  the definition:
 \begin{lemma} \label{regular2}  If $(\sigma,\omega)$ is simplicial  then 
$\det(\sigma,\omega)=1$ iff $(\sigma,\omega)$ is regular.
\end{lemma}
 \begin{proof} If $\det(\sigma,\omega)=1$ then $v_1,\ldots,v_k$ generate that lattice $N_\sigma/N_\omega$. So $\sigma\simeq\langle v_1,\ldots,v_k \rangle \times \omega$ is relatively regular.

 \end{proof}

 %\begin{lemma} \label{product2} If $\omega=\langle v_1,\ldots,v_k\rangle$, and $\sigma \langle v_1,\ldots,v_k,w_1,\ldots,w_k\rangle$, are both simplicial.  Then 
%$$\det(\sigma)=\det(\omega)\cdot\det(\sigma,\omega).$$
%\end{lemma}
%\begin{proof}  $$\det(\sigma,\omega)=N_\sigma/(N_{\Ver(\sigma)}+N_\omega)=\frac{N_\sigma/N_{Ver(\sigma)}}{{N_\omega}/N_{Ver(\omega)}},$$
%since $N_{Ver(\omega)}=N_{Ver(\sigma)}\cap N_\omega$.

%The proof is identical as the proof of Lemma
%\ref{product}.

%\end{proof}

%\begin{lemma} \label{product3} Let $\sigma=\langle v_1,\ldots,v_k,w_1,\ldots,w_k,z_1,\ldots,z_s\rangle$,
 %$\omega=\langle v_1,\ldots,v_k\rangle$, and  $\omega=\langle z_1,\ldots,z_s\rangle$. Assume that $(\sigma,\omega)$ is  simplicial  and  $w_i+N_\tauN^{\bf Q}$ are primitive in $N_\sigma/N_\tau$. Then 
%$$\det(\sigma,\omega)=\det(\tau)\cdot\det(\sigma/\tau,(\omega+\tau)/\tau).$$
%\end{lemma}
%\begin{proof} 

%\ref{product}.

%\end{proof}

\subsection{The invariant of relative cones}

We introduce a somewhat richer invariant measuring a progress under the star subdivisions at  minimal vectors.
Set $$\mu(\sigma,\omega):=(\dim(\omega),\det(\sigma,\omega)) \in \NN^2.$$ We  order the set of values $\NN^2$ of $\mu(\sigma,\omega)$ lexicographically.

The following extends Lemma \ref{des}(\cite{KKMS}):

\begin{lemma} \label{des2} Assume that $(\Sigma,\Omega)$ is simplicial and 
let $(\tau,\omega_\tau)\in (\Sigma,\Omega)$ 
 Let $v\in \inte(\tau)$ be  a minimal point of $(\tau,\omega)$.  Then for any pair $(\sigma,\omega)$, where $\sigma\in \Star(\tau,\Sigma)$ the resulting pairs $(\sigma',\omega')$ in $v\cdot (\sigma,\omega)\subseteq v \cdot \Sigma$ of the relative complex $(\Sigma,\Omega)$ have a strictly smaller invariant $\mu(\sigma',\omega')$  than  $\mu(\sigma,\omega)$.
\end{lemma}
\begin{proof}
Let  $\omega_\tau=\langle w_1,\ldots,w_r\rangle$,\quad 
$\tau=\langle v_1,\ldots,v_l, w_1,\ldots,w_l \rangle$,\\ 
$\sigma=\langle v_1,\ldots,v_s, w_1,\ldots,w_k \rangle$, 
and write $$v=a_1v_1+\ldots+a_kv_l+c_1w_1+\ldots+c_rw_r$$ with $0\leq a_i<1$.

There are two types of  the faces of maximal dimension of $v\cdot \sigma$ denoted by $\sigma_i$, and $\delta_j$ , where
$$\sigma_i:=\langle v, v_1,\ldots,\check{v_i},\ldots,v_s,w_1,\ldots,w_k\rangle,$$ for $i\leq l$  contain $\omega$, and $$\delta_j:=\langle v_1,\ldots,v_r\rangle+\omega_j,$$
 containing $\omega_j\in \Omega$, which 
  is a facet (codimension one face) of $\omega$. 
 
 Then $(\delta_j,\omega_j)\in (v\cdot\Sigma,\Omega)$ and  $\mu(\delta_j,\omega_j)<\mu(\sigma,\omega)$, since $\dim(\omega_j)<\dim(\omega)$.
On the other hand,  for each  cone $\sigma_i$ we have
$$\det(\sigma_i,\omega)=|\det(\overline{v},\overline{v}_1,\ldots,\check{\overline{v}_i},\ldots, \overline{v}_k)|=a_i|\det(\overline{v}_1,\ldots,\overline{v}_k)|=a_i\det(\sigma,\omega)<\det(\sigma,\omega), $$
where $\overline{v_i}$ are the images of $v_i$ in $N_\sigma/N_\omega$. This implies that $\mu(\sigma_i,\omega)<\mu(\sigma,\omega)$.
\end{proof}

\subsection{Barycentric subdivision}
 \begin{definition} Let $(\Sigma,\Omega)$
 be a relative conical complex.   Let $\sigma\in \Sigma$ be a relatively  irreducible cone.   Denote by  $w_\sigma$  the sum of all vertices of $\Omega\cap\sigma$, and by $z_\sigma$ the sum of all the minimal internal vectors in $\sigma$.  
  %$$z_\sigma:=w_1+\ldots+w_r$$ 
  %(Definition \ref{mi}).
 Then we shall call the sum $v_\sigma:=w_\sigma+z_\sigma\in \inte(\sigma)$
 the {\it canonical barycenter} of $\sigma$.  
\end{definition}
%\begin{proof} Otherwise $v_\sigma$  is in the relative interior of a certain proper face $\tau$ of $\sigma$ containing $\sing(\sigma)$ and $\omega$ and $\sigma$ is not irreducible.

%\end{proof}

 \begin{definition} \label{barycenter2} Let $(\Sigma,\Omega)$ be a relative conical complex. By the {\it canonical irreducible barycentric subdivision} of $(\Sigma, \Omega)$ we mean the sequence of the  star subdivisions at the sets of all barycenters $v_\sigma$ of all the irreducible faces in $\Sigma\setminus \Omega$ of the same  dimension in order of decreasing dimension.
\end{definition}

\begin{lemma} \label{bar4} If  $(\Delta,\Omega)$ is the canonical irreducible barycentric subdivision of $(\Sigma,\Omega)$ then $(\Delta,\Omega)$ is simplicial. Moreover, all faces of $\Sigma$ which are
not subdivided are relatively regular.
\end{lemma}
\begin{proof} %The proof is nearly identical as for Lemma \ref{bar}. 
Let $\delta$ be a face of $\Delta$. Then, by definition of the star subdivision,  all its new rays (vertices) are linearly independent of the other "old" rays, and these old vertices form a face which is in $\Sigma$. Hence $\delta$ has a unique maximal face $\sigma\in \Sigma$. 
Then $\sigma$ and its irreducible face  $\omega:=\sing_\Omega(\sigma)$ are intact, and thus $\omega$ is in $\Omega$, since otherwise it would have been decomposed. This implies, by Lemma \ref{relative regular}, that $\sigma$ is relatively  regular, and $(\sigma,\omega)$ is regular.
Moreover, the rays defined by the vertices $\Ver(\delta)\setminus \Ver(\sigma)$  are not in $\Omega$, so  $\omega$ is a unique maximal face of $\delta$ which is in $\Omega$, so $\delta$ is balanced.
Since  $\Ver(\delta)\setminus\Ver(\sigma)$ are linearly independent from $\Ver(\sigma)$, and $\Ver(\sigma)\setminus \Ver(\omega)$ are linearly independent from $\Ver(\omega)$ we conclude that  $\Ver(\delta)\setminus\Ver(\omega)$ are linearly independent from $\Ver(\omega)$.

\end{proof}

\subsection{Marked relative complexes}
\begin{definition} A {\it marking} on a relative complex $(\Sigma,\Omega)$ is a partially ordered subset $V$ of the set of all vertices $\Ver(\Sigma)$ of $\Sigma$ such that the following conditions are satisfied.
\begin{enumerate}

\item $\Ver(\Omega)\subseteq V$
\item Set $V(\sigma):=V\cap \Ver(\sigma)$ for any $\sigma\in \Sigma$. The set $V(\sigma)\setminus \Omega$ is linearly independent of the remaining vertices in $(\Ver(\sigma)\setminus V(\sigma))\cup \Omega$. 
%It means $\sum_{v_i\in \Ver(\sigma)} c_iv_i=0$ implies that $v_i=0$ for each $v_i\in V(\sigma)$. 
\item For any $\sigma\in \Sigma$ the order on $V$ is total on each subset $V(\sigma)$.

\end{enumerate}

A face $\sigma\in \Sigma$ is {\it unmarked} if  $V(\sigma) =\emptyset$. 
A face $\sigma\in \Sigma$  with  $(\sigma,\omega)\in (\Sigma,\Omega)$ is {\it relatively unmarked} if $(\sigma,\omega)$ is relatively simplicial, and all the vertices in $\Ver(\sigma)\setminus \Ver(\omega)$ are unmarked.
 The set of all unmarked faces of $(\Sigma,\Omega)$ forms the {\it maximal  unmarked subcomplex} $U(\Sigma,\Omega)$. Note that $|U(\Sigma,\Omega)|\cap |\Omega|=\{0\}$.
 A marked relative complex $(\Sigma,\Omega)$  is {\it regularly marked}  if it is relatively simplicial and all the relatively unmarked faces are relatively regular. In particular all the unmarked faces in $\Sigma$ are 
 regular.  
  %A subcomplex $\Sigma_0$ will be called {\it completely marked} if $V(\Sigma_0):=V\cap \Ver(\Sigma_0)$ is equal to $\Ver(\Sigma_0)$.

\end{definition}
One can extend Lemma \ref{mark} to the relative case.

\begin{lemma} \label{mark2} Let $(\Sigma,\Omega)$ be a relative complex with marking $V\subset \Ver(\Sigma)$.
Let $(\Sigma',\Omega)$ be obtained by a sequence of star subdivisions of a complex $(\Sigma,\Omega)$ at the consecutive centers $V_1,\ldots, V_k$. (disjoint from $|\Omega|$). Then there exists a natural  marking $V':=V \cup V_1\cup\ldots\cup V_k$, which extends the order on $V\setminus (V_1\cup\ldots\cup V_k)\supseteq \Ver(\Omega)$ and such that
\begin{enumerate}
 \item All the vectors in $V_1\cup\ldots\cup V_k$ are greater than vertices in $V\setminus (V_1\cup\ldots\cup V_k)$ 
\item $v_i<v_j$ if $i<j$, and $v_i\in V_i$, $v_j\in V_j$, and
 $v_i, v_j$ are in a face of $\Sigma$.
\end{enumerate}
\end{lemma}
\begin{proof} The reasoning is identical as in the 
proof of Lemma \ref{mark}.

\end{proof}

\begin{lemma} \label{mark3} Given  any marked relative complex $(\Sigma,\Omega)$, one can define the canonical partial order $\leq_V$ on the set of all the vectors in  $ |\Sigma|$.

\end{lemma}

\begin{proof}
Let  $\sigma=\langle v_1,\ldots, v_k,w_1,\ldots, w_s\rangle$  be a face of $\sigma$, with $V(\sigma)=
\{w_1,\ldots, w_s\}$.
We order the set of vertices $V(\sigma)=\{w_1,\ldots, w_s\}$   according to the order on $V$, which is total on each face of $\Sigma$. For any $v\in \sigma$ we
write $$v=\sum c_iv_i+\sum d_jw_j.$$ 

Then we define the lexicographic order on the  the coefficients  $d_j$ in the presentation $$\pi(v):=\sum d_jw_j.$$

Since $v_1,\ldots, v_k$ are linearly independent of
$w_1,\ldots, w_s$ the vector $\pi(v)$ is unique. Although the presentation of $\pi(v)$ is, in general, not unique there is   a unique smallest presentation of $\pi(v)$ with respect to the lexicographic order determined by the total order on $w_1,\ldots, w_s$.
We put $v>v'$ if $\pi(v)>\pi(v')$.
\end{proof}

\begin{lemma} \label{uni2} In any regularly marked face, there exists a unique small vector which is minimal for the marking order.
\end{lemma}
\begin{proof} Note  that in any regularly marked face there exists a unique maximal unmarked face which is regular. Then we apply the same argument as in the proof of Lemma \ref{uni}.

\end{proof}

\begin{lemma}\label{can mar2} Let $(\Sigma,\Omega)$ be a regularly marked  simplicial relative  complex. 
There exists a  canonical multiple  star desingularization $V_1\cdot\ldots\cdot V_k\cdot\Sigma$ of $\Sigma$ such that
\begin{enumerate}
\item The centers lie 
 in $|\sing_\Omega(\Sigma)|$, and no  faces in $\Reg_\Omega(\Sigma)$ are affected.

\item The centers of the consecutive  star subdivisions $V_i$ are the sets of   minimal vectors (definition \ref{minrel}) in the interior of nonempty relatively irreducible faces of $\Sigma$ which are not in $\Omega$.

\item  
The algorithm is functorial for regular local projections and local isomorphisms of complexes, preserving the order.
% in the sense that the centers transform functorially with the trivial subdivisions removed. 
\end{enumerate}

\end{lemma}
\begin{proof} The proof is nearly identical to the proof  of Lemma 
\ref{can mar2}. By the assumption, $(\Sigma,\Omega)$ is regularly marked, and , in particular, it is   relatively simplicial.
% This follows from the condition 
%Note that there is a partial order on the set of cones in $\Sigma$, defined by the comparison of the vertices ordered lexicogarphically. Any two cones of the same dimension which are contained  in a certain  face can be compared.
We consider the polynomial invariant associated with
any subdivision $(\Delta,\Omega)$ of the relative complex $(\Sigma,\Omega)$:
$$P\mu(\Delta,\Omega)=\sum c_{ij}x^iy^j,$$ 
 where $c_{ij}$ is the number of the maximal faces $\sigma$ of $\Delta$ with the invariant $\mu(\sigma.\omega)=(i,j)$ from Lemma \ref{des2}. We consider the lexicographic order on the monomials $x^iy^j$. This extends to the lexicographic order on the polynomials.

We apply the star subdivisions at the sets of the minimal points for  the partial order defined by the marking.  
By Lemma \ref{uni}, these sets of minimal vectors are uniquely determined, and there is at most one in each face.

 Note also that, by Lemma \ref{mark2}, the intermediate subdivisions of  $(\Sigma,\Omega)$ are marked.

 %If there are several points for which the invariant is maximal  then any cone in $D(\Delta)$ contains at most one such a point. This is beacuse we always take centers in the interior of the irreducible faces of maximal dimension of $\Sigma$, and in each such a face there is exactly one such a minimal point.
%So the corresponding stars in $D(\Delta)$ (which are contatined in the stars in $\Sigma$) have  disjoint the relative interia and the star subdivisions commute.
%We run the algorithm until the relative faces of the resulting subdivision have no minimal points and thus  are relatively regular. 

By Lemma \ref{des2}, after each star subdivision the invariant $P\mu(\Delta,\Omega)$ drops. So the procedure terminates. It terminates when there is no minimal vector in $(\Sigma,\Omega)$. By Lemma \ref{minrel2}, the relative complex $(\Delta,\Omega)$ without minimal vectors is  relatively regular.

\end{proof}

\subsection{Canonical desingularization of  relative conical complexes}
\begin{theorem} \label{can des2} Let $(\Sigma,\Omega)$ be a relative conical complex. Assume that there is a partial order on the set of vertices $\Ver(\Omega)$, which is total on each set $\Ver(\omega)$, with $\omega\in\Omega$.
 %with a  completely marked subcomplex $\Omega$. 

There exists a canonical  desingularization $(\Delta,\Omega)$ of a  relative conical complex $(\Sigma,\Omega)$, that is a 
sequence of star subdivisions $(\Sigma_i,\Omega))_{i=0}^n$ of $(\Sigma,\Omega)$ such that $(\Sigma_0,\Omega)=(\Sigma,\Omega)$ and $(\Sigma_n,\Omega)=(\Delta,\Omega)$ is {\it  regular}, that is $\sing(\Delta,\Omega)= \Omega$. \footnote{Section \ref{reli2}
}

Moreover,
\begin{enumerate}
\item  All the centers of the star subdivisions are in $|\sing(\Sigma,\Omega)|$. The centers are in the relative interiors of relatively irreducible faces $\sigma$\footnote{Section \ref{reli}} of the intermediate subdivisions.

\item The centers of the consecutive star subdivisions are  either the barycenters of the relatively irreducible faces 
 or  the minimal vectors  of the relatively simplicial irreducible cones.\footnote{Definition \ref{minrel}} %in particular, they lie in the relative interior of  relatively  irreducible faces of the induced subdivisions.
\item  The subdivision does  not affect the set  $\Reg(\Sigma,\Omega)$.

\item  The algorithm is functorial for regular local projections and local isomorphisms of complexes preserving the order on $\Omega$, in the sense that the centers transform functorially with the trivial subdivisions removed. 

\end{enumerate}

\end{theorem}
\begin{proof} The proof is a slight modification of the proof of Theorem \ref{can des}. The order on $V:=\Ver(\Omega)$ defines a marking on $(\Sigma,\Omega)$. 
%We use the modified version of the lemma.
%$U(\Sing(\Sigma,\Omega))=\{0\}$. This implies that $(\Sing(\Sigma),\Omega\cap \Sing(\Sigma))$ is completely marked and can be canonically resolved by Lemma \ref{can mar2}.

 %By Lemma \ref{subdivision2}, any regular star  subdivision of  $\Sing(\Sigma,\Omega)$ at centers $\sing(\Sigma,\Omega)$ 
%extends uniquely to  a    regular subdivision of $(\Sigma,\Omega)$. It suffices to desingularize $\Sing(\Sigma,\Omega)$.

Consider  the  canonical relative barycentric subdivision  $(\Delta,\Omega)$ of $(\Sigma,\Omega)$ as in Lemma \ref{bar2}. %By Lemma \ref{mark4}, we create some  marking on the set of new vertices, defined by the dimensions of the  subdivided cones.  The subdivision will not affect any subdivision in $\Sing(\Sigma))$, and all the  subdivision from $\sing(\Sigma)$ will be subdivided. 
This induces, by Lemmas \ref{bar4}, \ref{mark2}, a relatively simplicial complex $(\Delta,\Omega)$, with the induced marking $V$ such that all
unchanged cones are relatively regular. 
Consequently, 
the  relative complex $(\Delta,\Omega)$ is regularly marked and can be canonically resolved by Lemma \ref{can mar2}.

%and such that its maximal unmarked subcomplex $U(B(\Sing(\Sigma,\Omega)))$ is regular (as in the absolute case). 
%Indeed, the untouched faces of $\Sing(\Sigma)$ are exactly the regular cones.

\end{proof}
\subsection{Finite group actions and obstructions to an equivariant  relative desingularization}

%Note that a semigroup of a cone $\tau$ containing no minimal vectors is   necessarily generated by  $v_i$, and thus $\tau$ is regular. 

\begin{example} \label{obstruction} Consider the monoid $$\sigma\cap N_\sigma=\{(a_1,a_2,a_3) : \sum a_i \in 2\ZZ\}\quad \cap  \quad  \QQ_{\geq 0}^3,$$
generated by $(1,1,0)$, $(1,0,1)$, $(0,1,1)$, and $(2,0,0)$, $(0,2,0)$, $(0,0,2)$ as in the Abramovich Example
\ref{Abramovich}.
Let $G=\ZZ_2$ acts  on $N_\sigma$ by permuting first two coordinates. Let 
$$\omega=\langle (2,0,0), (0,2,0) \rangle.$$
be the face defining the submonoid $\omega\cap N_\sigma$ generated by $(2,0,0),(1,1,0), (0,2,0)$.

Suppose that there is a $G$-equivariant desingularization $(\Delta,\omega)$
of $(\sigma,\omega)$. Then there is a unique  $3$-dimensional relatively regular simplicial cone $\sigma_0\in \Delta$ containing the face $\omega$. But such a cone  is $G$-stable. So it can be written as  $$\sigma_0=\langle (2,0,0), (0,2,0), w \rangle,$$, where
$w$ is $G$-invariant and has a form $$w=(a,a,2b),$$ where $a, b\in\ZZ_{>0}$.
But then  the pair $(\sigma_0,\omega)$ is not regular.

Thus  to have  an equivariant desingularization  one has to  restrict permutations of the vertices of $\Omega$. To this end we  introduce an order of the vertices of $\Omega$, and add the condition that the automorphisms should respect this order. 
\end{example}

\section{Functorial desingularization of strict toroidal and toroidal embeddings} \label{desing-toroid}
\subsection{Canonical stratification defined by an NC divisor}

Let  $E$ be a divisor that  has SNC on a toroidal embedding $(X, D)$  (as in Definition \ref{nc}).  Then $E$ induces the  divisorial stratification $S_E$.  Its  closed strata are defined by the intersections of the divisorial components of $E$. This stratification is preserved by  \'etale morphisms. This allows to extend the construction of the stratification to  NC divisors
by taking locally the images of the strata on \'etale  neighborhoods.

Let $\phi: (U,E_U)\to (X,E_X)$ be an \'etale  neighborhood with $V=\phi(U)$ for which  the inverse image $\phi^{-1}(E_X)=E_U$ of the NC divisor $E_X$ is an SNC divisor $E_U$. We define the strata of $E_X\cap V$ to be the images $\phi(s)$  of  $s\in S_{E_U}$.
This is well defined, and is independent of the choice of \'etale neighborhoods. 

 Let $\phi': (U',E'_U)\to (X,E_X)$ be another \'etale neighborhood of $x$. 
Let $x\in s$ be the image of  points $y\in U $, and $y'\in U'$, so we can write $x=\phi(y)=\phi'(y')$. Then the points $y,y'$  define a point $y''$ in the component $U''$ of the \'etale morphism of fiber product $$U\times_XU'\to X.$$ The strata $s_y$ and $s_{y'}$ through $y$, and $y'$ are the images of the stratum $s_{y''}$ on $U''$ under the relevant projection. Hence the images of $s_y$ and $s_{y'}$ and $s_{y''}$ coincide in a neighborhood of $x$ defining the same stratum. 

Moreover, the closure of the stratum is the union of strata. If $s$ is a stratum through $x$, and assume that $x\in \overline{s_1}$. Then for a point $y$ over $x$, there is a stratum $s_U$ containing $y$ and dominating $s$. Moreover, since $\phi$ is open and thus closed under generization there is a stratum $s_{U1}$ dominating  $s_1$, and such that $\overline{s_{U1}}$ contains $y$. So the stratum $s_U$ intersects $\overline{s_{U1}}$, and thus $s_U\subset \overline{s_{U1}}$. Consequently,  the image  $\phi(s_U)\subset \phi(\overline{s_{U1}})\subset\overline{s_1}$.  But $\phi(s_U)\subset s$ is open and dense. So locally we have $$s\subset \overline{\phi(s_U)}\subset \phi(\overline{s_{U1}})\subset\overline{s_1}.$$

%then $\phi^(-1)(s)$ is the union of strata $\bigcup s_{U,i}$ of the same dimension $U$, so the closure $\overline{\phi^{-1}(s)}=\bigcup s_{U,i}\subset  \phi^{-1}(\overline{s})$ is the union of strata. Let $s_0$ be the stratum through $x\in V \cap \overline{s}$ then its inverse image is the union of strata in $\phi^{-1}(\overline{s})$.

\subsection{Transforming a relative NC divisor into a relative SNC divisor}

%Let $E$ be a relative $NC$ divisor on a toroidal embedding $(X, D)$. Then it locally lifts to an SNC divisor $E_U$ on \'etale strict toroidal neighborhood $U, D_U$.
%Then one considers  the divisorial stratification $S_{E_U}$,  which descends to a uniquely defined stratification of $E$.

\begin{definition} We say that a locally closed subscheme  $C$ on a toroidal embedding $(X, D)$ is {\it relatively nonsingular} if its ideal is locally generated by a  set of free coordinates on $(X, D)$ (see Definition \ref{free}).  %We say that
%$C$ has SNC (resp NC) on 

\end{definition}

%The strata on toroidal embeddings are smooth. Strata of $NC$ divisors similarly satisfy:
\begin{lemma} Let $E$ be a relative $NC$ divisor on a toroidal embedding $(X, D)$. Then all the locally closed strata of $E$ are  relatively nonsingular. 

\end{lemma}
\begin{proof} Let $s$ be a  stratum of $E$, and $x\in s$. We can assume that $s$ is minimal and thus closed and irreducible in a certain Zariski neighborhood $V$ of point $x$. Consider an \'etale neighborhood $\pi: U\to V\subset X$ of $x$, such that $\pi^{-1}(E)$ is a relatively SNC divisor on $U$.  %It is closed by minimality, so we may assume that it is unique.  
 Let $\overline{x}\in U$ be a point over $x$. Then, by shrinking $U$, if necessary, we can assume that the inverse image $s_U:=\pi^{-1}(s)$ is irreducible, contains $s$,  and thus is a minimal stratum on $U$. 
  Consider the ideal $I_s\subset\cO_{x,X}$.
  By definition of a relative NC divisor, $s_U$ is  relatively nonsingular on $(U,\pi^{-1}(D))$ and the induced ideal
 $I_{s_U,\overline{x}}=I_s\cdot\cO_{\overline{x},U}\subset \cO_{\overline{x},U}$ is generated by  free parameters $(u_1,\ldots,u_k)$ at $\overline{x}$.
  %The completion of ${\cO_{\overline{x},U}}$ is equal to $$\widehat{\cO_{\overline{x},U}}=K(\overline{x})[[u_1,\ldots,u_k,u_{k+1},\ldots,u_r,P_\sigma]].$$
  Since  $\cO_{x,s}\to\cO_{\overline{x},\overline{s}}$ is  a \'etale  there is a natural isomorphism $$\widehat{\cO_{\overline{x},U}}=\widehat{\cO_{{x},X}}\otimes_{K(x)}K(\overline{x})$$
  
   Consider a set of  generators of $I_{s_U,\overline{x}}\subset\cO_{x,X}$, and choose a subset  $v_1,\ldots,v_k\in \cO_{x,X}$ of $k$ functions with linearly independents linear parts at $x$. Then $(v_1,\ldots,v_k)$ generates $\widehat{I_{s_U,\overline{x}}}=(u_1,\ldots,u_k)=(v_1,\ldots,v_k)$ since it generates $$(I_{s_U,\overline{x}}+m_{\overline{x},{s_U}}^2)/m_{\overline{x},{s_U}}^2=((I_{{x},{s}}+m_{{x},{s}}^2)/m_{{x},{s}}^2)\otimes_{K(x)}K(\overline{x})$$ so by Nakayama, it generates ${I}_{\overline{s}}$. Since $\cO_{X,x}\to \widehat{\cO_{X,x}}$ is faithfully flat it generates  $I_s$ as well.
\end{proof}

\begin{definition} Let $E$ be 
%an SNC divisor on 
%$a strict toroidal embedding  $(X, D)$ (resp.
an NC divisor on a toroidal embedding $(X, D)$.
We say that a relatively nonsingular locally closed subscheme  $C$ on $(X,D)$
% a toroidal embedding 
has {\it SNC} %(resp NC) 
with $E$, if there is 
%a  neighborhood $(U,D_U)$ (resp. 
an \'etale  neighborhood $(U,D_U)$ such that the center and the components of $E$ are described by a part of a free coordinate system on $(U, D_U)$.
\end{definition}
\begin{remark} Let  $E$ be an SNC divisor on $(X,D)$  and a $C$ has SNC on $(X,D)$. Then, by 
Lemma \ref{2}, the subscheme $C$, and $E$ are described by a part of a free coordinate system locally on Zariski open neighborhoods $(U, D_U)$

\end{remark}

\begin{lemma} Let $E$ be an  an NC divisor on a toroidal embedding $(X, D)$).
Let $C$ be a closed relatively nonsingular center having SNC 
%(respectively NC)
 with $E$. % and having NC with $NC$ divisor $E_1\cup E_2$ on $(X, D)$. 
Consider the blow-up $X'\to X$ of $C$. Denote by $E_C$ is the exceptional divisor and by $E'$, and $D'$  the strict transforms of $E$ and $D$. Then the divisor $E_C$ has SNC on $(X', D')$.
% and 
%the induced divisor $E'\cup E_C$ has  SNC (respectively NC) on $(X', D')$.

\end{lemma}
\begin{proof} Assume that $(X, D)$ is a strict toroidal embedding and  $E$ has SNC on $(X, D)$.
Then there is a  system  of free parameters $u_1,\ldots,u_r$   on a neighborhood $U$ of $X$ such that the center $C$ is described by the ideal $I_C=(u_1,\ldots,u_k)$, and $E$ by the equation $u_{j_1}\cdot \ldots\cdot u_{j_l}=0$. Then for  any  $x\in U$ we have $$\widehat{\cO_{{x},U}}=K({x})[[u_1,\ldots,u_r,P_\sigma]].$$

The blow-up of $C$ will transform the center $C$ into the exceptional divisor $E_C$ defined locally, without loss of generality, by
$x:=u'_1:=u_1$, with $u'_i
=u_i/u_1$, for $i=1,\ldots,k$, and  $u'_j=u_j$ for $j=k+1,\ldots,r$. The local rings and their completions will be transformed accordingly
$$\widehat{\cO_{\overline{y},U'}}=K(\overline{y})[[u'_1,\ldots,u'_r,P_\sigma]],$$
where $u'_1,\ldots,u'_r$ are free parameters on $X$.
The divisor $E'\cup E$ is, up to mutiplicities described as $\sigma^*(u_{j_1}\cdot \ldots\cdot u_{j_l})=u'_{j_1}\cdot \ldots\cdot u'_{j_l}\cdot (u'_1)^d$ so has SNC on $X$.

 The proof for NC divisor $E$ is the same, except we need to consider the effect of the blow-up on an \'etale neighborhood. Note  that $E_C$ is locally (on $X'$) defined by a parameter and thus it is SNC on $(X', D')$.
\end{proof}

\begin{proposition} \label{normal} Let $E$ be a relative $NC$ divisor on a toroidal embedding $(X, D)$ %(respectively a strict toroidal embedding $(X, D)$). 
Then there exists a canonical sequence of blow-ups of strata of $E$ transforming the divisor $E$ into  a relative SNC divisor on the resulting toroidal embedding $(X', D')$,
%(respectively a strict toroidal embedding $(X', D')$), 
where $D'$ is the closure of $D\setminus E$ on $X'$. 
The procedure is functorial for smooth 
morphisms.

\end{proposition}

\begin{proof} Consider the blow-ups of all the strata of $E$ of the minimal dimension $r$. By minimality, the strata are disjoint closed, and relatively nonsingular. 
The blow-up will create a new NC divisor of the form $E^1\cup E_1$, where $E_1$ is the exceptional divisor, and $E^1$ is the strict transform of $E$. The divisor $E_1$ is relatively SNC, and both $E^1$ and $E^1\cup E_1$ are a relatively NC divisor, with $E^1$ having no strata in dimension $\leq r$. (All such strata were blowned up). Then we blow-up the strata of $E^1$ of the minimal dimension $r+1$. They have SNC with $E_1$. We create $E^2\cup E_2$, where $E_2$ is the union of the exceptional divisors and the strict transforms of $E_1$, and $E^2$ are the strict transforms of $E^1$. As before $E_2$ is a relatively SNC, and  $E^2$ and $E^2\cup E_2$ are a relatively NC divisors, with $E^2$ having no strata  in dimension $\leq r+1$. We continue the process until we eliminate all the components in $E^k$.
Functoriality follows from the construction of the centers.

\end{proof}

\subsection{Blow-ups of toroidal valuations}
\subsubsection{ Toric valuations }
We interpret the  star subdivisions in the desingularization theorem \ref{can des} as the blow-ups at some functorial centers associated with the sets of valuations.

Each vector $v\in N$ defines a linear integral function on $M$ which
determines a discrete 
valuation ${\rm val}(v)$ on
$X_{\Sigma}$. For any regular function $f=\sum_{w\in M} a_wx^w\in K[T]$ we set
$$\val(v)(f):=\min\{(v,w)\mid a_w\neq 0\}.$$ 
Thus $N$
can be perceived as the lattice of all $T$-invariant
integral valuations of the function field of $X_{\Sigma}$.
 
An integral  vector $v\in |\Sigma|$, and the  valuation $\val(v)$ define the coherent ideal sheaves on $X_\Sigma$: $$\cI_{\val(v),a}:=\{f\in {\cO}_{X_\Sigma}\mid 
\val(v)(f)\geq a\}$$
 for all natural $a\in \NN$.
%Let $T_\sigma\subset T$ be the subtorus corresponding to thesublattice $N_\sigma:=\lin(\sigma)\cap N$ of $N$. Then bydefinition, $T_\sigma$ acts trivially on$K[O_\sigma]=K[\sigma^\perp]=K[X_\sigma]^{T_\sigma}$. Thus$T_\sigma=\{t\in T\mid tx=x,  \mbox{ for all } x\in O_\sigma\}$.

\subsubsection{Toric valuations}

Let $\Sigma$ be a fan, and $X_\Sigma$ be the corresponding toric variety.

Recall that any nonnegative function $F:|\Sigma|\to \QQ$ which is convex and piecewise linear on each cone $\sigma\in \Sigma$ defines a collection  monoids $${\bf I}_{\sigma}=\{m\in \sigma^\vee\cap M\mid m_{|\sigma}\geq F_{|\sigma}\}.$$ These monoids generate ideals $ I_{F\sigma}$ of $K[X_\sigma]$ which glue to form a coherent ideal sheaf of ideals $\cI_F$.

 Conversely, a coherent $T$-stable ideal sheaf $\cI$ is determined by the monoids 
$${\bf I}_\sigma=\{ m\in (\sigma^\vee)^\integ\cap \cI_{|X_\sigma}\}\subset \sigma^\vee$$ which define a piecewise linear function $$F_\cI=\min \{m(w): m\in {\bf I}_\sigma \}.$$
If $F$ is integral, i.e $F(\sigma^{\integ})\subset \NN\cup\{0\}$ for any $\sigma\in \Sigma$ then $F_{I_F}=F$.

\begin{lemma} \label{primo}
 Any primitive vector $v$ in $|\Sigma|$ defines a piecewise  linear function $F_{v,a}: |\Sigma|\to Q$, such that for any vector $w\in \sigma \subset |\Sigma|$, we have $$F_{v,a}(w):=\min \{L(w): L\in \sigma^\vee, \quad L(v)=a\}.$$

Moreover, 

\begin{enumerate}
\item $F_{v,a}$ is a piecewise  linear function which is  convex on each face.

\item It is integral if $a\in \NN$ is sufficiently  divisible.

\item $I_{F_{v,a}}=I_{\val(v),a}$, and for sufficiently divisible $a\in \NN$, $F_{I_{\val(v),a}}=F_{v,a}$.

\item The star subdivision at $v$ creates an exceptional $Q$-Cartier divisor $D$ on the birational modification $X_{v\cdot \Sigma}$ corresponding to $v$ such that the multiple $E=aD$ is a Cartier divisor. 
\end{enumerate}
\end{lemma}
\begin{proof} It is a well known fact. To construct $F_{v,a}$ in more explicit terms, we consider the star subdivision $\langle
v\rangle\cdot\Sigma$ of $\Sigma$, and define $F_{v,a}$ on all vertices of $1$-dimensional rays of $\langle
v\rangle\cdot\Sigma$. We put 
$F_{v,a}(u)=0$ for $u\in \Ver(\Sigma)\setminus \{v\}$, and $F_{v,a}(v)=a$. This canonically extends to
all faces of $\langle
v\rangle\cdot\Sigma$ by linearity. If $a$ is sufficiently divisible then $F_{v,a}$ is integral function.

The function $F_{v,a}$ defines monoids $$I_{\val(v),a,\sigma}=\{m\in M \mid m\geq F_{v,a}= \{m\in (\sigma^\vee)^\integ \mid m(v)\geq 0\}$$ which give rise to the sheaf $\cI_{\val(v,a)}$. 

Finally the integral function $F_{v,a}$ on the star subdivision 
$\langle
v\rangle\cdot\Sigma$, is linear on each cone. Thus it defines a Cartier divisor $E$. 
For any vertex $w$ in $\Sigma$ let $D_w$ be the corresponding divisor. Then the Cartier divisor $E$ can be written  as $$E=\sum_{w\in \Ver(\Sigma)} F_{v,a}(w)D_w=m D_v.$$
\end{proof}

\subsubsection{Toroidal valuations}
One can easily extend the result to strict toroidal embeddings 
\begin{lemma}  Let $(X, D)$ be a strict toroidal embedding, and $\Sigma$ be the associated conical complex.
Any vector $v$ in $|\Sigma|$ defines a piecewise  linear function $F_{v,a}: |\Sigma_X|\to Q$, as above which   is a convex on each face and piecewise  linear function.
It is integral if $a$ is sufficiently  divisible.
For such $a$, there is a unique coherent ideal $I_{F_{v,a}}$ and a toroidal valuation $\val(v)$ on $X$, such that 
$I_{F_{v,a}}=I_{\val(v),X,a}$.
\end{lemma}
\begin{proof} To define the valuation  $\val(v)$  on $X$ we consider the canonical birational transformation $Y\to X$ associated with the star subdivision $\langle
v\rangle\cdot\Sigma$. One can describe the valuation $\val_Y(v)$ by using \'etale charts. By the previous lemma it follows that $\val_Y(v)$  is
  the valuation of  the exceptional divisor $D$ of $Y\to X$ as in the toric case. If $a$ is sufficiently divisible, the multiple $aD$ as in   Lemma \ref{primo},  is a Cartier divisor on $Y$. It is locally a pull-back of a toric Cartier divisor. Since $Y\to X$ is birational $\val_Y(v)$ determines a unique valuation on $\val_X(v)$ on $X$, with $K(X)=K(Y)$.

\end{proof}

\subsubsection{Blow-ups at toroidal  valuations}
 Denote by $\bl_{\cJ}(X)\to X$ the
blow-up of  any coherent sheaf of
ideals ${\cJ}$.

\begin{definition} \label{de: blow} Let $X$  be a toric variety (respectively strict toroidal embedding), with the associated fan (respectively the associated complex) $\Sigma$. Let $v\in |\Sigma|$ be an integral vector. 

By the
{\it blow-up} $\bl_{\val(v)}(X)$ of $X$ at a toric
(respectively toroidal) valuation $\val(v)$ we mean the normalization of 
$$\Proj({\cO\oplus \cI}_{\val(v),1}\oplus {\cI}_{\val(v),2}\oplus\ldots).$$
\end{definition}

%\begin{lemma} For any natural $l$, $$\Proj({\cO\oplus \cI}_{\nu,1}\oplus
%{\cI}_{\nu,2}\oplus\ldots)=\Proj({\cO\oplus
%\cI}_{\nu,l}\oplus {\cI}_{\nu,2l}\oplus\ldots).$$ 
%\qed
%\end{lemma}

\noindent \begin{lemma}(5.2.8 \cite{Wlodarczyk}) \label{le: blow-up valuation} Let $X_\Sigma$ be a toric variety (resp. strict toroidal embedding)
associated with a fan (resp. with a conical complex) $\Sigma$  
and $v\in |\Sigma|$ be an integral vector. Then $\bl_{\val(v)}(X_\Sigma)$ is the toric 
variety (toroidal embedding) associated with the star subdivision $\langle
v\rangle\cdot\Sigma$ of $\Sigma$. Moreover, for any sufficiently
divisible integer $d$, $\bl_{\val(v)}(X_\Sigma)$ is the normalization of
the blow-up of  ${\cI}_{\val(v), d}$. The   valuation $\nu$ is induced by an irreducible ${\bf Q}$-Cartier
divisor on the variety $\bl_{\val(v)}(X)$.
\end{lemma} 
\begin{proof}
By \cite{KKMS}, Ch. II,  Th. 10, the normalized blow-up of ${\cI}_{\val(v), D}$ corresponds to
the minimal subdivision $\Sigma'$ of $\Sigma$ such that $F_{{\cI}_{\val(v), D}}=F_{v, d}$ is 
linear on each cone in $\Sigma'$.  
On the other hand, if $d$ is sufficiently divisible then ${\cI}_{\val(v),nd}={\cI}_{\val(v), d}^n$ for any $n\in \NN$.
See also Lemma 5.2.8 \cite{Wlodarczyk-toroidal}.

\end{proof}
With any irreducible Weil divisor $D$, which is  Cartier at its generic point $\nu=\nu_D$ we associate a valuation $\val_D$, such that for any $f$ which is regular at $\nu$, $\val_D(f)=n$, if  we have the equality of ideals $(f)=\cI_D^n$  in  the local ring $\cO_{X,\nu}$.
%%%%%%%%%

%valuation  on a  toric variety (resp. strict toroidal embedding) $X$, and   
 $$\pi: Y=\bl_\nu(X)\to X$$ be the associated blow-up with the exceptional Weil, $\QQ$-Cartier divisor $D$.

%%%%%%
\begin{corollary} \label{divisors} Let $\nu$ be  a toric (respectively toroidal) 
valuation  on a  toric variety (resp. strict toroidal embedding) $X$, and    $$\pi: Y=\bl_\nu(X)\to X$$ be the associated blow-up with the exceptional Weil, $\QQ$-Cartier divisor $D$.

Then for the ideals  $\cI_{D,n}:=\cI_{\val_{D},n}=\{f\in {\cO}_{X}\mid 
\nu_D(f)\geq a\}$ on $Y$ we have 
$$\pi_*(\cI_{D,n})=\{f\in \pi_*{\cO}_{Y}={\cO}_{X}\mid 
\nu_D(f)\geq a\}=
{\cI}_{\nu, a}.$$
Thus the   valuation $\nu$ is induced by an irreducible  exceptional Weil (${\bf Q}$-Cartier)
divisor $D$  on the variety $Y=\bl_\nu(X)$.

\end{corollary}

%\noindent \begin{lemma}  Let $(X, D)$ be a strict toroidal embedding, and $\Sigma$ be the associated conical complex. 
% Then any integral vector  $v\in |\Sigma|$ determines a locally monomial valuation $\nu:=\val(v)$, and its normalized blow-up $\bl_{\val(v)})(X)\to X$ corresponds to the star subdivision
 %$\langle
%v\rangle\cdot\Sigma$. Moreover, for any sufficiently
%divisible integer $d$, $\bl_{\val(v)}(X)$ is the normalization of
%the blow-up of  ${\cI}_{\nu, D}$.
%\end{lemma} 
%\begin{proof} Locally we  use Lemma \ref{le: blow-up valuation}. By quasicompactness of $X$ one can find a common sufficiently divisible $d$ which defines the center of the blow-up ${\cI}_{\nu, D}$

%\end{proof}

\subsubsection{Filtered centers}

The centers of the blow-ups of the valuations $\nu$ are  represented by the sets of the ideals  $\cI_{\nu,a}$. It is possible to associate with the center  one of these ideals for a sufficiently divisible $a$, but this correspondence won't be functorial. 
%To maintain the functoriality property  important for glueing properties we consider the centers of the ideals defined by the locally toric valuations.   This leads to the following definition

\begin{definition}\label{filtered} By a {\it filtered center} of a blow-up we mean a filtration $\{\cI_n\}_{n\in \NN}$ of ideals $\cI_1\supseteq \cI_2\supseteq \ldots$ on a variety $X$ defining the graded Rees algebra $${\cO\oplus \cI}_{1}\oplus {\cI}_{2}\oplus\ldots.$$ By the blow-up of the filtered center $\{\cI_n\}_{n\in \NN}$ we mean
%such we have the equality of  the blow-ups:
$$\Proj({\cO\oplus \cI}_{1}\oplus {\cI}_{2}\oplus\ldots).$$
%for sufficiently divisible $d$.
\end{definition}

\subsection{Canonical desingularization of strict toroidal embeddings}

\begin{theorem} \label{th: resolution}
 For any strict toroidal embedding $(X, D)$ over a  field $K$ there exists  a canonical  resolution of singularities i.e., a birational projective morphism $f: (Y, D_Y)\to (X, D_X)$ such that
 \begin{enumerate}
 \item $(Y, D_Y)$ is a nonsingular strict toroidal embedding.  In particular, $Y$ is nonsingular, and $D_Y$ is an SNC divisor.
 \item $f$ is an isomorphism over the open set of nonsingular points. In particular, it is an isomorphism over $U:=X\setminus D$.
 \item The inverse image $f^{-1}(\Sing(X))\subset Y\setminus U=D_Y$ of the singular locus $ \Sing(X)$ is a simple normal crossing  divisor on $Y$.
 \item $f$ is a composition of the normalized blow-ups of the locally monomial filtered centers $\{\cJ_{in}\}_{n\in \NN}$ defined by the sets of valuations.  \item 
 $f$ commutes with smooth morphisms  and field extensions, in the sense that the centers are transformed functorially, and the trivial blow-ups are omitted.
 \end{enumerate}

\end{theorem}

\begin{proof} Consider the  conical complex $\Sigma$ associated with $(X, D)$. Its combinatorial desingularization $\Delta$ from  Theorem \ref{can des} determines a canonical  desingularization $(Y, D_Y)$ of $(X, D)$, with 
$D_Y$ being SNC. Note that any toroidal divisor on a smooth strict toroidal embedding is SNC.

\end{proof}

\bigskip
%\section{Functorial resolution of toroidal embeddings}
\subsection{Canonical  desingularization of toroidal embeddings} 

\begin{theorem} \label{des toroidal} Let $(X, D)$ be a toroidal embedding over a  field of any characteristic.

There exists a canonical desingularization $(Y, D_Y)$ of $(X, D)$, that is a projective birational toroidal morphism $f: (Y, D_Y)\to (X, D)$  such that 
\begin{enumerate}
\item $Y$ is a smooth variety, and $D_{Y}$ is  an SNC -divisor (respectively NC- divisor). 

\item If $\Sing(X)$ is the set of the singular points on $X$ then $f^{-1}(\Sing(X))$ is an SNC divisor.
 
\item $f$ is an isomorphism for all points where $X$ is smooth, and $D_X$ is and an SNC divisor (respectively NC divisor).
 
 \item $f$ is a composition of the normalized blow-ups of the locally monomial filtered centers $\{\cJ_{in}\}_{n\in \NN}$ defined by the sets of valuations.

\item $f$ commutes with  smooth  morphisms.
%\item If $(X,U)$ is strict toroidal then 

\item $f$ commutes with the field extensions

%\item The exceptional locus $D_Y$ can be further modify to an  SNC divisor $D_Y'$  by a sequence of blow-ups at  smooth strata. 
\end{enumerate}

\end{theorem}
\begin{proof}

Let  $X_0$ be an \'etale cover  of $X$ consisting of strict toroidal embeddings, and let 
$\pi_i: X_1:=X_0\times_XX_0$, $i=1,2$ be two natural projections. 
Consider the functorial resolution of $X_0$ from Theorem \ref{th: resolution}.
 By functoriality, the centers lift to $\pi^*_{1}(\cI_{jn})=\pi^*_{2}(\cI_{jn})$, and thus by the flat descent the ideals $\cI_{jn}$ on $X_0$ generate the ideals $\cI_{jnX}$ on $X$. The blow-ups at $\cI_{jnX}$ determine the functorial birational modification $Y$ of $X$, which is 
smooth in \'etale topology, and hence it is smooth.  The inverse image of the $f^{-1}(\Sing(X)$ of the singular locus $ (\Sing(X)$ is SNC in \'etale topology, and thus it is NC in Zariski topology. 
By Proposition \ref{normal},  the exceptional locus $D_Y$ can be further modified functorially, if needed, to an  SNC divisor $D_Y'$   by a sequence of blow-ups at  smooth strata.

\end{proof}
\subsection{Canonical partial desingularization of toroidal and strict toroidal embeddings}

\begin{theorem} \label{des toroidal2} Let $(X, D)$
 be a toroidal embedding  over a  field of any characteristic with $U=X\setminus D$. Let  $V\subset X$ be an open  saturated toroidal 
 subset of  $X$. Denote by $D_V=V\cap D$ the induced divisor on $V$, and assume that $D_V$ has locally ordered components. \footnote{Definitions \ref{order}}
 
   There exists a canonical  resolution of singularities of $(X, D)$ except for $V$ i.e. a projective  birational toroidal morphism  $f: Y\to X$ such that
 \begin{enumerate}
 \item $f$ is an isomorphism over the open set $V$ .
 %\item The variety $(Y,V)$ is a toroidal embedding.
 %\item $f$ is toroidal  with respect to $(Y,U)\to (X,U)$
 \item The variety $(Y, D_Y)$ is a  strict toroidal embedding, where $D_Y:=\overline{D_V}$ is the closure of the divisor $D_V$ in $Y$.
 
 \item The variety $(Y, D_Y)$ is the  saturation
 of $(V, D_V)$.\footnote{Definition \ref{saturation}}
 \item The   complement   $E_{V,Y}:=Y\setminus V$  is a divisor which  has  simple normal crossings  with $D_Y$, and so does the exceptional divisor  $E_{\exc}\subset E_{V,Y}$.  
\footnote{Definition \ref{nc}} 
In particular, $(Y, D_Y)$ is a strict toroidal embedding  .
 %\item If $(V, D_V)$ is a strict toroidal embedding then  $(Y, D_Y)$ is such.
 % defined by the sets of valuations.
%\item In particular, if $V$ is smooth and $D_V$ is an SNC divisor on  $V$ then 
%$Y$ is smooth and $D_Y\cup E$ is an SNC divisor.
 
 %\item If $(X, D)$, is a strict toroidal embedding  then $(Y, D_Y)$ is also such.
 %\item $f$ is a composition of the normalization and the normalized blow-ups of the locally monomial filtered centers $\{\cJ_{in}\}_{n\in \NN}$.
 \item
 $f$ commutes with field extensions  and smooth morphisms 
  preserving the order of the components $D_V$, in the sense that the centers are transformed functorially, and the trivial blow-ups are omitted.
%\item The exceptional locus $E$ can be further modify to a relative SNC divisor $E'$ (so that $E'$ has SNC with $D_Y$) by a sequence of blow-ups at relatively smooth centers. 

 \end{enumerate}

\end{theorem}

\begin{proof} Assume first  that $(X, D)$ is a strict toroidal embedding with the associated complex $\Sigma$. %We can also assume that $(V, D_V)$ is saturated in  $(X, D_X)$, where $D_X$  is the closure of $D_X$ taking the appropriate saturation (enlarging $V$). Hence $V$ contains $U$ moreover by using charts we see that $V$ is saturated in $(X, D)$. 
Thus the open saturated subset $V$ has a form $V=X(\Omega)$, where $(\Sigma,\Omega)$ is a relative complex. Consider its canonical relative desingularization  $(\Delta,\Omega)$ from Theorem \ref{can des2}. This, by Lemma \ref{regular}, defines the associated strict toroidal variety $(Y,U)$, such that 
$(Y,\overline{D_V})$  is also a  strict toroidal embedding, and 
the complement $E_{V,Y}=Y\setminus V$ is a divisor which  has SNC with $\overline{D_V}$.

Now if $(X, D)$ is a toroidal embedding, then we repeat the reasoning from the proof of Theorem \ref{des toroidal2}. We obtain  a  toroidal embedding $(Y,\overline{D_V})$, which is
strictly toroidal in the \'etale charts. Since $(Y,\overline{D_V})$ contains an open subset $(V, D_V)$ intersecting all strata and which is strictly toroidal then
$(Y,\overline{D_V})$ is strictly toroidal by Lemma \ref{extension1}.  Consequently, $E_{V,Y}:=X\setminus V$ has NC on $(Y,\overline{D_V})$. To  obtain the condition that $E_{V,Y}$ has SNC with $\overline{D_V}$ it suffices to apply the additional blow-ups from  Proposition \ref{normal}. The latter process commutes with
smooth morphisms and field extensions. The previous steps commute 
 with field extensions  and smooth morphisms  preserving the order of the components $D_V$ 

\end{proof}

\section{Canonical desingularization of locally toric varieties} \label{stratified}

\subsection{Locally toric and locally binomial varieties}

\begin{definition} \label{locally} A variety $X$ will be  called  {\it  locally toric} (respectively {\it \'etale locally toric}) if for any point $x\in X$ there
is an \'etale morphism called a {\it chart} 
$\varphi_x:U_x\to X_{\sigma}$ 
from an open neighborhood (respectively an \'etale neighborhood)  of $x$ to a  toric variety $X_{\sigma}$. Similarly,  $X$  will be called {\it locally binomial} (respectively {\it \'etale locally binomial}) if it locally (respectively { \'etale locally }) admits an \'etale morphism to a variety $Z\subset \AA^n=\Spec(K[x_1\ldots,x_n])$ with the (reduced) ideal $\cI_Z=(F_1,\ldots,F_s)$ defined by the binomomials $$F_i=x^{\alpha_i}-x^{\beta_i},$$
with $\alpha_i,\beta_i\in \ZZ^k$.
\end{definition}

Consider the homomorphism of tori $\phi: T^k\to T^s$, where $$ T^k=\Spec(K[x_1,x_1^{-1}\ldots,x_n,x_n^{-1}])\subset \AA^n,\quad T^s=\Spec(K[x_1,x_1^{-1}\ldots,x_s,x_s^{-1}])\subset \AA^s$$ given by $$\phi(x)=(x^{\alpha_1-\beta_1},\ldots,x^{\alpha_s-\beta_s}).$$ Its kernel defines a variety $T_Z:=\ker(\phi)=T^k\cap Z$, 
which is an open subvariety of $Z$. It is also a closed subtorus of $T^k$. Thus  $Z\supset T_Z$ is a   variety with a torus action with a big open orbit on which torus $T_Z$ acts transitively. It can be  described as  $Z=\Spec(K[G_1,\ldots,G_s])$ , where $G_i=y^{\gamma_i}$ are Laurent monomials in $K[T_Z]=K[y_1,y_1^{-1},\ldots,y_r,y_r^{-1}]$. The normalization  transforms
$Z$ into a normal toric variety.

Consequently,  the normalization of \'etale locally binomial variety transforms it into a \'etale locally toric  variety.

\subsection{Stratified toric varieties}
%We shall assume for simplicity that the base field $K=\overline{K}$ is algebraically closed in all the sections
\label{sse: stratified toric}

Locally toric  varieties can be equipped with various natural stratifications making them into a more general class   of {\it stratified toroidal} varieties.
\subsubsection{Stratified toric varieties}
\begin{definition} \cite[Definition 3.1.3]{Wlodarczyk-toroidal} 
\label{sss} A {\it stratified toric variety} is a toric variety $(X,T)$ over a field $K$ with a 
$T$-stable equisingular stratification, such that for any stratum $s$ and two
geometric $\overline{K}$-points $x,x' \to s$ the completions of local rings $\widehat{{\cO}}_{X^{\overline{K}},x}$ and
$\widehat{{\cO}_{X^{\overline{K}},x'}}$ are isomorphic over $\Spec{\overline{K}}$, where $\overline{K}$ is the algebraic closure of $K$, and $$X^{\overline{K}}:=X\times_{\Spec{K}}{\Spec{\overline{K}}}$$
\end{definition}

 The important
difference between toric varieties and 
stratified toric varieties is that the latter come with 
stratifications which may be coarser than the one given by
orbits.   As a consequence, the
combinatorial object associated with a stratified toric
variety, called a {\it semifan}, consists of those faces of
the fan of the toric variety 
corresponding to strata. The faces  which do not
correspond to the strata  are
ignored (see Definition \ref{de: embedded}). 

%The equisingularity condition means that the completion of the local rings of  $K$-points in a stratum are isomorphic.

\subsubsection{Demushkin's theorem and equisingularity}

Recall that $\sing(\sigma)$ is a unique maximal irreducible singular face of $\sigma$. (Definition \ref{irreducible})

The equisingularity condition  can be well understood by applying the following criterion:

\begin{theorem} (\cite{Demushkin}, \cite[Theorem 2.5.1]{Wlodarczyk-toroidal}) \label{th: Dem} Let $\sigma$ and
$\tau$ be
two cones of maximal dimension in isomorphic lattices
$N_\sigma\simeq N_\tau$. Set  $\widehat{X}^{K}_{\sigma}:=\Spec(K[[(\sigma^\vee)^\integ]]$, $\widehat{X}^{K}_{\tau}:=\Spec(K[[(\sigma^\vee)^\integ]]$  for any field $K$. %Let $\overline{K}$ denote the algebraic closure of $K$.
Then the following conditions are equivalent: 
\begin{enumerate}

\item $\sigma \simeq \tau$.

\item $\widehat{X}^{{K}}_{\sigma} \simeq \widehat{X}^{{K}}_{\tau}$. 
\item $\widehat{X}^{\overline{K}}_{\sigma} \simeq \widehat{X}^{\overline{K}}_{\tau}$. 
\item $\sing(\sigma)\simeq \sing(\tau)$.
\end{enumerate}
\end{theorem}
\begin{remark} Theorem \ref{th: Dem} was originally proven over algebraically closed field. But we have obvious implications $(1)\Rightarrow (2)\Rightarrow (3)$.

\end{remark}

\subsubsection{Semifans}

For any subset $\Sigma_0$ of  a fan $\Sigma$  by the {\it closure of $\Sigma_0$} we shall mean the smallest subcomplex $\overline{\Sigma_0}$ of $\Sigma$ containing $\Sigma_0$.
\begin{definition} \cite[Definition 3.1.5]{Wlodarczyk-toroidal} \label{de: embedded} \enspace
An {\it embedded semifan\/} is a subset $\Omega \subset \Sigma$ of a
fan $\Sigma$ in a $N_{\bf Q}$ such that for every $\sigma
\in \Sigma$ there is a unique maximal face $\omega(\sigma) \in \Omega$ satisfying
\begin{enumerate}

\item $\omega(\sigma) \preceq \sigma$ and any other 
$\omega \in \Omega$ with $\omega \preceq \sigma$ is a face of
$\omega(\sigma)$,

\item $\sigma = \omega(\sigma) \oplus {\rm r}(\sigma)$ for some
regular cone ${\rm r}(\sigma) \in \Sigma$.
\end{enumerate}

An {\it embedded fan \/} is an   embedded semifan  $\Omega \subset \Sigma$, where $\Omega$ is a subfan of $\Sigma$. 

A {\it semifan\/} in a lattice $N$ is a set $\Omega$ of cones in $N^\QQ$
such that the set $\Sigma=\overline{\Omega}$ of all faces of the cones of $\Omega$ is a
fan in $N^\QQ$ and $\Omega \subset \Sigma$ is an embedded semifan.

By the {\it support} of the semifan $\Omega$ we mean the union of all
its faces, 
$|\Omega |=\bigcup_{\sigma\in \Omega}\sigma$.

\end{definition}
As in Definition \ref{irreducible}, we can consider
the unique decomposition $$\sigma=\sing(\sigma)\oplus\reg(\sigma)$$ in the context of fans.

\begin{example}\cite[Example 3.1.2]{Wlodarczyk-toroidal} \label{sing} (see Definition \ref{irreducible}) Consider any fan $\Sigma$, and the subset $$\sing(\Sigma):=\{\sing(\sigma)\mid \sigma\in \Sigma\}.$$
Then $\sing(\Sigma) \subset \Sigma$ is an embedded semifan.
\end{example}

%The idea of  embedded fan  is to associate the sets of cones with the orbits in strata. In fact we have
\begin{proposition}[Proposition 3.1.7]\cite{Wlodarczyk-toroidal} \label{le: semifans correspondence} \enspace
 Let $\Sigma$ be a fan with a lattice $N$, and let $X$ denote the
associated toric variety. There is a canonical bijective  correspondence
between the toric equisingular stratifications of $X$ and the embedded
 semifans $\Omega \subset \Sigma$:

\begin{enumerate}
\item If $S$ is a toric stratification of $X$,  then the
corresponding embedded semifan $\Omega \subset \Sigma$ consists of
all those cones $\omega \in \Sigma$ that describe the big orbit of
some stratum $s \in S$.  

\item If $\Omega \subset \Sigma$ is an embedded semifan, then
the strata of the associated toric stratification $S_\Omega$ of $X$ arise
from the cones of $\Omega$ via 
$$ \omega \mapsto {\rm strat}(\omega) 
:= \bigcup_{\omega(\sigma) = \omega} O_\sigma. $$
\end{enumerate}

\end{proposition} 
\begin{proof} %Sketch of the proof. 
(1) Let $S$ be  a $T$-invariant equisingular startification on $X$.
 Since the strata of $S$ are $T$-invariant and
disjoint, each
orbit $O_\tau$ belongs to a
unique stratum $s$. Let $\omega \in \Omega$ describe the
big open orbit of $s$. Write $s=\strat(\omega)$.
If  the orbit $O_\tau$ is contained in $s$, then it is contained in 
the
closure of $O_{\omega}$. Hence $\omega=\omega(\tau)$ is a face of $\tau$. 
If any other $\omega'\in \Omega$ is a face of $\tau$ with $O_\tau\subset s$ then $O_\tau$ is in the closure of
$\overline{O_{\omega'}}=\overline{\strat(\omega')}$.
Consequently, the stratum $s=\strat(\omega)$ is contained in  $\overline{\strat(\omega')}$. This implies that $\omega'$ is a face of $\omega$ and  shows Condition (1) of Definition \ref{de: embedded}. Moreover,
by Definition \ref{sss},  there exists
an isomorphism of the completions of the local
rings   
$\overline{{\cO}_{X_\tau,x}}$ and $\overline{{\cO}_{X_\omega,y}}$ of two 
points $x\in O_{\tau}$ and $
y\in O_{\omega} $.
By Theorem \ref{th: Dem}, and since 
$\sing(\tau)\supset\sing(\omega)$, we infer that
$\sing(\tau)=\sing(\omega)$. Hence
$\omega=\sing(\omega)\oplus \reg(\omega)$ and 
$$\tau=\sing(\tau)\oplus \reg(\tau)=
\omega\oplus {\rm r'}(\tau),$$
 where ${\rm r}(\tau)={\rm r}(\omega)\oplus {\rm r'}(\tau)$. This proves Condition (2) of Definition \ref{de: embedded}.
  
(2) By definition, the strata ${\rm strat}(\omega) 
$ in $S_\Omega$  are the unions of orbits that correspond to the collections of cones:
$$\{\sigma\in\Sigma\mid \omega(\sigma)=\omega)\}=
{\rm Star}(\omega,{\Sigma})\setminus \bigcup_{\omega\prec
\omega' \in \Omega} {\rm Star}(\omega',{\Sigma})$$
Under this correspondence  $$\overline{\strat(\omega)}=\overline{O_\omega}= \{O_\tau: \tau \in{\rm Star}(\omega,{\Sigma})\}   $$  In particular, 
$$\strat(\omega)=\overline{\strat(\omega)}\setminus \bigcup_{\omega \prec \omega'} \overline{O}_{\omega'}$$
 are
locally closed, and
$$\overline{\strat(\omega)}= \bigcup_{\omega \preceq \omega'} \overline{O}_{\omega'}.$$
Moreover, by definition, $\strat(\omega)$ are disjoint, and each orbit $O_\sigma$ is contained in a unique stratum $\strat(\omega)$, where $\omega=\omega(\sigma)$.
 Thus the sets $\strat(\omega)$, where $\omega\in \Omega$ form a stratification on $X$.
%Since $\tau=\omega\oplus {\rm r}(\tau)$ we have $\Gamma_\tau=\Gamma_\omega$.
By Section  \ref{orbits}, each $\strat(\omega)\subset \overline{O_\omega}$ is a toric
variety with a fan
$$\Sigma':=\{\tau/\omega \quad | \quad \omega(\tau)=\omega\}$$
%=\{\frac{\tau+\spa(\omega)}{\spa(\omega)}|\omega(\tau)=\omega\}$$
in $(N^{\bf Q})':=N^{\bf Q}/\spa(\omega)$.
Since $\sing(\tau)=\sing(\omega)$, then $\tau/\omega$ is regular and the strata $\strat(\omega)$ are smooth. Moreover, by Theorem \ref{th: Dem}, the  points in $\strat(\omega)$  have isomorphic 
local rings (and their completions). 
\end{proof}

\begin{corollary} The strata on a stratified toric variety are smooth.

\end{corollary}
%\begin{proof} Let $\tau\in \strat(\omega)$.
%Since $\tau=\omega\oplus {\rm r}(\tau)$ the quotient cone
%$\tau/\omega$ is isomorphic to
%the regular cone ${\rm r}(\tau)$ in $(N^{\bf Q})'$ (see Section \ref{orbits}).
 %Thus the stratum
%$\strat(\omega)$ is  smooth. 
%\end{proof}
%The local rings of closed points of the stratum
%$\strat(\omega)$ have the same isotropy group $\Gamma_\omega$.

\subsection{Isomorphisms of  embedded semifans over  partially ordered sets}

\begin{definition}\label{auto} By an {\it automorphism of an embedded semifan}  $\Omega \subset \Sigma$ in $N^{\bf Q}$ we mean an  automorphism  $\alpha: N^{\bf Q}\to N^{\bf Q}$ of the vector space $N^{\bf Q}$, preserving the lattice $N$,  which defines an automorphism of the fan $\Sigma$ such that $\alpha(\omega)=\omega$ for any $\omega\in \Omega$. The group of the automorphisms will be denoted by  $\Aut(\Sigma,\Omega)$.
\end{definition}

\begin{remark} The automorphisms 
in $\Aut(\Sigma,\Omega)$ of the embedded semifan $\Omega\subset \Sigma$ induce the automorphisms  of the toric variety $X_\Sigma$ preserving the strata in $S_\Omega$ (and the torus $T\subset X_\Sigma$).
\end{remark}
%\begin{remark} The automorphisms in $\Aut(\Sigma,\Omega)$ induce the automorphisms  of the toric variety $X_\Sigma$ preserving the strata in $S_\Omega$.\end{remark}

The cones in $\Omega$ are in the natural bijective correspondence with the strata of $S_\Omega$, and this bijection respects the natural partial order. The other cones of $\Sigma$ have no geometric meaning and are dependent of a particular torus action. When comparing two embedded semifans associated with the same stratification $S$, it is convenient to consider the  natural bijection $\Omega\to S_\Omega$ of the partially ordered sets. 

%In such a case we say that the the embedded semifan $\Omega\subset \Sigma$
%is {\it defined over} a partially ordered set $S$. Note that we do not need to assume that $\Omega\to S$ is bijection, and allow any maps respecting the order.
%This leads to 
%bthe following slightly more general definition.

\begin{definition}  Let $S$ be a partially ordered set.
By {\it an embedded semifan  over $S$} we mean a semifan $\Omega\subset \Sigma$ with an injective  map $j_{\Omega, S}:\Omega\to S$ of the partially ordered sets respecting the order.
\end{definition}
\begin{definition}
Let $\Omega_i \subset \Sigma_i$ in $N_i^{\bf Q}\supset N_i$, for $i=1,2$,   be two embedded semifans over $S$.
By the {\it isomorphism  $(\Sigma_1,\Omega_1)\to (\Sigma_2,\Omega_2)$  of embedded semifans    over $S$}  we mean  a vector space isomorphism  $i_{N_1,N_2}: N_1^{\bf Q}\to N_2^{\bf Q}$ preserving the lattice structures  which defines an isomorphism of the fans $j_{\Sigma_1,\Sigma_2}: \Sigma_1\to \Sigma_2$ and induces a bijection  $j_{\Omega_1,\Omega_2}:\Omega_1\to \Omega_2$ of the sets of isomorphic cones,  commuting with the maps to $S$: $$j_{\Omega_2, S}\circ j_{\Omega_1,\Omega_2}=j_{\Omega_1, S}$$ 
\end{definition}
\begin{remark} Any embedded semifan $\Omega\subset \Sigma$  is defined over the stratification $S_\Omega$ with  the natural map $j_{\Omega, S}: \Omega\to S_\Omega$ identifying the cones of $\Omega$ with the strata of $S_\Omega$. This  interpretation is useful for glueing properties when comparing different charts.

 Any embedded semifan  $\Omega\subset \Sigma$ can be considered naturally as an embedded semifan over $S:=\Omega$ with   the  identical map $\Omega\to \Omega=S$. 
 Then the isomorphisms of $\Omega\subset \Sigma$ into itself over $S$ are simply the automorphisms in $\Aut(\Sigma,\Omega)$.
\end{remark}

Analogously one can define the stratified toric varieties over $S$ and their isomorphisms.
\begin{definition}  Let $S$ be a partially ordered set.
By {\it a stratified toric  variety  over $S$} we mean a  stratified toric  variety $X$ with a stratification $T$, and an injective  map $j_{T, S}:T\to S$ of the partially ordered sets respecting the order.

%Let $(X_i,T_i)$  for $i=1,2$,   be two st over $S$

By the {\it isomorphism  $(X_1,T_1)\to (X,T_2)$  of  stratified toric varieties   over $S$}  we mean   an isomorphism  of toric varieties preserving stratifications  which defines  a bijection  $j_{T_1,T_2}:T_1\to T_2$ of the sets of strata,  commuting with the maps to $S$: $$j_{T_2, S}\circ j_{T_1,T_2}=j_{T_1, S}$$ 
\end{definition}

\subsubsection{Generalized Demushkin's theorem}
Recall that for any cone $\tau$ by $\overline{\tau}$ we mean its closure, that is, the fan consisting of the faces of $\tau$.

\begin{theorem} \cite[Theorem 4.6.1]{Wlodarczyk-toroidal} \label{th: Dem2} Let $\overline{K}$ be an algebraically closed field.
Let $S$ be a partially ordered set.

Let $\sigma$ and
$\tau$ be
two cones of maximal dimension in isomorphic lattices
$N_\sigma\simeq N_\tau$. Let $S_\sigma$, and $S_\tau$ denote the equisingular stratifications on $X_\sigma$, and $X_\tau$, and $\Omega_\sigma\subset  \overline{\sigma}$, $\Omega_\tau\subset  \overline{\tau}$ be the corresponding embedded semifans  with the natural bijections $\Omega_\sigma\to S_\sigma$ and $\Omega_\tau\to S_\tau$, commuting with given injective maps $S_\sigma\to S$, and $S_\tau\to S$. 
%Let $S_\tau,S_\sigma\to S$ be two given bijections identifying the strata. 

Then the following conditions are equivalent: 
\begin{enumerate}

\item $(\overline{\sigma},\Omega_\sigma)$ and $(\overline{\tau},\Omega_\sigma)$
 are isomorphic over $S$.
 \item $(\widehat{X}^{\overline{K}}_{\sigma},\widehat{S}^{\overline{K}}_{\sigma})$ and $(\widehat{X}^{\overline{K}}_{\tau},\widehat{S}^{\overline{K}}_{\tau})$ are isomorphic over $S$ (defining the correspondence between the strata in $\widehat{S}^{\overline{K}}_{\sigma}$ and $\widehat{S}^{\overline{K}}_{\tau}$).

\end{enumerate}
\end{theorem}

\begin{proof} We sketch   the idea of the proof. It is very similar to the proof of Theorem
\ref{th: Dem}.
For the details see
\cite{Wlodarczyk-toroidal}.
  
  The isomorphism  $(\widehat{X}_{\sigma},\widehat{S}_{\sigma}) \to (\widehat{X}_{\tau},\widehat{S}_{\tau})$ over $S$ defines two different  toric actions on the same variety $(\widehat{X}_x,\widehat{S}_x):=(\widehat{X}_{\sigma},\widehat{S}_{\sigma})$ by the relevant tori $T_\tau$ and  $T_\sigma$. The action of these tori determine uniquely (up to constants) the
semi-invariant parameters generating the dual cones $\sigma^\vee$, and $\tau^\vee$.

The problem translates into the fact that the tori $T_\tau$ and  $T_\sigma$ are maximal and conjugate in the proalgebraic group $\Aut((\widehat{X}_x,\widehat{S}_x))$  of all the automorphisms preserving the stratification, that is $T_\sigma=\phi^{-1} T_\tau \phi$, where $\phi\in \Aut((\widehat{X}_x,\widehat{S}_x))$.  (In fact there is a natural map to the linear group of the tangent space, and its kernel is a unitary proalgebraic group (see Section \ref{inverse}.)

The conjugation $\phi$ defines the desired automorphism  of  $(\widehat{X}_x,\widehat{S}_x)$ which translates into isomorphism $\psi: (\widehat{X}_{\sigma},\widehat{S}_{\sigma}) \to (\widehat{X}_{\tau},\widehat{S}_{\tau})$, respecting tori actions. Thus it  defines the isomorphism of the cones $$(\sigma^\vee)^{\integ}\simeq \psi^*((\sigma^\vee)^{\integ}=(\tau^\vee)^{\integ}$$ and their relevant faces, that is the embedded semifans $(\overline{\sigma},\Omega_\sigma)$ and $(\overline{\tau},\Omega_\sigma)$.\end{proof}

\subsection{Stratification on toric varieties by the singularity type}

Let $X_\Sigma$ be a toric variety of dimension $n$ associated with $\Sigma$.
 For any closed point $x\in X_\Sigma$ with the residue field $K(x)\supseteq K$ we  define   {\it the singularity type} $\sing(x)$  to be the isomorphism class of the cones  $(\sing(\sigma),N^{\bf Q}_{\sing(\sigma)}),$
 where $\sigma$  is a cone of dimension $n$ such that  $\widehat{X}_x=\widehat{X}^{K(x)}_\sigma$.

 By  the Demushkin Theorem \ref{th: Dem}, $\sing(x)$  is independent of the toric structure on $X=X_\Sigma$. If $x\in O_\sigma\subset X_\sigma$, then $\sing(x)$ is the isomorphism class of $\sing(\sigma)$. We shall write it as  $\sing(x)=\sing(\sigma)$.
 
   \begin{lemma} \cite[Lemma 4.1.7]{Wlodarczyk-toroidal}\label{strat1}  The canonical stratification $\Sing(X_\Sigma)$ on $X_\Sigma$ defined by the singularity type $\sing$ on a toric variety $X_\Sigma$ corresponds to the embedded semifan $\sing(\Sigma)\subset \Sigma$ from Example \ref{sing}. 

 That is,  for a stratum $s\in \Sing(X_\Sigma)$ through $x\in X_\Sigma$ we have
 $$s\cap U=\{x\in U\mid \sing(x)=\sing(y)\},$$
for some open neighborhood $U$ of $x$.
\end{lemma}
\begin{proof}  %Thenthe singularity type in each $\strat(\sigma)$ is the same, and equal to $(\sigma,N_\sigma)$. 
Consider the stratification corresponding to the embedded semifan  $\Sing(\Sigma)\subset \Sigma$.
Let $x\in s=\strat(\sigma)$ where  $\sigma$ is an irreducible face in $\Sing(\Sigma)$. 
Then  $x\in O_\tau$, with $\tau$ satisfying $\sing(\tau)=\sigma$ in the  open set $X_\tau$. We can write $$\tau=\sing(\tau)\oplus
\reg(\tau),$$ so the singularity type $\sing$ (defined up to isomorphism of cones) is the same  for all points of the stratum $\strat(\sigma)$ and equal to $\sing(\sigma)$. It  differs from the points in $X_\tau\setminus s$ which have the singularity types of the irreducible cones of smaller dimension which are the proper faces of $\sing(\sigma)$. Thus any stratum $\strat(\sigma)$, where $\sigma\in\sing(\Sigma)$ can be characterized as the set of points with locally constant singularity type equal $\sing(\sigma)$.
\end{proof}

\subsection{Toric varieties with toric divisors}
One can adapt the terminology of relative complexes to the situation of fans. (See Sections \ref{reli}, \ref{reli2}, \ref{reli4})

\begin{definition} \label{rel fan} By {\it a relative fan} we mean a pair $ (\Sigma,\Omega)$ where $\Omega\subseteq \Sigma$ are fans. A relative fan is {\it regular} if any cone $\sigma\in \Sigma$ contains a unique  maximal face $\omega\in \Omega$ such that $\sigma=\omega\oplus\tau$ for a regular cone $\tau$.
A relative fan $ (\Sigma,\Omega)$ will be called {\it saturated} if any face of $\Sigma$ with vertices (one dimensional rays) in $\Omega$, is in fact in $\Omega$. Equivalently we say that a subfan $\Omega\subset \Sigma$ is {\it saturated}.
For any $\sigma\in \Sigma$, by $\sing_\Omega(\sigma)$ we mean the smallest face  containing $\sing(\sigma)$, and the  faces of $\sigma$ which are in $\Omega$. Then $\sigma=\sing_\Omega(\sigma)\oplus \tau$, where $\tau$ is a regular cone. Moreover, $\sing_\Omega(\sigma)$ is a unique  minimal face of $\sigma$ which gives such decomposition.
\end{definition}

\begin{example} \label{embedd} Consider a relative  fan $(\Sigma,\Omega)$. Let $$\sing(\Sigma,\Omega):=\{\sing_\Omega(\sigma)\mid \sigma\in \Sigma\}.$$
Then $\sing(\Sigma,\Omega) \subset \Sigma$ is an embedded semifan. Moreover, $\sing(\Sigma,\Omega)\supseteq\Omega$.
\end{example}

For any toric variety $X_\Delta\supset T$ associated with the fan
$\Delta$ denote by $$D_\Delta:=X_\Delta\setminus T$$ the toric divisor of the complement of the big torus $T$. It is the maximal toric divisor on $X_\Delta$ (see Definition \ref{toric divisors}).
The following lemma rephrases Lemma \ref{cr} in the context of fans and toric varieties.
\begin{lemma} \label{cr2} 
 
There is a bijective correspondence between the toric divisors $D$ on a toric variety $X_\Sigma$ and the  saturated subfans $\Omega$ of $\Sigma$.
Given a toric divisor $D$  on $X_\Sigma$, 
there is a unique saturated subfan $\Omega\subset \Sigma$ which is defined by the set of vertices of $\Sigma$ corresponding to the components of $D$. Conversely, $D=\overline{D_\Omega}$, where $D_\Omega=X_\Omega\setminus T$ is a toric divisor on $X_\Omega$.
\end{lemma}
\begin{proof} The proof is identical to  the proof of Lemma \ref{cr}.

\end{proof}

 %%%%%%%%%%%%%%%%%%%

\subsubsection{Embedded fans}

There is a relation between the notions of embedded fans (Definition \ref{de: embedded}) and relative fans as both represent properties of stratifications. The following lemma is a consequence of definition.
\begin{lemma} A relative fan $ (\Sigma,\Omega)$ is an embedded fan iff it is a regular relative fan. Moreover, if $ (\Sigma,\Omega)$ is an embedded fan then $\Omega\subset \Sigma$ is a saturated subfan, and $\Reg((\Sigma,\Omega))=\Omega$.\qed

\end{lemma}

%%%%

 \begin{corollary} \label{embedded fan} Let $X_\Sigma$ be  a toric variety  associated with the fan $\Sigma$. Let
 $\Omega$ be a  saturated  subfan of  $\Sigma$, corresponding to a toric divisor  $D=\overline{D_\Omega}$.
 Then
 \begin{enumerate}
  \item $(X, D)$ is a strict toroidal embedding  if  $(\Sigma,\Omega)$ is   a regular relative fan (or embedded fan). Then  $(X, D)$ is the saturation of $(X_\Omega, D_{\Omega})\subset (X, D)$, and $E:=\overline{D_X\setminus D}=X\setminus X_\Omega$ is a relative SNC divisor on $(X, D)$.
   \item Conversely, if $(X, D)$ is a strict toroidal embedding, where $D\subset D_X$ then $D=\overline{D_\Omega},$ where $(\Sigma,\Omega)$ is regular relative fan.
 \item In general, the toroidal locus $(X, D)^{\tor}$ of $(X, D)$ is defined by the saturated  subfan $\Reg(\Sigma,\Omega)$ of $\Sigma$. It is the toroidal saturation 
of  $(X_\Omega, D_\Omega)$ in $(X, D)$.

  \end{enumerate}
 \end{corollary}
%%%%%%%%%%%%%%%%%%%%%

\begin{proof} The proof is identical to the proof of the analogous statement for complexes in  Lemma   \ref{regular}.

%Follows from Proposition \ref{le: semifans correspondence} and the definition.
 
%Moreover, $\strat(\omega)\cap X_\Sigma=O_\omega$
%So the strata in $S_\Omega$ extend the orbit strata on $X_\Omega$.

%The components of $\overline{D_\Omega}$ correspond to ${\rm Star}(\rho,{\Sigma})$, where $\rho$ are one- dimensional rays (vertices) of $\Omega$.

%We need to show that $(X_\Sigma,\overline{D_\Omega})$ is a strict toroidal embedding . But for any $x\in {\rm strat}(\omega)$ we see that $x\in O_\sigma$, where $\sigma=\omega\oplus r(\sigma)$. So $x\in X_\sigma$, where there is a smooth morphism to $X_\omega$ defined by the natural projection $\sigma=\omega\oplus r(\sigma)\to \omega$. 
%Since $\overline{\sigma} \cap \Omega$ consists of the faces of $\omega$.
%Moreover, the strata on $X_\sigma$ are defined by the
%$\strat(\omega')=\{O_{\omega'\oplus\tau}\mid \tau\preceq r(\sigma)\}$. So they are the inverse images of the orbits $O_{\omega'}$ under the projection $X_\sigma\to X_\omega$.

\end{proof}

%\begin{corollary} 

%\end{corollary}

 \subsubsection{Canonical stratification on toric varieties with toric divisors} \label{inv}
  Let $(X, D)$ be a  pair of a  toric variety $X=X_\Sigma$ and a toric divisor $D$.

  Let  $x\in X$,   $x\in O_\sigma$, with $\Omega_\sigma:=\sigma\cap \Omega$.
  We  define   {\it the singularity type} at $x$ on $(X, D)$ to be  $$\sing_D(x):=(\sing(x),n_D(x)),$$ where $n_D(x)$ is the number of the components of $D$  through $x\in X$,  and by $\Sing_D(X)$ the corresponding stratification.
  
  Recall that the divisor $D$ determines the 
  divisorial stratification $S_D$ with closed strata being the intersection components of $D$.
  %of relative complexes.

  %\begin{corollary}\label{embedded fan2}
%Let  $\Omega\subset \Sigma$ be an inclusion of fans. Consider the divisor $D_\Omega:=X_\Omega\setminus T$ on $X_\Omega$, and let $\overline{D_\Omega}$ be its closure in $X_\Sigma$. 

%Then  the saturation of $(X_\Omega, D_\Omega)$ in $({X_\Omega},\overline{D_\Omega})$ is a toric variety corresponding to the  fan $Reg_{\Omega}(\Sigma)$ which consists of all the cones $\sigma \in \Sigma$, for which
%\begin{enumerate}
%\item There exists a unique maximal face $\omega\subset \sigma$
%\item $\sigma=\omega\oplus \tau$, where $\tau$ is regular.
%\end{enumerate}
%In other words, all the faces which are relatively regular with respect to $\Omega$.
%\end{corollary}
The vertices of $\Sigma$ associated with components of $D$ determine the saturated subfan $\Omega\subset \Sigma$.   \begin{lemma} \label{strat2} Let $X_\Sigma$ be a toric variety and $D=\overline{D_\Omega}$  be a toric divisor on  $X_\Sigma$  corresponding to a saturated subfan $\Omega\subseteq \Sigma$. Then the  stratification $\Sing_\Omega(X_\Sigma)$ on $X_\Sigma$ defined by  the embedded semifan $\sing(\Sigma,\Omega)\subset \Sigma$  from  Example \ref{embedd},
 coincides with the stratification $\Sing_D(X)$,  defined by  the singularity type $\sing_D(x)$. Moreover,
 
  \begin{enumerate}
 
 \item The strata of  $\Sing_\Omega(X_\Sigma)$ on $X_\Sigma$ correspond to the cones in $\sing(\Sigma,\Omega)$. In particular, 
 $\Sing_\Omega(X_\Sigma)$ contains the strata which are extensions of the orbits in $X_\Omega$, and $\sing(\Sigma,\Omega)$ contains $\Omega$.
 % and $\Omega\subset \sing(\Sigma,\Omega)$. saturated.
\item 
%, in a sense that if vertices in $\Omega$ form a face in $\sing(\Sigma,\Omega)\subset
%\Sigma$ then this face belongs to $\Omega$. 
The stratification $S_D$ is coarser then  $\Sing_\Omega(X_\Sigma)$.
\item The saturation of $(X_\Omega, D_\Omega)$ in $(X_\Sigma,\overline{D_\Omega})$ is an open subset $X_{\Reg(\Sigma,\Omega)}$ corresponding to the subfan  $\Reg(\Sigma,\Omega)$   of $\Sigma$ containing $\Omega$. Moreover, it defines the  toroidal locus  $$(X_\Sigma, D)^{tor}=X_{\Reg(\Sigma,\Omega)}.$$

%\item  The strata in the saturation of $(X_\Omega, D_\Omega)$ correspond to $\Omega$, and thus 
 %$$\Omega=\sing(\Sigma,\Omega)\cap \Reg(\Sigma,\Omega).$$ 
 %in particular, the intersections of the strata of $D$ are the union of the strata

\item The restriction of the stratification $\Sing_\Omega(X_\Sigma)=\Sing_D(X)$ to $V=(X_\Sigma, D)^{tor}$ coincides with the  divisorial stratification $S_{D|V}$.  
\item In particular, if  $(X, D)$ is a strict toroidal embedding then the stratifications $S_D$ and $\Sing_D(X)$ coincide.

%Two different saturations in the sense of divisor and stratification????

\end{enumerate}

\end{lemma}
\begin{proof} Consider the embedded semifan $\sing(\Sigma,\Omega)\subset \Sigma$. % It defines a stratified toric variety. 
Let $\sigma\in \sing(\Sigma,\Omega)$ be a relatively irreducible cone of $\Sigma$, with its subfan $\Omega_\sigma=\Omega\cap \sigma$.
If $x\in \strat(\sigma)$, then  $x\in O_\tau$ with $\sing_\Omega(\tau)=\sigma$. This implies that
 $\sing_D(x)=(\sing(\sigma),n_D)$, where $n_D$ is the number of vertices in the subfan $\Omega_\sigma=\Omega\cap \tau$.
If $y\in X_\tau$, and $y\not\in \strat(\sigma)$ then 
$y\in O_{\tau'}$, where $\tau'\leq \tau$ with $\sing_\Omega(\tau')\prec \sigma$. By definition $\sing_\Omega(\tau')$ is the smallest face  containing $\sing(\tau')$ and $\Omega_{\tau'}$. So either $\sing(\tau')<\sing(\sigma)$ or
$\Omega_{\tau'}\subsetneq \Omega_\sigma$.
%Since $\Omega$ is a saturated subfan we conclude that
 %if  $\Omega_{\tau'}\subsetneq \Omega_\sigma$ then the number of the vertices in $\Omega_{\tau'}=n_D(y)$ is smaller than the number of vertices of the vertices in $\Omega_{\tau'}$ equal to $n_D(x)$. 
 In any case $$(\dim(\sigma_y),n_D(y))< (\dim(\sigma_x),n_D(x))$$  with respect to the lexicographic order. This shows that the stratifications $\Sing_{\Omega}(X_{\Sigma})$ and $\Sing_D(X)$ coincide.

(1) Follows from Definition.  
%By the assumption, $\Omega$  is saturated in $\Sigma$. Since  $\sing(\Sigma)$ contains $\Omega$ we conclude that $\Omega$ is saturated in $\sing(\Sigma)$. 

(2) 
By Lemma \ref{cr2}, the  closed strata of the stratification $S_D$ are the intersections of the irreducible divisorial componenents. 
These components correspond to the vertices of $\Omega$. Since $\Omega$ is saturated  their intersections correspond to
the cones in $\Omega$. Thus the cones in the subfan $\Omega$ correspond to the  strata and the closed strata in $S_D$.
% These components intersect if the corresponding rays (vertices) form 
In this correspondence a cone $\omega\in \Omega$ corresponds to the closed stratum $\overline{\strat(\omega)}=\overline{O_\omega}$.
The fact that $\sing(\Sigma)$ contains $\Omega$ means   that the $S_D$ is coarser than $\Sing_D$.

(3) Follows from Lemma \ref{embedded fan}(3).

(4) Follows from (1) and (3). In this case we also have $$\sing((\Reg(\Sigma,\Omega),\Omega)=\Omega.$$

%(5) %The  toroidal locus $(X_\Sigma, D)^{tor}$ contains $X_\Omega$ and intersects the strata  of  $S_D$. So it contains the saturation of $X_\Omega$. On the other hand,  the closures of the  strata on toroidal locus belong to $\overline{S_D}$, so they intersect $X_\Omega$.
%So by  Definition \ref{saturation},  the toroidal saturation of coincides with the $(X_\Sigma, D)^{tor}$

%$D=\overline{D_\Omega}$  since the variety $(X_\Sigma, D)$ is toroidal at the generic points of the divisorial components, an by Lemma \ref{cr2}, all the divisorial strata in of $(X_\Sigma, D)^{tor}$ intersect $X_\Omega$. It follows that the divisorial stratification is coarser. 
%Thus it suffices to observe that the invariant $\sing_D(x)$ is constant on the divisorial strata on $(X_\Sigma, D)^{tor}=X_{\Reg(\Sigma,\Omega)}$. If $x\in \strat(\sigma)$, where $\sigma\in \Omega$. Then $x\in O_\tau$ with $\sing_\Omega(\tau)=\sigma$, so  $\tau=\sigma\oplus \reg(\sigma)$, and $$\sing_D(x)=(\sing(\tau),n_D(x))=(\sing(\sigma),n_D(x)),$$ where $n_D(x)$ is the number of vertices in $\sigma=\sing_\Omega(\tau)$ so the invariant is the same as on the generic point of $O_\sigma\subset \strat(\sigma)$. 

(5) Follows from (4).
% If $(X, D)$ is a strict toroidal embedding then  $\Sigma,\Omega)$ is regular, and $\Omega\subset \Sigma$ is embedded fan.
%Consequently, the closures of the strata on $X_\Sigma$ are the closures of the strata of the  strata $\strat_{X_\Omega}(\omega)=O_\omega$ on $X_\Omega$ so they are the intersections of the components of $D$.
%are divisorial on $X_\Sigma$.

\end{proof}

\subsection{Stratified toroidal varieties}
\begin{definition}\cite[Definition 4.1.6]{Wlodarczyk-toroidal}
By a {\it stratified toroidal variety} we mean a stratified
variety $(X,S)$ such that for any $x\in s$, $s\in S$ there
is an \'etale map called a {\it chart} 
$\varphi: (U, S\cap U)\to (X_{\sigma},S_\sigma)$ 
from an open neighborhood to a stratified toric variety $(X_{\sigma},S_\sigma)$ such that all strata in $U_x$ are the preimages of
strata in $S_{\sigma}$. 
\end{definition}

\begin{remark} \label{over} The stratified toric varieties $(X_{\sigma},S_\sigma)$ in the above definition correspond to the embedded semifans $\overline{\sigma}\supseteq \Omega_\sigma$, where $\overline{\sigma}$ is the fan consisting of all the faces of $\sigma$, and $\Omega_\sigma$ describes the generic orbits of the strata in $S_\sigma$.  Moreover,
the definition implies the existence of a natural injective map from  $\Omega_\sigma \to S_\sigma\to S$, associating with faces the corresponding strata.  So the induced semifans can be considered as semifans over the stratification $S$. This observation is important and helps to understand how gluing works.

\end{remark}

Toroidal embeddings with the natural stratification are  particular
example of a stratified toroidal variety. 

\subsubsection{Canonical stratification on locally toric varieties}

Using   Demushkin's  theorem, we are in position to define a combinatorial invariant on a locally toric variety.

We associate with any closed point $x\in X$, with residue field $K(x)$  the invariant $\sing(x)$ to be the isomorphism class of the cones  $\sing(\sigma)$, where $\sigma$ are the cones  of the maximal dimension, where $$\widehat{X}_x \simeq \widehat{X}^{K(x)}_{\sigma}.$$ The invariant $\sing(x)=\sing(\sigma)$ is well defined by   Theorem \ref{th: Dem}.

\begin{lemma} \label{can0} \cite[Lemma 4.2.1]{Wlodarczyk-toroidal} Let $X$ be a locally toric variety. There is a locally closed stratification $\Sing(X)$ with smooth irreducible strata $s\in \Sing(X)$, such that for any point $y\in X$, and the stratum $s$ through $y$, there is  an open neighborhood $U$ such that $$s\cap U=\{x\in U\mid \sing(x)=\sing(y)\}.$$

Moreover, the stratification $\Sing(X)$ is locally defined via charts by the inverse images of strata on  $(X_\sigma,\Sing(X_\sigma)$ associated with the embedded semifan $(\overline{\sigma}, \sing({\sigma}))$. \end{lemma}
\begin{proof} 

Consider an \'etale chart $U\to X_\sigma$.
By Lemma \ref{strat1}, the stratification $\Sing(X_\sigma)$ on $X_\sigma$ is defined by the singularity type $\sing(x)$.  So the stratification $\Sing(X)$ is  induced locally by the stratification $\Sing(X_\sigma)$. The strata of $\Sing(X)$ are locally the inverse images of strata in $\Sing(X_\sigma)$.
\end{proof}

\subsubsection{Canonical stratification on  locally toric varieties with divisors}

\begin{definition}\label{divisor} Let $X$ be locally toric variety. 
A divisor $D$ on $X$ will be called a {\it locally toric divisor}
if there exists locally a chart $U\to X_\sigma$, where $D\cap U$ is the inverse image of a toric divisor on $X_\sigma$.

\end{definition}

The above extends to the relative case.
Let $X$ be a locally toric variety, and $D$ be a locally toric divisor on $X$. For any closed point $x\in X$
set as before  (see Section \ref{inv})$$\sing_D(x):=(\sing(x),n_D(x)).$$ 
%where $(\widehat{X}_x,\widehat{D}_x) \simeq (\widehat{X}^{K(x)}_{\sigma},\widehat{D}_\omega)$. 
%, with the residue field $K(x)$  there is a 
\begin{theorem} \label{can} Let $X$ be locally toric variety with a locally toric divisor $D$. There is a locally closed  stratification $\Sing_D(X)$ defined by the invariant $\sing_D(x)$ such that $(X,\Sing_D(X))$ is a stratified toroidal variety.
%Moreover,  for any point $y\in X$ t
%the  smooth irreducible strata $s\in \Sing_D(X)$  there is  an open neighborhood $U$ such that $$s\cap U=\{x\in U\mid \sing_D(x)=\sing_D(y)\}.$$
Moreover, the divisorial stratification  $S_D$  is coarser than $\Sing_D(X)$.
If $(X, D)$ is a strict toroidal embedding then the divisorial stratification $S_D$ coincides with the stratification $\Sing_D(X)$.

%Moreover, the stratification $\Sing_D(X)$ is locally defined via local \'etale charts $\phi: (X, D)\to  (X_\sigma, D_{\Omega_\sigma})$, where $\Omega_\sigma$ is a subfan  of $\overline{\sigma}$, by 
 %the inverse images of strata on  $(X_\sigma,\Sing_D(X_\sigma)$ associated with the embedded semifan $(\overline{\sigma}, \sing(\overline{\sigma},\Omega_\sigma))$.
\end{theorem}
\begin{proof} The proof uses Theorem \ref{strat2}. The rest is identical as  for the proof of Theorem \ref{can0}. 
\end{proof}

%\begin{remark}

%Note that the same stratification is defined by the invariant 
%$$\sing_D(x)=(\sing_x,n_D(x)),$$ where $n_D(x)$ is the number of %the components of $D$  through $x\in X$.
%\end{remark}
%\begin{proof} Consider the local chart $\phi: U\to X_\sigma$, with irreducible $\sigma$. Then $s=\phi^{-1}(O_\sigma)$.\end{proof}

%if we consider the
%stratification described in Section \ref{se: toroidal embeddings}  
%(see also Definition \ref{de: toroidal embeddings}).

\subsection{Conical semicomplexes}
\subsubsection{Semicones}
Next we generalize the notion of  cone. 
The stratified toroidal varieties are locally described  by the charts to  affine stratified toric
varieties $(X_\sigma, S_\sigma)$. The induces semifan 
consists of the cone $\sigma$ and some of their faces -those corresponding to the strata in $S_\sigma$. We will call  such semifans  semicones.
In analogy to usual cones,  we  denote the semicones by small Greek letters $\sigma, \tau$, etc.:

We also shall  consider {\it semicones  over partially ordered} set $S$, i.e., a   semifan $\sigma$ over $S$ consisting of
some  faces of $\sigma$ with an injective map of partially ordered sets $\sigma\to S$. 

Again by the previous Remark \ref{over}, the semicones arising from the charts are naturally defined over the stratification $S$.

%By an isomorphism $\phi_{\sigma,\sigma'}:\sigma\to \sigma'$???  o

\smallskip

\begin{definition}\enspace \cite[Definition 4.3]{Wlodarczyk-toroidal} \label{semii}
Let $N$ be a lattice in the vector space $N^\QQ$. A {\it semicone\/} in $N^{\QQ}$ is a semifan $\sigma$
in $N^{\QQ}$ such that the support $|\sigma|$ of $\sigma$ occurs as an
element of $\sigma$.  
\end{definition}

This implies that the semifan $\sigma$ consists of the cone $|\sigma|$ and some of their faces.  It  defines an embedded semifan 
 $(\overline{\sigma},\sigma)$ , where $\overline{\sigma}$ is the the fan consisting of the faces of the cone $|\sigma|$.
Moreover, one can write the semicone $\sigma$ as the collection of the faces $$\sigma=\{|{\tau}| \mid \tau\leq \sigma\}.$$ 

The cone $|\sigma|$ corresponds to the minimal stratum $\strat(|\sigma|)=O_{|\sigma|}$. By abuse of notation we will use the notation $O_\sigma=\strat(\sigma)$ in the future. 

\smallskip

The {\it dimension\/} of a semicone is the dimension of its
support. Moreover, for an injection $\imath \colon N^\QQ \to (N')^\QQ$ of
vector spaces, the {\it image\/} $\imath(\sigma)$ of a semicone $\sigma$ in
$N^\QQ$ is the semicone consisting of the images of all the elements of
$\sigma$.

Note that every cone becomes a semicone by replacing it with the set
of all its faces. Moreover, every semifan is a union of maximal
semicones. Generalizing this observation, we build up in the next sections the {\it semicomplexes} associated with stratified toroidal varieties 
from semicones.

We will rephrase the Demushkin Theorem \ref{th: Dem2} in this language.
\begin{corollary} \label{th: Dem3}
Let $S$ be a partially ordered set. 
Let $\sigma_1$ and
$\sigma_2$ be
two semicones over $S$  with isomorphic lattices
$N_1\simeq N_2$. Let $S_{\sigma_1}$, and $S_{\sigma_2}$ denote the equisingular stratifications on $X_{\sigma_1}$, and $X_{\sigma_2}$ with injective maps $S_{\sigma_1}\to S$, and $S_{\sigma_2}\to S$.
For $i=1,2$ denote by $(\widehat{X}^{\overline{K}}_{\sigma})=\Spec(\widehat{\cO({X}^{\overline{K}}_{\sigma_i x_i}})$, where $x_i \in O^{\overline{K}}_{\sigma_i}$ is a closed point.
 
Then the following conditions are equivalent: 
\begin{enumerate}

\item The semicones $(\sigma_1,N_{\sigma_1})$ and $(\sigma_2,N_{\sigma_2})$ are isomorphic over $S$.
\item $(\widehat{X}^{\overline{K}}_{\sigma_1,x_1},\widehat{S}^{\overline{K}}_{\sigma,x_1})$ are isomorphic $(\widehat{X}^{\overline{K}}_{\sigma_2,x_2},\widehat{S}^{\overline{K}}_{\sigma_2,x_2})$ over $S$. 
%(defining the correspondence between the strata in $\widehat{S}_{\sigma}$ and $\widehat{S}_{\tau}$). 
\end{enumerate}
\end{corollary}

\begin{proof} The theorem is a rephrasing of the Theorem \ref{th: Dem2}.
Let $\delta_1=|\sigma_1|\times \tau_1$, and $\delta_2=|\sigma_2|\times \tau_2$ be two cones of maximal  dimension, in $N_i^{\bf Q}$,  where  $\tau_i$ are regular.

Then $(\widehat{X}^{\overline{K}}_{\sigma_i,x_i},\widehat{S}^{\overline{K}}_{\sigma,x_1}) \simeq (\widehat{X}^{\overline{K}}_{\delta_i},\widehat{S}^{\overline{K}}_{\delta_i})$, and their isomorphism implies  the isomorphism of the embedded semifans 
$(\overline{\delta_1},\sigma_1)\to (\overline{\delta_2},\sigma_2)$, which in turn defines the isomorphism  of the semicones $(\sigma_1,N_{\sigma_1})\to (\sigma_2,N_{\sigma_2})$.

\end{proof}

%\subsection{The invariants of the canonical stratification revisited}

%Let $(X,S)$ be  a stratified toroidal variety. Consider a local \'etale chart
%$U\to X_\sigma$ such that each stratum on $(U,S_{|U})$  is the inverse image of a stratum in $S_\sigma$ on $X_\sigma$. In particular, let $s=s_\sigma$ be the inverse image of $\strat(\sigma)=O_\sigma$
 %This defines a natural map $\sigma\to S$, and associates with a stratum $s$ the semicone $(\sigma,N_\sigma)$ over $S$.
 
% It also defines the invariant, For any $x\in s$ we put 
 %$\sigma_x$ to be the isomorphism class of  the semicone $(\sigma,N_\sigma)$ over the stratification $S$.
 
 %in particular, we define $\sing_x$ to be the isomorphism class of 
%$\sing(\sigma)$.

\subsubsection{Conical semicomplexes}

Let $\sigma$ be as semicone over $S$.
Denote by  $\Aut(\sigma)$ the group of automorphism of the semicone $\sigma=(\sigma,N_\sigma)$ (see Definition \ref{auto}),
and let $\Aut_S(\sigma)$ the groups of automorphisms of the   semicone $\sigma$ over $S$. 
In particular, we have  $\Aut(\sigma)=\Aut_{\sigma}(\sigma)$ is  the group of automorphisms $\sigma\to \sigma$ defined over $S=\sigma$. 

\begin{definition} \cite[Definition 4.3.1]{Wlodarczyk-toroidal} \label{de: conical semicomplex} \enspace

Let $\Sigma$ be a finite collection of semicones $\sigma$ in
 $N_{\sigma}^{\bf Q}\supset N_{\sigma}$ with $\dim(\sigma) = \dim(N_{\sigma})$. Moreover, suppose that there is a partial ordering ``$\le$'' on
$\Sigma$. We associate with each $\sigma\in \Sigma$  the group of automorphisms $\Aut(\sigma)$.

We call $\Sigma$ a {\it semicomplex\/} if for any pair $\tau \le
\sigma$ in $\Sigma$ there is an associated linear injection
$\imath^{\sigma}_{\tau} \colon N^\QQ_{\tau} \to N^\QQ_{\sigma}$ such that 
$\imath^{\sigma}_{\tau}(\tau)\subset \sigma$. In particular,
$\imath^{\sigma}_{\tau}(|\tau|)\in \sigma$ is a face of the cone $|\sigma|$, 
$\imath^{\sigma}_{\tau}(N_{\tau}) \subset N_{\sigma}$ is a saturated
sublattice and
\begin{enumerate}
\item  $\imath^{\sigma}_{\tau} \circ
\imath^{\tau}_{\varrho} =
\imath^{\sigma}_{\varrho}\alpha_\rho,$ where $\alpha_\rho\in \Aut(\rho)$.
\item $\imath^{\sigma}_{\varrho}(|\varrho|) =
\imath^{\sigma}_{\tau}(|\tau|)$ implies $\varrho = \tau$, 

\item $\sigma =\bigcup_{\tau \le \sigma}
\imath^{\sigma}_{\tau}(\tau)= 
\{\imath^{\sigma}_{\tau}(|\tau|)\mid \tau \le \sigma\}$. 
\end{enumerate}

\end{definition}
\begin{remark} \cite{Wlodarczyk-toroidal} The condition (1) is equivalent to the following condition: 
$$\imath^{\sigma}_{\tau} \circ
\imath^{\tau}_{\varrho}(\varrho) =
\imath^{\sigma}_{\varrho}(\varrho).$$

Indeed there is an induced isomorphism $\overline{\imath}^{\sigma}_{\varrho}:\rho\to \imath^{\sigma}_{\varrho}(\varrho)$, with the inverse \\ $(\overline{\imath}^{\sigma}_{\varrho})^{-1}: \imath^{\sigma}_{\varrho}(\varrho)\to \rho$ which defines the automorphism $$ \alpha_\rho:=(\overline{\imath}^{\sigma}_{\varrho}
)^{-1}(\imath^{\sigma}_{\tau} \circ
\imath^{\tau}_{\varrho}).$$
The gluing of the faces and the subdivisions of the semicomplex are defined up to the automorphisms in $\Aut(\sigma)$.
\end{remark}

\begin{definition}\label{semico} We say that the semicomplex $\Sigma$ is defined over a partially ordered  set $S$ if there is  an injective  map
$\Sigma\to S$ respecting the order.
\end{definition}

\begin{remark} \label{semico3} \begin{enumerate}

 \item Any semicomplex  $\Sigma$ can be considered as a semicomplex over itself with the identical map ${\rm Id}: \Sigma\to \Sigma$. %This approach allows to think of the semicones $\sigma\in \Sigma$ as the ones defined over $\Sigma$ with the natural  injective map $j_\sigma:\sigma
%\to \Sigma$ associating with faces of $\sigma$ the relevant faces of $\Sigma$. 
%Moreover, the maps ${\imath}^{\sigma}_{\tau}$ induce the injective maps of the semicones $j^{\sigma}_{\tau}: \tau\to \sigma$ such that $j_\sigma j^{\sigma}_{\tau}=j_\tau$.
 %\item The definition of the semicomplex is equivalent to the properties \begin{enumerate}
 %\item $j_\sigma(\sigma)=\{\tau\in \Sigma\mid \tau\leq \sigma$\}
 %\item $j_\sigma j^{\sigma}_{\tau}=j_\tau$
 %\end{enumerate}
 
 \item If a semicomplex $\Sigma$ is defined over $S$ then the
semicones $\sigma\in \Sigma$ are  defined over $S$ with 
 natural injective map map $i_\sigma: \sigma\to  \Sigma\to S$.  %   Similarly $i_\sigma j^{\sigma}_{\tau}=i_\tau$.
  \item As we see later the semicomplexes  associated with toroidal stratification $S$ are naturally defined over $S$.
%\item If $\Sigma\to S$ is a bijection the properties 2a), 2b) allow to identify semicones in $\Sigma$ with the   subsets of $S$, of the form $S_a:=\{b\in S\mid b\leq a\}$ and the induced maps $j^{\sigma}_{\tau}$ between them with the natural inclusions. 

\end{enumerate}
\end{remark}

%One can rewrite the preceding Definition \ref{semico} in a simpler equivalent form
%\begin{definition} %Let $S$ be a partially ordered set.

%Aa semicomplex  is a partially ordered finite collection $\Sigma$ of semicones   such that
%\begin{enumerate}
%\item Any semicone $\sigma$ is  a semifan in
 %$N_{\sigma}^{\bf Q}\supset N_{\sigma}$ with $\dim(\sigma) = \dim(N_{\sigma})$. It has associated group $\Aut(\sigma)$.
 
 %\item There is an injective map $\Sigma\to S$  of the partially ordered sets respecting order.
%\item For any semicone $\sigma \in \Sigma$ there is an injective map $j_\sigma:\sigma
%\to \Sigma$ associating with faces of $\sigma$ the relevant faces of $\Sigma$ and commuting with a map to $S$.
%\item For for any pair $\tau \le
%\sigma$ in $\Sigma$ there is an associated linear injection
%$\imath^{\sigma}_{\tau} \colon N^\QQ_{\tau} \to N^\QQ_{\sigma}$ %commuting with maps to $S$.

%$\imath^{\sigma}_{\tau}(\tau)\subset \sigma$. In particular, $\imath^{\sigma}_{\tau}(|\tau|)\in \sigma$ is a face of the cone $|\sigma|$, 
%$\imath^{\sigma}_{\tau}(N_{\tau}) \subset N_{\sigma}$ is a saturated sublattice. \item The induced map $\imath^{\sigma}_{\tau}; |\tau|\to |\sigma|$ of cones defines the map of semicones $j^{\sigma}_{\tau}:\tau\to \sigma$ for which $\imath^{\sigma}_{\tau}(|\delta|)=|j_\sigma(|\delta|)$ for a cone $|\delta|\in \tau$ and  
%such that $j_\sigma j^{\sigma}_{\tau}=j_\tau$.
%\end{enumerate}\end{definition}

\begin{definition} \cite{Wlodarczyk-toroidal} By an  {\it isomorphism} of the semicomplexes $\Sigma\to \Sigma'$ over $S$ we mean a  bijection   of the sets $\Sigma\to \Sigma'$ commuting with the maps to $S$ and defining the isomorphisms of the corresponding semicones $\phi_{\sigma,\sigma'}:\sigma\to \sigma'$  over $S$. 
\end{definition}

\begin{lemma} If $\phi: \Sigma\to \Sigma'$ is an isomorphism over $S$ then for any $\tau\leq \sigma$ there is a unique $\tau'\leq \sigma'$. Moreover, $\phi_{\sigma,\sigma'}\imath^{\sigma}_{\tau}=\imath^{\sigma'}_{\tau'}\phi_{\tau,\tau'}\alpha_\tau$, for some $\alpha_\tau\in \Aut(\tau)$.
\end{lemma}

\begin{proof} Note that the semicones $\phi_{\sigma,\sigma'}\imath^{\sigma}_{\tau}(\tau)$ and $\imath^{\sigma'}_{\tau'}\phi_{\tau,\tau'}\alpha_\tau(\tau)$ are equal as 
both are equal to the subset $\imath^{\sigma'}_{\tau'}$ of $\sigma'$ corresponding to the same subset of $S$. So they define the same semifan and two different isomorphisms $\tau\to  \imath^{\sigma'}_{\tau'}(\tau')$ over $S$ which are different by the automorphism  of $\tau$ over $S$ so an element $\alpha_\tau\in \Aut(\tau)$.

\end{proof}

As a special case of the above notion, we recover the notion of a
(conical) complex introduced by Kempf, Knudsen, Mumford and Saint-Donat in \cite{KKMS}:

\begin{definition} \cite[Definition 4.3.2]{Wlodarczyk-toroidal} \label{de: conical complex} \enspace
A {\it conical complex\/} is a semicomplex $\Sigma$ such that every $\sigma
\in \Sigma$ is a fan consisting of all the faces of its support cone $|\sigma|$. \end{definition}

\begin{remark} If $\Sigma$ is a complex, then  the the semicones  can be identified with the cones without loss of information. The groups $\Aut(\sigma)$ are trivial since  the automorphisms preserve the vertices of  $\sigma$ and are identical on each cone $|\sigma|$. In such a case, we obtain the definition of the conical complex. (Definition \ref{complex1}).

\end{remark}

%??? By the automorphism of the semicomplex

\subsection{Associated semicomplexes}

\begin{definition} \cite[Definition 4.8.1]{Wlodarczyk-toroidal}\label{de: associated semicomplex}  
Let $(X,S)$ be a stratified toroidal
variety. 
We say that a semicomplex $\Sigma$  over $S$ is {\it associated\/} with
$(X,S)$ if there is a bijection $i: \Sigma \to S$ of ordered sets with the following
properties: 

Let $\sigma \in \Sigma$ map to
$s=\strat_X(\sigma) \in S$. Then any $x \in
s$ admits an open neighborhood
$U_\sigma \subset X$ and  $ \varphi_\sigma \colon U_\sigma \to X_{\sigma}$ of stratified varieties such that  $s
\cap U_\sigma$ equals $\varphi_\sigma^{-1}(O_{\sigma})=\varphi_\sigma^{-1}(\strat({\sigma}))$ and the intersections $s' \cap
U_\sigma$, $s' \in S$, are precisely the inverse images of  the relevant strata of
$X_{\sigma}$ (defined by the bijection $i$): $$s'\cap
U_\sigma=\varphi_\sigma^{-1}(\strat({\tau})),$$ 
where $\tau\in \sigma$ corresponds to $s'=i(\tau)$.
Recall that $O_{\sigma}\subset X_{\sigma}$ is a stratum corresponding to the maximal cone $|{\sigma}|\in \sigma$. (See the formula after Definition \ref{semii})

We call the  smooth morphisms $U_\sigma\to X_{\sigma}$ from the above definition 
 {\it charts}. A collection of charts
satisfying the conditions from the above definition is
called an {\it atlas}.

\end{definition}

\begin{lemma}\label{le: associated semicomplex} \cite[Lemma 4.8.2]{Wlodarczyk-toroidal} For any 
stratified toroidal variety  $(X,S)$ there
exists a unique (up to an isomorphism over $S$) associated semicomplex
$\Sigma$ over $S$.  %Moreover, $\tau\leq\sigma$ iff
%$\overline{\strat_X(\tau)}\supset \strat_X(\sigma)$. 
Moreover, $(X,S)$ is a toroidal embedding iff
$\Sigma$ is a complex.
\end{lemma}\begin{proof}
%{\bf Sketch of the Proof.} (For details see \cite{Wlodarczyk-toroidal}).

For any 
stratum $s$ consider the associated local \'etale chart $U\to X_\sigma$ which defines  the  corresponding {\it semicone} $\sigma=\sigma_s$ over $S$, i.e.   semifan over $S$ consisting of
the faces of $\sigma_s$ 
corresponding to the strata in ${\rm Star(s,S)}$. This correspondence defines  a natural map  $\sigma\to S$  between faces $\tau$ in  $\sigma$ the strata $s_\tau\in S$.

 The semicone is defined uniquely up to an isomorphism over $S$. If one chart associates to a stratum a semicone $\sigma^1_s$ over $S$ , and
another associates $\sigma^2_s$ over $S$, then  by Demushkin's   theorem \ref{th: Dem3}  $\sigma^1_s$, and $\sigma^2_s$, are isomorphic over $S$. 
This gives us the
correspondence between  strata and semicones and defines a bijection $\Sigma\to S$.

Let  $s\leq s'$. Consider the semicones $\sigma_{s'}$ and $\sigma_s$ defined by the charts. 
 Then there is a face $\tau_{s}$ of $\sigma_{s'}$ corresponding to the stratum $s$ with a face inclusion $i: \tau_s\to \sigma_{s'}$ over $S$. Composing it with Demushkin's  isomorphism $\alpha_s: \sigma_s\to \tau_s$ over $S$ produces a face inclusion $i_{ss'}:\sigma_s\to \sigma_{s'}$ over $S$.
 
 The conditions (1) (2) (3) are satisfied. They  exactly mean that the face inclusions are defined over $S$. (See also the equivalent conditions in Remark \ref{semico3}(2).

%The faces  and the gluing is defined up to Demuskin's automorphisms preserving semicones. So the faces of a semicomplex do not form a topological space. 

%Finally the set $S$ of strata is no longer necessary for the 
%structure of semicomplex. The bijection $\Sigma\to S$ allows to replace $s$ indexing.

 By construction and definition the
semicomplex $\Sigma$ is determined uniquely up to isomorphism over $S$.
%Its  faces are indexed by the stratification. 

For any two such semicomplexes $\Sigma$ and $\Sigma'$ defined over $S$ there is a natural bijection $\Sigma\to \Sigma'$. Moreover, the corresponding cones $\sigma_s$ and $\sigma'_s$ are defined over $S$ and  thus related by the Demushkin isomorphism over $S$.

If a
stratified toroidal variety is a toroidal embedding then
a semicone $\sigma\in \Sigma$ consists all the faces of the cone $|\sigma|$. One can   replace  semicones  with cones without loss of data. The associated semicomplex is equivalent to a usual conical complex in that case.

%. 
% There are  no nontrivial automorphisms in $\Aut(\sigma)$ so Condition (1) is equivalent to Condition (3).

%Condition (2)  in Defintion is equivalent to Condition(2). 
%Condition (3) is no longer needed as it follows from Condition (2).

Conversely, if the associated semicomplex of the stratified toroidal variety is a complex, then locally, the strata are defined by toric orbits, so are induced by the intersecting of the divisors  (codimension one strata) defining the structure of a toroidal embedding.
\end{proof}
\begin{example} \label{assoc}
Let $(X_\Sigma,S_\Omega)$ be a stratified toric variety associated with the embedded semifan $(\Sigma,\Omega)$ over $S=S_\Omega$.
Then the associated semicomplex is given by $$\Omega^{red}:=\{(\omega,N^{\bf Q}_\omega)\mid \omega\in \Omega\}$$
 with the natural face inclusions. It is a semicomplex defined over the stratification $S$
\end{example}

\subsection{Relative semicomplexes and  locally toric varieties with a divisor}

One can extend the definition of  saturated subsets and toroidal divisors on strict toroidal embeddings as in Definition \ref{toroidd} to the case of the stratified toroidal varieties.

\begin{definition}\label{divisorial strata}  Let $(X,S)$ be a stratified toroidal variety. We say that a subset $V$ of $X$ is {\it saturated} in $(X,S)$ if it is the union of strata of $S$.  A divisor  $D$ on $X$ will be called {\it toroidal} if it saturated.

The strata whose closures are defined by the intersecting  components of $D$ will be called {\it divisorial strata}.

\end{definition}

%\begin{definition} In the future  consideration we are not going to use the na
%Finally the set $S$ of strata is no longer necessary for the 
%structure of semicomplex. The bijection $\Sigma\to S$ allows to replace $s$ indexing.
%\end{definition}
\begin{lemma}\label{local divisor} Let $X$ be a locally toric variety, and $D$ be a locally toric divisor. Then $D$ is a toroidal divisor on the induced stratified toroidal variety $(X,\Sing_D(X))$.
\end{lemma}
\begin{proof} It is a rephrasing of the statement in  Theorem \ref{can} that $S_D$ is coarser than $\Sing_D(X)$.
\end{proof}

We extend Definition \ref{omega2}, and  Lemma \ref{cr}, to the situation of stratified toroidal varieties.

\begin{definition}\label{omega3} By a {\it relative semicomplex} we mean a pair $(\Sigma,\Omega)$, of a semicomplex $\Sigma$ and a subcomplex $\Omega\subset \Sigma$.
 A complex $\Omega\subset \Sigma$ will be called {\it saturated} if 
any cone of $\sigma\in \Sigma$ with vertices (one dimensional faces) in $\Omega$ is in $\Omega$.

\end{definition}

%\begin{definition} By a toroidal divisor $D$   on a stratified  toroidal variety $(X,S)$ we mean a divisor whose components are  closed strata $\overline{s}$, for some $s\in S$.???????

%\end{definition}

We extend Theorem \ref{can}:
\begin{theorem} \label{can2} Let $X$ be locally toric variety with a locally toric divisor $D$.
Let $V:=(X, D)^{tor}$ be its toroidal locus.
 Consider a locally closed toroidal stratification $\Sing_D(X)$, and let $\Sigma$ be the associated semicomplex.
 
 Then 
 \begin{enumerate}
 \item $\Sing_D(X)_{|V}$ is the divisorial stratification associated with the strict toroidal embedding $(V, D\cap V)$. It 
corresponds to a subcomplex $\Omega$ of $\Sigma$.  
% \item $\Sing_D(X)_{|V}$ which is a subcomplex of $\Sigma$ associated with the toroidal embedding $(V, D_{|V})$.
 \item $\Omega$ is a maximal subcomplex of $\Sigma$. In particular, it is saturated in $\Sigma$.
 \item The toroidal locus of $(X,D)$ is described as $$V=(X, D)^{\tor}=X(\Omega):=\bigcup_{\omega\in \Omega}\strat(\omega),$$ 
 
 \item The vertices of $\Omega$ correspond to the irreducible components of $D$.

 \item $D$ is a toroidal divisor on $(X,\Sing_D(X))$.
 
 \item $\codim(X\setminus V)\geq 2$ 
 \item  
 $D=\overline{D_{|X(\Omega)}}.$

 %intersects all the codimension one components of $D$.
 \end{enumerate}

\end{theorem}
\begin{proof}  (1) By Theorem \ref{can}, the divisorial stratification on $V$ and the stratification $\Sing_D(X)_{|V}$ defined by the singularity type coincide on a toroidal embedding $(V,D_{|V})$. They both  correspond to a subcomplex $\Omega$ of $\Sigma$.

%(2) The complex $\Omega$ correspond to equisingular strata intersecting $V$. Any face $\tau$ of $\Sigma$ which is a face of the cone $\omega\in \Omega$ corresponds to a stratum $\strat(\tau)$ which is in the closure of $\strat(\omega)$. So $\strat(\omega)$ intersects open set $V$, and $\tau\in \Omega$. So $\Omega$ is a subcomplex.

(2) Let $\Omega'\subset \Sigma$ be a subcomplex. 
Then, by the definition, $X(\Omega')$ is the union of locally closed strata which is closed under generization, so open. The stratification $\Sing_D(X)$ on $X(\Omega')$ corresponds to a complex $\Omega'$. So locally it corresponds to a toric variety  $X_\sigma\supset T$ with the orbit stratification corresponding to  a toric divisor $D_\sigma=X_\sigma\setminus T$, and thus it is toroidal.
Consequently  $X(\Omega')\subseteq V$, and $\Omega'\subseteq \Omega$. 

(3) Follows from (2).

(4) Any irreducible component of $D$ defines a  toroidal embedding around its generic point. So it intersects $V$, and defines  a vertex of $\Omega$. On the other hand, any vertex of $\Omega$ corresponds to an irreducible component of $D_{|V}$.

%Since $\Sing_D(X)_{|V}$ corresponds to $\Omega$ we see the all the vertices of $\Omega$ correspond to the irreducible components of $D$, as any such components is toroidal embedding around its generic point.
(5) Follows from the fact that $S_D(X)$ is coarser than $\Sing_D(X)$.
 %The irreducible components are of $D$ are the closed strata of the stratification $\Sing_D(X)$, and thus are the union of strata, and so are their intersections. This implies that $D$ is the union of strata. 

(6) Follows from (4).

(7) follows from (6)
%By the definition $X(\Omega)$ is the union of locally closed strata which is closed under generization, so open. The stratification $\Sing_D(X)$ on $X(\Omega)$ corresponds to a complex $\Omega$. So locally it corresponds to a toric variety  $X_\sigma\supset T$ with the orbit stratification corresponding to  a toric divisor $D_\sigma=X_\sigma\setminus T$, and thus is toroidal. Thus $V=(X, D)^{\tor}$ contains $X(\Omega), D_\Omega$. On the other hand,  $(V, D_{|V})$ corresponds to  a subcomplex $\Omega'$ of $\Sigma$ containing $\Omega$. Then the vertices of $\Omega'$ correspond to the irreducible components of $D$.

%(4) If several components of the divisor $D$ intersect at a  component with generic point $x$ , then 
%consider the local toric chart around this point.
%The situation reduces to a toric variety, and follows from Lemma \ref{strat2}(2). The generic point $x$ determines the stratum which is necessarily in $\Omega$.

%(5) 
%For any codimension one component $E$ of the divisor $D$ let  $U_E$ be the complement for other components. Then its smooth locus  $U_E^{\sm}$ is contained in $V$. But the complement of the union $\bigcup U^{\sm}_E$ is of codimension 2.

\end{proof}

\subsection{Inverse systems  of affine algebraic groups} \label{inverse}

We shall give the groups of the automorphisms of the 
completions of the local rings the structure of proalgebraic groups. 
All the proalgebraic groups here are considered over an algebraically closed field $K=\overline{K}$.

\begin{definition}\cite[Definition 4.4.1]{Wlodarczyk-toroidal}\label{de: affine proalgebraic group}    
By an {\it affine proalgebraic group} we mean an affine
group scheme that is the limit of
 an inverse system $(G_i)_{i\in
 {\bf N}}$ of affine algebraic groups and 
algebraic group homomorphisms $\varphi_{ij}:G_i\rightarrow
G_j$, for $i\geq j$.
\end{definition}

\begin{lemma} \cite[Lemma 4.4.2]{Wlodarczyk-toroidal} \label{le: epimorphisms} Consider the natural morphism
$\varphi_i:G\to G_i$. Then $H_i:=\varphi_i(G)$ is an algebraic subgroup
of $G_i$, all induced morphisms $H_j\to H_i$ for $i\leq j$
are epimorphisms and 
$G=\lim_{\leftarrow}G_i=\lim_{\leftarrow}H_i$. In
particular $K[H_i]\subset K[H_{i+1}]$ and $K[G]=\bigcup K[H_i]$.
\end{lemma}

\begin{lemma} \cite[Lemma 4.4.3]{Wlodarczyk-toroidal} \label{le: K-points}
 The set $G^K$ of $K$-rational points  of $G$  is an
abstract group which is the inverse limit
$G^K=\lim_{\leftarrow}G^K_i$ in the category of abstract groups.
\end{lemma}

By abuse of notation, we shall identify $G$ with $G^K$. 

\begin{example} \cite[Example 4.5.4]{Wlodarczyk-toroidal}\label{ex: differential} Let $\varphi_n:{\rm Aut}(\widehat{X}_x,S)\to {\rm Aut}(X^{(n)}_x,
S)$ denote the natural morphisms. For $n=1$ we get 
the differential mapping:
$$d=\varphi_1:{\rm Aut}(\widehat{X}_x, S) \longrightarrow
{\rm Aut}(X^{(1)}_x,S)\subset
{\rm Gl}({\rm Tan}_{X,x}).$$  
\end{example}

\bigskip
\bigskip
\begin{definition} \cite[Definition 4.9.1]{Wlodarczyk-toroidal} \label{de: orientation}
 We shall call a proalgebraic group
$G$ {\it connected} if
it is a connected affine scheme. 
For any proalgebraic group $G=\lim_{\leftarrow}G_i$, denote by 
$G^0$  its maximal
connected proalgebraic subgroup $G^0=\lim_{\leftarrow}G^0_i$.
\end{definition}
\subsection{Oriented semicomplexes}
%The following notion of oriented  semicomplex i

Let $\sigma$ be a semicone in $N_\sigma^{\QQ}$.
Denote by $\Aut(\sigma)$  the group of automorphisms
of the semicone $\sigma$. Consider the natural inclusion $\phi: \Aut(\sigma)\to {\rm Aut}(\widehat{X}^{\overline{K}}_{\sigma})$. Then set $$\Aut(\sigma)^0:=\phi^{-1}({\rm Aut}(\widehat{X}^{\overline{K}}_{\sigma})^0)$$
 %(The definition is equivalent to \ref{de: oriented semicomplex})

%\begin{lemma} \label{le: extensions} Let $f: (X,S)\to (Y,R)$ be  a smooth morphism of relativedimension $k$ between stratified toroidal schemes. Let $x\in X$ bea closed $K$-rational point. Let $g_1,g_2:X\to {\bf A}^k$ be any two equivariant morphisms such that $g_1(x)=g_2(x)=0$ and $f\times g_i: (X,S)\to (Y,R) \times {\bf A}^k$ are \'etale for $i=1,2$. Then $f\timesg_1$ and $f\times g_2$ determine the same orientation at$x$.  \end{lemma}
%\begin{lemma}\label{le: sections} 
%Let $\varphi_i: (X,S)\to (Y,R)$  for $i=1,2$ be two 
%smooth morphisms of stratified toroidal schemes such that
%$\varphi_1(x)=\varphi_2(x)=y\in Y$ for a $K$-rational point $x\in
%X$ and   
 %strata in $S$ are preimages of
%strata in $R$ and $\varphi^{-1}_1(y)=\varphi^{-1}_2(y)$. Assume
%there exists a smooth scheme $V$ with a  morphism $g:
%X\to V$ such that  $\varphi_i\times
%%g: (X,S)\to (Y,R) \times V$ are smooth. Define$X':=g^{-1}(g(x)), S':=\{s\cap X'\mid s\in S\}$. Then $\varphi_1$ and $\varphi_2$ determine thesame orientation iff their restrictions $\varphi'_i: (X',S')\to(Y,R)$ do. 
%\end{lemma}

%\begin{lemma}???? The group $\Aut(\sigma)$, depends only on the monoid $\sigma\cap N$.\end{lemma}

%\begin{lemma} \label{le: s2} \cite{Wlodarczyk-toroidal} For any semicomplex $\Sigma$ and any
%$\sigma\leq\tau\leq\gamma$ in $\Sigma$, there is an automorphism 
%$\alpha_\sigma\in {\rm Aut}(\sigma)$ such that 
%$\imath^\gamma_{\tau}\imath^\tau_{\sigma}=\imath^{\gamma}_\sigma\alpha_\sigma$
%\end{lemma}
\begin{definition} \cite[Definition 4.11.1]{Wlodarczyk-toroidal} By an {\it oriented semicomplex}  we mean a semicomplex  $\Sigma$ together with the associated groups ${\rm Aut}(\sigma)^0$ such that for any $\sigma\leq\tau\leq\gamma$ there is 
$\alpha_\sigma\in {\rm Aut}(\sigma)^0$ for which 
$\imath^\gamma_{\tau}\imath^\tau_{\sigma}=\imath^{\gamma}_\sigma\alpha_\sigma$.
\end{definition}

 We will not use the notion  of  oriented semicomplexes in this paper. It was introduced for \cite[Definition 4.10.1]{Wlodarczyk-toroidal} to apply  certain nonfunctorial algorithms.
Since  our algorithm of the  combinatorial desingularization in Section  is functorial for arbitrary automorphisms of semicomplexes our reasoning will not require  this notion. However we will use the subgroups 
${\rm Aut}(\sigma)^0$.

\subsection{Subdivisions of semicomplexes}
%\begin{lemma} \cite{Wlodarczyk-toroidal} Let
If $\Sigma$ is a fan and 
 $\sigma$ be its face.
Then, by definition,   for any subdivision $\Delta$  of  $\Sigma$ and
 restriction  $$\Delta{|\sigma}:=\{\delta\in
\Delta \mid \delta\subset\sigma\}$$ of $\Delta$ to the cone $\sigma$ is a subdivision of the fan
$\overline{\sigma}$.

We shall use this notation in the context of semicomplexes.
Recall that we can associate with any semicone $\sigma$ a fan $\overline{\sigma}$ consisting of all faces of the cone $|\sigma|$.

\noindent \begin{definition} \cite[Definition 4.11.4]{Wlodarczyk-toroidal} \label{de: subdivision of
an oriented semicomplex} A {\it subdivision }  of a semicomplex (respectively an oriented
semicomplex) $\Sigma$ is a collection $\Delta=\{\Delta^\sigma\mid \sigma\in \Sigma\}$ of fans 
$\Delta^\sigma$ in $N_\sigma^{\QQ}$ where $\sigma\in \Sigma$
such that

\begin{enumerate}
\item For any $\sigma\in \Sigma$, $\Delta^\sigma$ is a
subdivision of the fan $\overline{\sigma}$ which is
$\Aut(\sigma)$-invariant (resp. $\Aut(\sigma)^0$-invariant).  

\item For any $\tau\leq \sigma$,
$\Delta^{\sigma}|{|\tau|}= \imath^{\sigma}_\tau(\Delta^\tau)$.

\end{enumerate}
\end{definition}

\begin{remark}\begin{enumerate} 
\item By abuse of terminology, we shall understand by
a subdivision of a semicone $\sigma$ simply a subdivision of the relevant
fan $\overline{\sigma}$.

\item By definition,  the vectors in faces
$\sigma$ of a semicomplex (respectively an oriented semicomplex $\Sigma$) are defined up
to automorphisms from $\Aut(\sigma)$ (respectively from $\Aut(\sigma)^0$).  Consequently, the
faces of subdivisions $\Delta^\sigma$  do not have any geometric  meaning as they are defined up
to automorphisms from $\Aut(\sigma)^0$. However, subdivisions $\Delta^\sigma$ describe the birational modifications locally.
% and in
%general do not give a structure of semicomplex. Only
%invariant faces of $\Delta^\sigma$ may define  a
%semicomplex. 

\item The condition that $\Delta^\sigma$ are
$\Aut(\sigma)$-invariant  
subdivisions is replaced for {\it canonical} subdivisions with a  somewhat stronger condition
of similar nature which says
that the induced morphism
$\widehat{X}_{\Delta^\sigma}:={X}_{\Delta^\sigma}\times_{{X}_{\sigma}}
\widehat{X}_{\sigma}\to \widehat{X}_{\sigma}$
is $\Aut(\widehat{X}_{\sigma})$-equivariant.
%The latter condition can be translated into the 
%condition that all "new" rays of the subdivision are in the {\it stable support} of $\Sigma$.  
 
\end{enumerate}
\end{remark}

\subsection{Toroidal modifications}
We will ignore the notion of orientation in the sequel, which makes the reasoning easier for most of the part.

\begin{definition} \cite[Definition 4.12.1]{Wlodarczyk-toroidal}\label{de: toroidal modification}  Let $(X, S)$ be 
a stratified toroidal variety. We say that $Y$ is a {\it toroidal
modification} of $(X, S)$  if

\begin{enumerate} 

\item There is given a proper morphism $f:Y\to X$ such that 
for any $x\in s=\strat_X(\sigma)$ there exists a chart $x\in U_\sigma\to
X_{\sigma}$, a subdivision $\Delta^\sigma$ of $\sigma$,
and a fiber square  
\[\begin{array}{rcccccccc}
&& & U_\sigma & \buildrel \varphi_\sigma \over\longrightarrow & X_{\sigma}&&&\\

&&&\uparrow {\scriptstyle f} & & \uparrow  &&&\\

U_\sigma \times_{X_{\sigma}}
X_{\Delta^\sigma} &&\simeq& f^{-1}(U_\sigma)
 &\buildrel \varphi^f_\sigma \over \longrightarrow &
X_{\Delta^\sigma} &&& \\

\end{array}\] 

\item (Hironaka's condition)
For any geometric  point $ {\overline{x}}: \Spec({\overline{K}})\to s\subset X$ in a stratum $s$,    every
automorphism  
$\alpha$ of $\widehat{X}^{\overline{K}}_{\overline{x}}$ preserving strata  
 can be lifted
to an automorphism $\alpha'$ of $Y\times_X\widehat{X}^{\overline{K}}_{\overline{x}}$.  Here $$\widehat{X}^{\overline{K}}_{\overline{x}}= (X\times_{\Spec(K)}(\Spec{\overline{K}}))_{\overline{x}}.$$
\end{enumerate}
\end{definition}

\subsection{Canonical subdivisions of semicomplexes}

We will rewrite  the Hironaka condition in the above definition in a more convenient form
\begin{lemma} \cite[Lemma 4.13.1]{Wlodarczyk-toroidal} \label{le: tilde} Let $(X,S)$ be a  
stratified toroidal
variety of dimension $n$ with associated
 semicomplex $\Sigma$ and let $f:Y\to (X,S) $ be a toroidal modification.
Let $ {\overline{x}}: \Spec({\overline{K}})\to \{x\}\subset X$ be a
geometric  point in the stratum $\strat_X(\sigma)\in S$, $\varphi_\sigma: U\to
X_{\sigma}$ a chart of a neighborhood $U$ of $x$,
and $\Delta^\sigma$ a subdivision of $\sigma$ for which
there is a fiber square 

\[\begin{array}{rcccccc}
&& & U & \buildrel \varphi_\sigma \over \longrightarrow &X_{\sigma} &\\

&&&\uparrow {\scriptstyle f} &&\uparrow&\\

&&&f^{-1}(U)

 & \buildrel \varphi^f_\sigma \over \longrightarrow &X_{\Delta^\sigma}& \\

\end{array}\]
\noindent where the horizontal morphisms are smooth.
Set $$\reg(\sigma):=\langle
e_1,\ldots,e_{n-\dim(N_\sigma)}\rangle=\langle
e_1,\ldots,e_{\dim(\strat_X(\sigma))}\rangle.$$

Then 
\begin{enumerate}
\item
there is a fiber square of \'etale extensions
\[\begin{array}{rcccccc}
&& & U & \buildrel \widetilde{\varphi}_\sigma \over \longrightarrow &X_{\widetilde{\sigma}} &\\

&&&\uparrow {\scriptstyle f} &&\uparrow&\\

&&&f^{-1}(U)

&\buildrel \widetilde{\varphi}^f_\sigma \over \longrightarrow 
&X_{\widetilde{\Delta}^\sigma}& \\

\end{array}\]
\noindent where the horizontal morphisms are \'etale and where
$$\widetilde{\sigma} := \overline{\sigma}\times
\reg(\sigma) \mbox{ and }  \widetilde{\Delta}^\sigma:=\Delta^\sigma\times \reg(\sigma).
$$

\item $X_{\widetilde{\sigma}}$ is a
stratified toric variety with the strata
described by the embedded semifan
$\sigma\subset \widetilde{\sigma}$. Moreover, the strata on
$U$ are exactly the inverse images of strata of $X_{\widetilde{\sigma}}$.

\item There is a fiber square of isomorphisms
\[\begin{array}{rcccccc}
&& & \widehat{X}^{\overline{K}}_x & \buildrel \widehat{\varphi}_\sigma \over \simeq &{\widetilde{X}_{\sigma}} &\\

&&&\uparrow {\scriptstyle \widehat{f}_x} &&\uparrow&\\

&&&Y\times_X \widehat{X}^{\overline{K}}_x

 & \buildrel \widehat{\varphi}^{f \overline{K}}_\sigma \over \simeq &{\widetilde{X}_{\Delta^\sigma}} & \\

\end{array}\]
where

 $$\widetilde{X}_{\sigma}:=
\widehat{X}^{\overline{K}}_{\widetilde{\sigma}} \mbox{ and } 
 \widetilde{X}_{\Delta^\sigma}:=
X_{\widetilde{\Delta}^\sigma}\times_{{X}_{
\widetilde{\sigma}}}
\widetilde{X}_{\sigma}=X^{\overline{K}}_{\widetilde{\Delta}^\sigma}\times_{{X}^{\overline{K}}_{
\widetilde{\sigma}}}
\widetilde{X}_{\sigma}..$$
\item $\widetilde{X}_{\sigma} $ is a
stratified toroidal scheme with the strata 
described by the embedded semifan
$\sigma\subset \widetilde{\sigma}$. The isomorphism
$\widehat{\varphi}_\sigma$ preserves strata.
\item
The morphism $\widehat{f}_x:Y\times_X\widehat{X}^{\overline{K}}_{\overline{x}}\to
\widehat{X}^{\overline{K}}_{\overline{x}}$ is
 ${\rm Aut}(\widehat{X}^{\overline{K}}_{\overline{x}},S)$-equivariant 
.

\item
The morphism $\widetilde{X}_{\Delta^\sigma}\to
\widetilde{X}_{\sigma} $  is 
${\rm Aut}(\widetilde{X}_{\sigma})$-equivariant.
 \qed

\end{enumerate}
\end{lemma}
\begin{proof} The Lemma is a reinterpretation of the Hironaka condition and follows  directly from Definition  \ref{de: toroidal modification}. For more details see \cite{Wlodarczyk-toroidal}.

\end{proof}

We shall assign to the faces of an 
semicomplex $\Sigma$ the collection of
connected proalgebraic groups over ${\overline{K}}$:
$$G_{\sigma}:=
{\rm
Aut}(\widetilde{X}_{\sigma}), \quad 
G_{\sigma}^0:=
{\rm
Aut}(\widetilde{X}_{\sigma})^0.$$

\begin{definition}\cite[Definition 4.13.3]{Wlodarczyk-toroidal} \label{de: canonical} A 
subdivision $\Delta=\{\Delta^\sigma\mid\sigma\in\Sigma\}$ of a  semicomplex  $\Sigma$ is
called {\it canonical}  if
 for any $\sigma \in \Sigma$, 
$G_{\sigma}$  acts
on $\widetilde{X}_{\Delta^\sigma}$ (as an abstract group) and the morphism
$\widetilde{X}_{\Delta^\sigma}\to
\widetilde{X}_{\sigma} $ is $G_{\sigma}$-equivariant.
\end{definition}

\begin{lemma}(\cite{Demushkin}, \cite[Lemma 7.3.2]{Wlodarczyk-toroidal})\label{le: surjection}  
Let $\sigma$ be a semicone. 
Then ${\rm Aut}(\widetilde{X}_{\sigma})^0\subset 
{\rm Aut}(\widetilde{X}_{\sigma})$ is a normal subgroup and
there is a natural surjection 
${\rm Aut}(\sigma) \rightarrow 
{\rm Aut}(\widetilde{X}_{\sigma})/{\rm Aut}(\widetilde{X}_{\sigma})^0.$
\end{lemma}

 \begin{theorem}(\cite[Theorem 4.14.1]{Wlodarczyk-toroidal})\label{th: modifications} Let  $(X,S)$ be
   stratified toroidal variety 
with the associated   semicomplex    
$\Sigma$. There exists a bijective  correspondence
between the toroidal modifications $Y$ of  $(X,S)$ 
and 
the canonical subdivisions $\Delta$ of $\Sigma$. 
\begin{enumerate}
\item

If $\Delta$ is a canonical subdivision of $\Sigma$ then
the
toroidal modification associated with it is defined locally by 
\[\begin{array}{rcccccccc}
&& & U_\sigma &  \rightarrow & X_{\sigma}&&&\\

&&&\uparrow {\scriptstyle f} & & \uparrow  &&&\\

U_\sigma \times_{X_{\sigma}}
X_{\Delta^\sigma} &&\simeq& f^{-1}(U_\sigma)
 & \rightarrow &
X_{\Delta^\sigma} &&& \\

\end{array}\] 

\item
If $Y^1\to X$,
$Y^2\to X$ are toroidal modifications associated with
canonical subdivisions $\Delta_{1}$ and $\Delta_{2}$ of $\Sigma$ then 
the natural birational map $Y^1\to Y^2$ is a morphism iff
$\Delta_{1}$ is a subdivision of $\Delta_{2}$.
\end{enumerate}
\end{theorem}
 \begin{proof}
 Sketch of the proof. 
 Let $\Delta$ be a canonical subdivision of $\Sigma$.
 The variety $Y$ is obtained by gluing the pieces $V=U\times_{X_\sigma}X_{\Delta^\sigma}$, defined by the toric charts $\phi: U\to X_\sigma$  which are birational to $U\subset X$. We need to show that the gluing is independent of the charts. 
 Suppose that there are two different charts $\phi_1$, $\phi_2$ inducing birational varieties $V_1$ and $V_2$ over $X$ which do not glue over
 a certain point $x\in s$, where $s$ is a stratum corresponding to the cone $\sigma$. If we consider the graph $V$ of $V_1\dashrightarrow V_2$, then at least one of the birational morphisms $V\to V_1$ or $V\to V_2$ is not an isomorphism over $x$, so contracts a curve.
 
 Observe that by shrinking and restricting charts  we can reduce the situation to two smooth charts of the form $\phi_1,\phi_2: U\to X_\sigma$. Then we can further assume that the charts are \'etale replacing $\phi_i$ with $\widetilde{\phi}_i:U\to X_\sigma=X_{\widetilde{\sigma}}$ with both charts taking $x$ to the closed orbit $O_{\widetilde{\sigma}}$. 
 
 Passing to the algebraic  closure $\overline{K}$, and 
 to the local rings at a geometric point $\overline{x}$ over $x$ we see that the spaces $V^{\overline{K}}_i:=U^{\overline{K}}\times_{X_{\widetilde{\sigma}}^{\overline{K}}}X^{\overline{K}}_{\Delta_{\widetilde{\sigma}}}$  does not glue over $\overline{x}$.
 % since they are not isomorphic over $\Spec(\cO_{X,x})$.
 In other words, the induced spaces  $\widehat{V}_i^{\overline{K}}:=\widehat{X}^{\overline{K}}_{\overline{x}}\times_ 
{\widetilde{X}_{\sigma}^{\overline{K}}}\widetilde{X}^{\overline{K}}_{\Delta^\sigma}$ are not isomorphic, 
for two different induced isomorphisms $\widehat{\phi}_{i{\overline{x}}}: \widehat{X}^{\overline{K}}_{\overline{x}}\to 
\widetilde{X}_{\sigma}$ since the map $\widehat{V}_1^{\overline{K}}\dashrightarrow \widehat{V}_1^{\overline{K}}$ is induced  from the map between $V_i\times_X\Spec(\cO_{X,x})$ by the  faithful flat  change of base  $\Spec(\widehat{\cO_{X^{\overline{K}},\overline{x}}})\to  \Spec(\cO_{X,x})$.
 %By the faithful flat descent $U_\sigma\times_{X_\sigma}X_{\Delta^\sigma}$ are also isomorphic over $x$.

On the other hand, by the assumption,  the isomorphisms $\widehat{\phi}_{i{\overline{x}}}$ differ by an  automorphism in ${\rm Aut}(\widetilde{X}_{\sigma})$, and this automorphism lifts to  an automorphism of $\widetilde{X}_{\Delta^\sigma}$ inducing  the isomorphism between $\widehat{V}_i^{\overline{K}}=U_\sigma^{\overline{K}}\times_{X_{\widetilde{\sigma}}^{\overline{K}}}X^{\overline{K}}_{\Delta_{\widetilde{\sigma}}}$,
 which is a contradiction.

The converse follows from the Definition and Lemma \ref{le: tilde}.
By definition, for any $\sigma$ there is a subdivision $\Delta^\sigma$ of $\sigma$ such that $\widetilde{X}_{\Delta^\sigma}\to
\widetilde{X}_{\sigma} $ is $G_{\sigma}$-equivariant.

One needs to show that the subdivisions $\Delta^\sigma$ of $\sigma$ are defined uniquely. To this end we use  diagram (3) from Lemma \ref{le: tilde}. 
The isomorphisms $\widehat{\phi}_\sigma$ are differ by element of $G_\sigma$. On the other hand, the induced morphisms $\widetilde{X}_{\Delta^\sigma}\to \widetilde{X}_{\sigma}$ are $G_\sigma$-equivariant. Thus  we conclude that there is a   isomorphism $\widetilde{X}_{\Delta^\sigma}\to \widetilde{X}_{(\Delta^\sigma)'}$ over $\widetilde{X}_{\sigma}$. It is a torus 
equivariant and  takes affine torus subschemes $\widetilde{X}_\delta\subset \widetilde{X}_{\Delta^\sigma}$ to $\widetilde{X}_{\delta'}\subset \widetilde{X}_{(\Delta^\sigma})'$, and defines an isomorphism of the cones of semminvariant functions $(\delta^\vee)^\integ\to ((\delta')^\vee)^\integ$ and $\delta\simeq\delta'$.

(2) This part follows from the analogous properties of toric varieties. We can assume that both morphisms are locally described in the same \'etale chart by the toric morphisms.
For the details, see \cite{Wlodarczyk-toroidal}.

 %\widetilde{X}_{\Delta^\sigma}\to\widetilde{X}_{\sigma}
 
 \end{proof}

\subsubsection{Canonical birational modifications  of strict toroidal embeddings} 
  The   theorem below shows that the  canonical birational modifications of  strict toroidal embeddings  can be considered as a particular case of the canonical morphisms of stratified toroidal varieties (see also Theorem \ref{sub}). In particular,  we have
  
  \begin{lemma} \label{le: tembeddings} \cite[Lemma 6.3.1(2)]{Wlodarczyk-toroidal}. Let  $(X,S)$ be a  strict toroidal
embedding with the associated complex $\Sigma$. Then 
for any  subdivision $\Delta^\sigma$ of $\sigma$  the induced morphism $\widetilde{X}_{\Delta^\sigma}\to \widetilde{X}_{\sigma}$
is $G_{\sigma}$-equivariant. Consequently, all subdivisions of $\Sigma$ are canonical.  %Moreover, any morphism  ${X}_{\Delta^\sigma}\to X_\sigma$ associated with any subdivision $\Delta^\sigma$ of $\sigma$ satisfies the Hironaka condition.

\end{lemma}

\begin{proof} Let  $\delta\in \Delta^\sigma$.  Since ${X}_{\sigma}$ is a toroidal embedding (and $\widetilde{X}_{\sigma}$ is the completion of its local ring)  each automorphism $g$ from $G_\sigma$  
preserves the toric irreducible divisors on $\widetilde{X}_{\sigma}$ defined by the rays of $\sigma$, hence it 
multiplies the generating monomials in $(\sigma^\vee)^\integ$ by invertible
functions. 
This implies that the action  $g$ on $\widetilde{X}_{\sigma}$
lifts to an automorphism $g'$ of
$\widetilde{X}_{\delta}=\widetilde{X}_{\sigma}
\times_{X_{\widetilde{\sigma}}}
{X}_{\widetilde{\delta}}$ which also multiplies monomials
by suitable invertible functions. %As $\cO(\widetilde{X}_{\delta})$ is a 
%$\cO(\widetilde{X}_{\sigma})$ module with generators in $(\delta^\vee)^\integ $ which are the quotient of monomials in  $(\sigma^\vee)^\integ$. 
Therefore $g$ lifts to $\widetilde{X}_{\delta}$ and to
to the scheme $\widetilde{X}_{\Delta^\sigma}=\bigcup_{\delta\in \Delta^\sigma} \widetilde{X}_{\delta} $. 
%The proof of "moreover part" is identical. 

\end{proof}

\subsection{Use of minimal vectors} 
 The following observations are critical. They allow
 to run certain  desingularization combinatorial algorithms on  stratified toroidal varieties.
 
 By abuse of terminology a vector $v\in \sigma$ will be
called 
 {\it
$G_{\sigma}$-invariant} (respectively {\it
$G^0_{\sigma}$-invariant}) if the corresponding valuation  
$\val(v)$ on  on $\widetilde{X}_{\sigma}$  is  $G^0_{\sigma}$-invariant (respectively {\it
$G^0_{\sigma}$-invariant}).
 
 There are not too many $G_{\sigma}$-invariant vectors, but  quite  a few $G_{\sigma}^0$-invariant ones.
 % to run the desingularization on the oriented semicomplexes or functorial desingularization on (non-oriented) semicomplexes.

 \begin{lemma} \cite[Lemma 5.3.15]{Wlodarczyk-toroidal} \label{minimal vectors} 
 Let $\sigma$ be a
semicone and
 $\widetilde{X}_{\Delta^\sigma}\to
\widetilde{X}_{\sigma}$ be a $G^0_{\sigma}$-equivariant
birational morphism induced by a toric morphism ${X}_{\Delta^\sigma}\to
{X}_{\sigma}$ associated with  subdivision $\Delta^\sigma$ of $\sigma$.
 \begin{enumerate}
 \item  \cite[Lemma 5.3.15(3,4)]{Wlodarczyk-toroidal} Let $\delta$  be an irreducible  
 face of  $\Delta^\sigma$.
Then   all minimal internal points of
 $\delta$ are $G^0_{\sigma}$-invariant.

\item \cite[Lemma 5.3.15(5)]{Wlodarczyk-toroidal} If $v$ is a vector in the ray (one dimensional face) of  the semicone $\sigma$ then $v$ is $G^0_{\sigma}$-invariant.
\item  \cite[Lemma 6.2.1(1)]{Wlodarczyk-toroidal}The set of the $G^0_{\sigma}$-invariant vectors in $\sigma$ is convex. %in particular,  the sum of minimal internal vectors of any $\delta\in\Delta^\sigma$ is  $G^0_{\sigma}$-invariant.
\end{enumerate}
\end{lemma}

%In our case we shall require that the procedure is $G_{\sigma}$-equiavariant. So we need 

\begin{corollary}  \label{minimal vectors1} With the above notation

\begin{enumerate}
\item Let $\delta\in\Delta^\sigma$ be an  irreducible face.  Then  the sum of the minimal internal vectors of  $\delta$  are  $G^0_{\sigma}$-invariant. In particular, the canonical barycenter of $\delta$ is $G^0_{\sigma}$-invariant.

\item Let $\delta\in\Delta^\sigma$. Then  the
 minimal vectors of $\delta$ are $G^0_{\sigma}$-invariant. 
\end{enumerate}
\end{corollary}

\begin{proof} A minimal vector of $\delta$ is in the relative interior of an irreducible face $\delta'$ of $\delta$. It is a minimal internal vector of $\delta'$ and, hence it is $G^0_{\sigma}$-invariant, by Lemma \ref{minimal vectors}.
	
\end{proof}

\begin{corollary}  \label{minimal vectors2} Let $\sigma$ be a
semicone containing a complex $\Omega_\sigma\subset \sigma$ and
 $\widetilde{X}_{\Delta^\sigma}\to
\widetilde{X}_{\sigma}$ be a $G^0_{\sigma}$-equivariant
birational morphism induced by a toric morphism ${X}_{\Delta^\sigma}\to
{X}_{\sigma}$ associated with  subdivision $(\Delta^\sigma,\Omega_\sigma)$ of $(\overline{\sigma},\Omega_\sigma)$. 

\begin{enumerate}

\item All the integral vectors in $|\Omega_\sigma|$ are $G^0_{\sigma}$-invariant.

\item Let $\delta\in\Delta^\sigma$ be a relatively irreducible face of $(\Delta^\sigma,\Omega_\sigma)$.  Then the canonical barycenter of $(\delta,\Omega_\delta)$ is $G^0_{\sigma}$-invariant.

\item Let $(\delta,\omega)\in (\Delta^\sigma,\Omega_\sigma)$ be a  simplicial pair
then  the the minimal vectors of $(\delta,\omega)$ are $G^0_{\sigma}$-invariant 
\end{enumerate}
\end{corollary}

\begin{proof} 
(1) The vertices of $\Omega_\sigma$ are $G^0_{\sigma}$-invariant by Lemma \ref{minimal vectors}(2) so, by Lemma \ref{minimal vectors2}, their nonnegative linear combinations  are also  $G^0_{\sigma}$-invariant.

(2) The canonical barycenter is the sum of the minimal internal vectors of $\sing(\delta)$ and the vertices in $\Omega_\delta:=\Omega\cap \delta$.

(3) The minimal vectors of the pair $(\delta,\omega)$ are the minimal vectors of $\delta$. Hence they are $G^0_{\sigma}$-invariant by Lemma \ref{minimal vectors}(3).

\end{proof}

\begin{remark}
Both corollaries state that the centers used in the desingularization algorithm of the cones (or relative cones) are $G^0_{\sigma}$-invariant. 
%In fact the "minimal" centers (valuations) used in  the combinatorial desingularization have some intrinsic meaning independent of the toroidal structure which makes them to some extent canonical. 
%For example minimal generators can be characterized as the vectors corresponding to the  valuations of the exceptional divisors occuring in  any toric resolution of the corresponding toric variety. 
\end{remark}

\subsection{Locally toric valuations}
Let $X$ be an
algebraic variety and 
 $\nu$ be a valuation of the field $K(X)$ of rational functions. By the valuative criterion of
separatedness and properness the valuation ring of $\nu$
dominates the local ring of a uniquely determined  point (in general nonclosed)
$c_{\nu}$ on a complete variety $X$. (If $X$ is not
complete such a point may not exist). We call the
closure of $c_\nu$ the {\it center of the valuation } $\nu$ and
denote it by $\ce(\nu)$ or $\ce(\nu,X)$. For any $x\in \ce(\nu)$
and $a\in {\bf Z}_{\geq 0}$ let
$$I_{\nu,a,x}:=\{f\in {\cO}_{X,x}\mid 
\nu(f)\geq a\}$$ \noindent be an ideal in ${\cO}_{X,x}$. 
For  a fixed $a$  these ideals define a coherent
sheaf of ideals ${\cI}_{\nu,a}$ supported at $\ce(\nu)$.  

 The following is a well-known fact from the theory of toric varieties. 
 \begin{lemma} \label{le: center} Let $\Sigma$ be a fan, and $X$ be the associated toric variety.
Let $v$ be an integral vector in the support of the fan
$\Sigma$. Then the toric valuation $\val(v)$ on $X$ is centered on
$\overline{O}_\sigma$, where $\sigma$ is the cone whose
relative interior contains $v$.
\end{lemma}
\begin{proof} If $v\in \inte(\sigma)$ then $\val(v)(m)=(v,m) >0$ for any $m\in (\sigma^\vee)^\integ \setminus (\sigma^\tau)^\integ$. But the set $(\sigma^\vee)^\integ \setminus (\sigma^\tau)^\integ$ generates the ideal $I_{O_\sigma}$ of $O_\sigma$.

\end{proof}

\begin{definition} 

Let $X$ be  a locally toric variety. We say that the valuation $\nu$ is locally toric if for any $x\in X$ there exists  a neighborhood $U$ of $x$, and an \'etale morphism $\phi: U\to X_\sigma$ and a vector $v\in \sigma\cap N_\sigma$ such that ${\cI}_{\nu,a}=\phi^{-1}(I_{\val(v),a})$ for any $a\in \NN$.
%If $\mu$ is a locally toric valuation on $X$ then by 
%$\mu_{|\widehat{X}_x}$ we denote the induced valuation on 
%$\widehat{X}_x\simeq\widehat{X}_\sigma$.
 \end{definition}

Smooth morphisms  to toric varieties (charts) allow to define  valuations locally:

\begin{lemma} \cite{Wlodarczyk-toroidal} \label{le: ind2} Let $X_\Sigma$ be a toric
variety (associated with a  fan $\Sigma$  and  $f: U\to X_\Sigma$ be a
smooth  morphism. Let $v\in \inte(\sigma)$, where $\sigma\in \Sigma$, be an integral vector. Assume that the inverse
image of $\overline{O}_\sigma$ is irreducible. Then there
exists a unique  valuation $\mu$ on $U$ such that $
{\cI}_{\mu,a}=f^{-1}({\cI}_{\val(v),a})\cdot{\cO}_{U}$.
\end{lemma}
\begin{proof} We consider the completion $\widehat{\cO}_{\overline{x},X^{\overline{K}}}$ of the local ring at a geometric $\overline{K}$-point $\overline{x}$ over a closed point $x\in f^{-1}({O}_\sigma)$. The smooth morphism $f$ defines  \'etale morphism $\overline{f}: U\to X^{\overline{K}}_\sigma\times \AA_{\overline{K}}^n$ in a sufficiently small neighborhood $U$. 

We have the induced isomorphism $\widehat{f}^*_x: \widehat{\cO}_{f(\overline{x}),X^{\overline{K}}_\sigma\times \AA_{\overline{K}}^n}\to \widehat{\cO}_{\overline{x},X^{\overline{K}}}$.
 The vector $v$ defines a valuation on $\widehat{\cO}_{f(\overline{x}),X^{\overline{K}}_\sigma\times \AA_{\overline{K}}^n}$ and on 
 $\widehat{\cO}_{\overline{x},X^{\overline{K}}}$. Its restriction to  ${\cO_{x,X}}$ defines a valuation on $U$. The verification of the condition $
{\cI}_{\mu,a}=f^{-1}({\cI}_{\val(v),a})\cdot{\cO}_{U}$ and  independence of $\overline{f}$ and $x$ is straightforward.

\end{proof}

\subsubsection{ Filtered centers}

As before one can represent the center of the blow-ups of a  locally toric valuation $\nu$  by the filtered sequence  of ideals  $\cI_{\nu,a}$. 

\begin{definition} \cite[Definition 5.2.6]{Wlodarczyk-toroidal}\label{de: blow2} ( see also Definition \ref{filtered})
By the
{\it blow-up} $\bl_{\nu}(X)$ of $X$ at a locally toric
valuation $\nu$ we mean the normalization of the blow-up at  the filtered center $\{{\cI}_{\nu,k}\}$:
$$\Proj({\cO\oplus \cI}_{\nu,1}\oplus {\cI}_{\nu,2}\oplus\ldots).$$
\end{definition}

\begin{proposition} \cite[Proposition 5.2.9]{Wlodarczyk-toroidal}\label{pr: blow} For any locally toric
valuation $\nu$ on $X$ there exists an
integer $d$ such that 

$\bullet$ $\bl_\nu(X)=\bl_{{\cI}_{\nu, d}}(X)$. 

$\bullet$ If $Y:=\bl_\nu(X)\to X$ is the blow-up of $\nu$,
then the exceptional divisor $D$ is irreducible and ${\bf Q}$-Cartier on $X$. Moreover,  $\nu=\nu_D$  on $Y$.  
\end{proposition}
\begin{proof} The proof is identical as the proof of Lemma \ref{le: blow-up valuation}. Locally the blow-ups are induced by blow-ups of toric valuations of the form $\bl_{\val(v)}(X)$. By quasi-compactness of $X$ and using Lemma \ref{primo}, we can find the same sufficiently divisible $d$ for the open cover with toric charts.

\end{proof}

Also we have

\begin{corollary} \label{divisors2} Let $\nu$ be  a locally toric
valuation  on a locally toric variety $X$, and    $\pi: Y:=\bl_\nu(X)=\bl_{{\cI}_{\nu, a}}(X)\to X$ be the associated normalized blow-up with the exceptional Weil, $\QQ$-Cartier divisor $D$.

Then for the ideals  $\cI_{aD}:=\cI_{\val_{D},a}=\{f\in {\cO}_{Y}\mid 
\nu_D(f)\geq a\}$ we have 
$$\pi_*(\cI_{aD})={\cI}_{\nu, a}.$$

Thus the   valuation $\nu$ is induced by an irreducible  exceptional Weil (${\bf Q}$-Cartier)
divisor on the variety $Y=\bl_{\val(v)}(X)$.

\end{corollary}
\begin{proof} The problem reduces locally to the toric situation, where we use Corollary \ref{divisors}.

\end{proof}

%\begin{definition} \label{filtered2} By the {\it filtered center} of the blow-up we mean the set of ideals $\{\cI_n\}_{n\in \NN}$ on a variety $X$ defining graded algebra ${\cO\oplus \cI}_{1}\oplus {\cI}_{2}\oplus\ldots$,
%such that the equality of  the blow-ups:
%$$\Proj({\cO\oplus \cI}_{1}\oplus {\cI}_{2}\oplus\ldots)=\bl_{{\cI}_{d}}(X).$$
%for sufficiently divisible $d$.
%\end{definition}

\subsection{Desingularization theorems}

\begin{theorem} \label{th: resolution2}
 For any \'etale locally binomial (or locally toric) variety $X$ \footnote{ Definition \ref{locally}} over a  field $K$ of any characteristic  there exists a canonical resolution of singularities i.e. a birational projective $f: Y\to X$ such that
 \begin{enumerate}
\item $Y$ is smoth over $K$.
 \item $f$ is an isomorphism over the open set of the  nonsingular points.
 \item The inverse image $f^{-1}(\Sing(X)$ of the singular locus $\Sing(X)$ is a  simple normal crossing  divisor on $Y$.
 %\item If $X$ is a  (Zariski) locally binomial variety then $f^{-1}(\Sing(X)$ is an SNC divisor.

 \item $f$ is a composition of the normalization and the normalized blow-ups of the locally monomial filtered centers $\{\cJ_{in}\}_{n\in \NN}$ \footnote{Definition \ref{filtered}} 
 defined locally by  valuations.

 \item
 $f$ commutes with smooth morphisms and field extensions, in the sense that the centers are transformed functorially, and the trivial blow-ups are omitted.

\item Moreover, if $D$  is an \'etale  locally toric divisor $D$ on a \'etale locally toric $X$\footnote{Definition \ref{divisor}} then there is functorial desingularization of $(X, D)$ as above such that the strict transform of $D$ has SNC with $f^{-1}(\Sing(X)$.
 \end{enumerate}

\end{theorem}

\noindent{\bf Proof.} First, we prove the theorem in the case of Zariski locally binomial variety $X$. By normalizing $X$ we can further assume that it is locally toric.

%Denote by $\overline{K}$  the algebraic closure of the base field $K$. 
%Let $\overline{X}:=X\times_{\Spec{K}}\Spec{\overline{K}}$ the induced variety over $\overline{K}$, which is toroidal.
Consider  the canonical stratification ${\rm
Sing}(X)$ on $X$ as in Theorem \ref{can0}. (In the case of a locally toric divisor $D$ on $X$, as in (5), we consider the stratification ${\rm
Sing}_D(X)$  as in Theorem \ref{can}, instead of ${\rm
Sing}(X)$.)
Then the pair  $(X,{\rm
Sing}(X))$  is a stratified toroidal variety. %In case $D$ is a locally toric divisor we consider the stratification ${\rm
%Sing}(X)$, with the additional property that the components of $D$ are the closed strata of  ${\rm
%Sing}(\overline{X})$ as in Theorem \ref{can}.  The rest of the proof is the same for both cases.

Let  $\Sigma$ be the associated conical semicomplex.
Consider the complex $\overline{\Sigma}$ consisiting of the disjoint union of the fans  $\overline{\sigma}$ associated with the semicones $\sigma\in \Sigma$.

 Let $$\overline{\Delta}:=V_{k}\cdot\ldots\cdot  V_1\cdot \overline{\Sigma}$$
 be the canonical desingularization of $\overline{\Sigma}$ as in Lemma \ref{can des}. It is obtained by a sequence of the star subdivisions at sets $V_i$ of the points which are 
either minimal vectors  or  the barycenters in singular irreducible faces. Denote by $$\overline{\Delta}_i=V_{i}\cdot\ldots\cdot  V_1\cdot \overline{\Sigma}$$
the intermediate subdivisions.
 
 The subdivision $\overline{\Delta}$  defines for any face $\sigma\in \Sigma$ the canonical   desingularization $$\Delta^\sigma=V^{\sigma}_{k}\cdot\ldots\cdot  V^{\sigma}_1\cdot \sigma,$$
where $V^{\sigma}_i=V_i\cap |\sigma|$, and the intermediate subdivisions:
 $$\Delta_i^\sigma=V^{\sigma}_{i}\cdot\ldots\cdot  V^{\sigma}_1\cdot \sigma$$

By Lemmas \ref{minimal vectors}(2), \ref{minimal vectors1}(1), \ref{minimal vectors2}(2), all the points in the sets $V^{\sigma}_i$ define  the $G_\sigma^0$- invariant valuations on $\widetilde{X_\sigma}$.
  The action of $G_\sigma^0$ on $\widetilde{X_\sigma}$ lifts to $\widetilde{X}_{\Delta_i^\sigma}$, and 
$\widetilde{X}_{\Delta_i^\sigma}\to \widetilde{X}_\sigma$ is $G_\sigma^0$-equivariant. Also,   by the canonicity of the algorithm, the action of $\Aut(\sigma)$ on $\sigma$ lifts to $\Delta_i^\sigma$, so that the subdivisions $\Delta_i^\sigma\to \sigma$ are 
$\Aut(\sigma)$ invariant. By Lemma \ref{le: surjection}, $G_\sigma$ is generated by $G_\sigma^0$ and $\Aut(\sigma)$. This implies  that the  action of $G_\sigma$ on $\widetilde{X_\sigma}$ lifts to each scheme $\widetilde{X}_{\Delta^\sigma_i}$. %where $\Delta_i^{\sigma}:=V_i\cdot\ldots V_1\cdot\sigma\to \sigma$. 
On other hand,  the functoriality of the algorithm implies that  $(\Delta_i^\sigma)_{|\tau}=\Delta_i^{\tau}$ for $\tau\leq \sigma$. 
 These  data define canonical subdivisions $\Delta_i=\{\Delta_i^\sigma\mid \sigma\in\Sigma\}$ of $\Sigma$ (see Definition \ref{de: subdivision of
an oriented semicomplex}). By Theorem \ref{de: toroidal modification}, there exists a unique toroidal modification $X_i$ of $X$ locally defined by the relevant  diagram.

The canonical subdivisions $\{\Delta_i^\sigma\}$ define the intermediate varieties $X_i$. The morphisms $\pi_i:X_i\to X_{i-1}$ are  locally described by the star subdivisions $\Delta^i_\sigma=V_i^\sigma\cdot \Delta^{i-1}_\sigma$ of fans  of $\Delta^{i-1}_\sigma$. Thus, by Corollary \ref{divisors2}, the exceptional divisors $D_{ij}$ of each $\pi_i$
define  locally toric valuations, coresponding  to vectors $v_{ij}\in V^\sigma_i=\{v_{i1},\ldots,v_{ik_\sigma}\} $.
 This implies that   the sets of the vectors $V^\sigma_i$ 
 correspond to the sets of locally toric valuations, and each
 morphism  $ X_i\to X_{i-1}$ is  a composition of the blow-ups at valuations  with disjoint centers.
 % in each centers define $G_\sigma$-invariant sets of valuations on each
%$\widetilde{X}_{\Delta_i^\sigma}\to \widetilde{X}_\sigma$ is $G_\sigma$-equivariant.

 In particular, the corresponding filtered ideals  $$ \cI_{i,n}:=\prod_{v\in V_i}\pi_*(\cI_{nD_i}) ,$$  on $X_{i-1}$, defining the blow-ups are locally described, by Lemma \ref{divisors}, as the pull-backs of the monomial ideals associated with $ V^\sigma_i$ on $X_{\Delta_i^\sigma}$: $$\cI_{V_i,n}:=\prod_{v_\in V^\sigma_i}\cI_{\val(v_{ij},n)}$$
Then, by Corollary \ref{divisors2}, the morphism $X_i\to X_{i-1}$ is the blow -up of the filtered ideal center $\cI_{i,n}$.

The resulting variety $Y=X_k$ defined locally by $\{\Delta^\sigma \mid \sigma\in \Sigma \}$ is  regular. Since  the desingularizations $\Delta^\sigma$   do not affect  regular cones of $\overline{\sigma}$ the nonsingular points remain unaffected. The morphism $Y\to X$ is the composition of the blow-ups of the functorial filtered centers $\cI_{i,n}$. The inverse image  $f^{-1}(\Sing(X))$ of the singular locus $ \Sing(X)$ is defined locally by a toric divisor on a nonsigular toric variety $X_{\Delta^\sigma}$ so it is SNC.

%If a variety $X$ is locally toric over a field $K$ then all the charts can be chosen over $K$ ,and all the centers (defined locally via charts) are defined over $K$. So the desingularization algorithm runs over $K$.  By faithful flatness of $K\subset \overline{K}$ this descent is unique.

The algorithm commutes with smooth  morphisms and field extensions.  The pull-backs of the charts can be used to describe the algorithm locally.

%The desingularization morphism $Y\to X$ of a locally toric variety $(X, D)$ except for $V$ is thus the composition of the blow-ups of the functorial filtered centers $\cI_{i,n}$.
 For  the condition  (5), the strict transforms of locally toric divisors, and the exceptional components  are both local pull-backs of  toric divisors on  smooth toric varieties and thus they have SNC.
 
Now assume that $X$ is \'etale locally binomial  variety. 
Consider its normalization. Since the normalization 
commutes with \'etale maps we obtain an \'etale locally toric  variety. By the previous case there exist compatible desingularizations on \'etale locally toric cover. These compatible desingularizations descend to  the desingualarization $Y$ of $X$. Moreover, they define  an SNC exceptional divisor $E$ on a strict toroidal cover which descends to an NC divisor on $Y$. We  apply Proposition \ref{normal} to further transform it to an SNC divisor.
We use 
 identical  arguments as in the proof of Theorem \ref{des toroidal}.

\subsection{Desingularization of \'etale locally toric varieties except for a toroidal subset}

\begin{theorem} \label{th: resolution4} Let $X$
 be an  \'etale locally toric variety over a %perfect 
 field $K$ with a locally toric Weil divisor $D$. Assume that $D$ has  locally ordered components.\footnote{Definition \ref{order}}\footnote{Definitions \ref{locally}, \ref{divisor}}. Let  $(V, D_V)\subseteq (X, D)$ be an open  saturated  toroidal subset  \footnote{Definition \ref{saturation}}, where  $D_V:=D\cap V$ in $(X, D)$.  
 
   There exists a canonical  resolution of singularities of $(X, D)$ except for $V$ i.e. a birational projective toroidal map $f: Y\to X$ such that
 \begin{enumerate}
 \item $f$ is an isomorphism over the open set $ V$.
 %\item The variety $(Y,V)$ is a toroidal embedding.
 
 %\item If $(V, D_V)$ is saturated in $(X, D)$ then $V$ is the set  where $f$ is an isomorphism.
 
 \item The variety $(Y, D_Y)$ is a  strict toroidal embedding, where $D_Y:=\overline{D_V}$ is the closure of the divisor $D_V$ in $Y$.
 \item The variety $(Y, D_Y)$ is the  saturation
 of $(V, D_V)$ in $(Y, D_Y)$. \footnote{Definition \ref{saturation}}
 \item The 
 complement  $E_{V,Y}:=Y\setminus V$ is  a divisor which has  simple  normal crossings  on  $(Y, D_Y)$, and so does the exceptional divisor   $E_{\exc}\subset E_{V,Y}$. 
 %$E=f^{-1}(X\setminus V)=Y\setminus V$ 
 % has  simple  normal crossings  with $D_Y$ \footnote{Definition \ref{nc}}.  So  $(Y, D_Y\cup E)$ is a strict toroidal embedding.
 \item If $V=(X, D)^{tor}$ is the  toroidal  locus of $(X, D)$ then $E_{\exc}=E_{V,Y}=Y\setminus V$.
 %, and $D_Y\cup E$ is the inverse image of $D$.
 %\item The set of points where $f$ is an isomorphism is the saturation  of $(V, D_V)$ in $(X, D_X)$, where $D_X:=\overline{D_V}$ is the closure of the divisor $D_V$ in $X$.
 % defined by the sets of valuations.
%\item in particular, if $(V, D_V)$ is a smooth toroidal subset of $(X, D)$ and $D_V$ is an SNC divisor on  $V$ then 
%$Y$ is smooth then $E$ is and SNC divisor and $D_Y\cup E$ is an  SNC divisor.
 
 %\item  %If $(V, D_V)$ is a strict toroidal embedding then $(Y, D_Y)$ is also such. If $(X, D_X)$  is a (Zariski) locally binomial variety, and $D_X$ is (Zariski) locally toric divisor then $E$ has simple normal crossings with $D_Y$.
 %\item $f$ is a composition of the normalization and the normalized blow-ups of the locally monomial filtered centers $\{\cJ_{in}\}_{n\in \NN}$.
 \item $f$ is obtained by a sequence of blow-ups at the canonical filtered ideals centers\footnote{Definition \ref{filtered}}.
 \item
 $f$ commutes with smooth morphisms and field extensions respecting the subset $V$, and the order of the components $D$. 
 %in the sense, that the centers are transformed functorially, and the trivial blow-ups are omitted.

 \end{enumerate}
\end{theorem}

\begin{proof}  

{\bf Case 1}. Assume that $X$ is a locally toric variety with a locally toric divisor $D$ and  $V=(X, D)^{tor}$ is the toroidal locus of $(X, D)$. 

By Theorem  \ref{can2}, the complement $X\setminus V$ is of codimension $2$. 
Also, the locally toric divisor $D$ on the variety $X$ 
 defines the natural stratification $S=\Sing_D(X)$, and the associated semicomplex $\Sigma$. 
 Moreover, the divisor $D$ defines a saturated subcomplex $\Omega\subset\Sigma$.   %The  saturated open toroidal subset $(V, D_V)$ is  a strict toroidal embedding corresponding to a 
 
 Consider  a chart $\phi: U\to X_\sigma$, where $\sigma\in \Sigma$ is a semicone. The intersection of $\sigma$ with $\Omega$ defines a unique saturated subcomplex $\Omega_\sigma$ of $\sigma$. It corresponds
 to the divisor $\overline{D}_{\Omega_\sigma}$ on $X_\sigma$, such that
 $D_{|U}=\phi^{-1}(\overline{D}_{\Omega_\sigma})$.
  Here $\overline{D}_{\Omega_\sigma}$
    is the closure of  $D_{\Omega_\sigma}\subset X_{\Omega_\sigma}$  on $X_\sigma$ (see Lemma \ref{cr2}, and Theorem \ref{can2}(7)).  
  
 Denote the embedded fan ${\Reg(\overline{\sigma},\Omega_\sigma})$ by $(\Omega_\sigma^0,\Omega_\sigma)$. Then, by Lemma   \ref{embedded fan}(3), the open subset $X_{{\Omega_\sigma}^0}$ is  the toroidal saturation of $(X_{\Omega_\sigma}, D_{\Omega_\sigma})$ in $X_\sigma$, and  $X_{{\Omega_\sigma}^0}=(X_\sigma,\overline{D}_{\Omega_\sigma})^\tor$ is the toroidal locus.
 
 Thus the toroidal locus $(U, D\cap U)^\tor= V\cap U$ is defined locally as $V\cap U=\phi^{-1}(X_{\Reg(\sigma,\Omega_\sigma)})$.

%The automorphisms $\Aut(\sigma)$ of the cone $\sigma$ preserve all faces of the semicone $\sigma$ and, in particular  the vertices of $\omega\in \Omega_\sigma\subset \sigma$. 
%and $G^0_{\sigma}$-invariant. (Since both groups preserve vertices of $\Omega_\sigma$)
 
 As before consider the complex $\overline{\Sigma}$ consising of the disjoint union of the fans  $\overline{\sigma}$,  of all the faces of the cones $|\sigma|$, where $\sigma\in \Sigma$, and its subcomplex $\overline{\Omega}$ defined by $\Omega_\sigma$ 
 %cones $\omega\in \overline{\sigma}$, 
 where $\sigma\in \Sigma$.

 Let $$\overline{\Delta}:=V_{k}\cdot\ldots\cdot  V_1\cdot \overline{\Sigma}$$
 be the canonical desingularization of the relative complex $(\overline{\Sigma},\overline{\Omega})$ from Theorem \ref{can des2}.
 
 In the process we the 
use centers which are either barycenters, so the sums of the minimal internal vectors in faces $\sigma$ and the vertices 
in $\sigma\cap\Omega$ or  the minimal generators in $\sigma$. Such  centers are $G^0_{\sigma}$-invariant by Lemma \ref{minimal vectors2}. Since the algorithm is functorial  they are also $\Aut(\sigma)$- invariant. By the same argument as before  $\widetilde{X}_{\Delta^\sigma_i}\to  \widetilde{X}_{\sigma}$ is $G^0_{\sigma}$-equivariant.

  Moreover, locally for any semicone $\sigma\in \Sigma$  we obtain a  subdivision complex $(\Delta^\sigma,\Omega_\sigma)$ which is relatively regular.
This defines a collection of the subdivisions $\{(\Delta^\sigma,\Omega_\sigma)\}$ which determines the subdivision $\Delta=\{\Delta^\sigma\}$ and, induces, by Theorem \ref{de: toroidal modification},
 a unique transformation $Y$ of $(X,S)$.  
 %As before this and the cooresponding birational modifications 
 %$Y=X_k$ and  $X_i$.

Repeating the reasoning from the previous proof we we see that
%that the variety  $Y=X_k$  is locally defined by a relatively regular  fans  $ (\Delta^\sigma,\Omega_\sigma)$ where $\sigma\in \Sigma$, and 
the modification is given by 
a sequence of the blow-ups $\pi_i:X_i\to X_{i-1}$ at invariant valuations.

% centered in the complement of $V$ which is locally the preimage  of ${\Reg(\sigma,\Omega_\sigma})$, not affected by the resolution algorithm.

 The desingularization $\Delta^\sigma$  does not affect  relatively regular cones with respect to  $\Omega_\sigma$.  Thus, in particular, the cones in $\Omega_\sigma^0$  remain unaffected. Consequently, the points in $V$ are not modified in the process. Since $(\Delta^\sigma,\Omega_\sigma)$ is a regular relative fan (an embedded fan) we get by Proposition \ref{embedded fan}, that:
\begin{enumerate} 
\item $(X_{\Delta^\sigma},\overline{D_{\Omega_\sigma}})$  is a strict toroidal embedding locally corresponding to $(Y, D_Y)$. 
\item The exceptional divisor $E_\sigma\subseteq  X_{\Delta^\sigma}\setminus X_{\Omega_\sigma}$   of $X_{\Delta^\sigma}\to X_\sigma$ has SNC. So $E_{\exc}$ has SNC
on $(Y, D_Y)$.
 \item $(X_{\Delta^\sigma},\overline{D_{\Omega_\sigma}})$ is the saturation of $(X_{\Omega_\sigma}, D_{\Omega_\sigma})$ and   of 
$(X_{\Omega_\sigma^0}, D_{\Omega_\sigma^0})$ so $(Y, D_Y)$ is the saturation of $(V, D_V)$.
\end{enumerate}
 By definition, $Y\setminus V$ coincides with the exceptional divisor $E_{\exc}$. Otherwise $Y$ would contain an unmodified toroidal open subset of $X$ strictly bigger than $V$. This contradicts the assumption that $V$ is the toroidal locus of $X$.

%If the  variety $X$ is locally toric  over a nonclosed field $K$ then we use the same argument with charts as in the proof of Theorem \ref{th: resolution2} to show that the desingularization is defined over $K$, and that it commutes with smooth morphisms.

{\bf Case 2.} Assume that the variety  $(X, D)$ is Zariski locally toric with a locally toric divisor, and  the open subset $V$ is an arbitrary saturated toroidal subset.  

Let $V_0:=(X, D)^{tor} \supset V$ be    the  toroidal locus of $(X, D)$. %$X\setminus V_0$ is of codimension $2$. 
 By definition, $V$ is a toroidal subset which is saturated in $V_0$.
 By Case 1, there is a desingularization except for $V_0$.  Now $(Y, D_Y)$ is a strict toroidal embedding and  $E:=Y\setminus V_0$ is a divisor having SNC on $(Y, D_Y)$.
 So $(Y, D_Y\cup E)$ is a strict toroidal embedding. Since $V$ is saturated in $V_0$ and $V_0$ is saturated in $Y$  it follows that
 $(V, D_V)$ is an open saturated toroidal subset  of $(Y, D_Y\cup E)$. It suffices to use Theorem \ref{des toroidal2} and apply the partial desingularization of toroidal embeddings $(V, D_V)\subset (Y, D_Y\cup E)$.
 
 {\bf Case 3.} Assume that the variety  $(X, D)$ is \'etale locally toric with an \'etale locally toric divisor, and  the open subset $V$ is an arbitrary saturated toroidal subset of $(X, D)$. We repeat the  reasoning  from the proof of Theorem \ref{th: resolution}. 
 
 We consider an \'etale  covering $(X^0, D^0)$ of $(X, D)$ consisting of locally toric varieties with locally toric divisor. Note that, by Lemma \ref{extension2}, the inverse image  $V^0$ of $V$ is a  saturated toroidal subset of $(X^0, D^0)$.
  By Step 2, there is    a desingularization at functorial centers 
on  an  $X^0$. By functoriality, the centers descend 
(by the flat descent) to the centers on $X$.  
We obtain a toroidal embedding $(Y, D_Y\cup E_{V,Y})$ containing a toroidal subset  $(V, D_Y)$. 
Since $(Y, D_Y)$ contains an open toroidal subset $(V, D_V)$ intersecting all strata  then
$(Y,\overline{D_V})$ is a strictly toroidal embedding, by Lemma \ref{extension1}.  Consequently, $E_{V,Y}:=X\setminus V$ has NC on $(Y, D_Y)$. To  obtain the condition that $E_{V,Y}$ has SNC with $\overline{D_V}$ it suffices to apply Proposition \ref{normal}. The  process commutes with
 field extensions  and smooth 
 morphisms 
  preserving the order of the components $D_V$

\end{proof}
\section{Comparison of the desingualarization algorithm for toroidal embeddings and locally toric varieties} \label{compari}
 
 When we forget the toroidal structure, a strict toroidal embedding is a locally toric variety.
The desingularization algorithm  is identical in both cases. 
If $\Sigma$ is the  complex associated with the toroidal structures then $\sing(\Sigma)$ defines  the  semicomplex  associated with the canonical singular stratification $S=\Sing(X)$ on $X$. 

The same algorithm in both situations is induced by the decompositions of  the faces of $\sing(\Sigma)$. In the case of toroidal embedding the decomposition of $\sing(\Sigma)$ simply  extends uniquely to decomposition of $\Sigma$ defining the relevant toroidal modification. In the locally toric situation the decomposition of $\sing(\Sigma)$ induces the modification directly via the charts. 

A point $x$ on a toroidal embedding belongs to an open neighborhood associated with a face $\sigma\in \Sigma$. The modification of $\sigma$ is induced by the subdivision of its singular part $\sing(\sigma)$.
The very same point on the same variety with  a locally toric structure with the singular stratification   will be associated via a chart with a face $\sing(\sigma)$ of the semicomplex $\sing(\Sigma)$, and the transformation will be again defined by the same subdivision of $\sing(\sigma)\in \sing(\Sigma)$. So both algorithms agree locally and globally.

\section{Counterexample to the Hironaka desingularization of locally binomial varieties in positive characteristic}
Let $x\in X$ be a point on a smooth variety over $K$.
Let $\cI$ be an ideal sheaf on $X$ of maximal order $p$. Denote by $\Sing_p(\cI)$ the locus  of the points of  order $p$. 
Recall that a {\it maximal contact} is a local parameter $u$ such that its zero locus $H=V(u)$ contains $\Sing_p(\cI)$, and this property will be preserved after any sequence of blow-ups of smooth centers contained in $\Sing_p(\cI)$.
The following example shows that there is no maximal contact for locally binomial varieties over a field of positive characteristic. Thus the Hironaka characteristic zero approach fails already in this case.
\begin{example} \label{Hir} Let $X\subset \bA^4=\Spec(K[x,y,z,w])$  be the hypersurface over a field of characteristic $p$, defined by a single  equation $x^p-y^pz$.
It admits the action of the group of automorphisms $\bA^1=\Spec(K[t])$:

$$x\mapsto x+ty, \quad y\mapsto y, \quad z\to z+t^p.$$
The locus $\Sing_p(X)$ of the points of maximal order $p$ is given by the  smooth subvariety $Z=V(x,y)$. Suppose that there is a maximal contact $u\in \cO_X$ such that $V(u)\supset \Sing_p(X)=V(x,y)$ and that the property will be preserved  after blow-ups at smooth centers contained in $\Sing_p(X)$.  Thus $u$ is a local parameter which can be written in the  form $$u=a(x,z,w)x+b(x,y,z,w)y.$$   Using automorphisms above we can always assume that $b(0,0,0,0)\neq 0$.
 Applying the blow-up at $x=y=z=w=0$ in the chart of $w$: $$(x,y,z,w)\mapsto (xw,yw,zw,w)$$ we transform
$x^p-y^pz$,    into $$w^p(x^p-y^pzw).$$ Moreover, the equation of the new maximal contact $$u'=1/w(a(xw,yw,zw,w))wx+b(xw,yw,zw,w)wy)=$$ $$a(xw,yw,zw,w))y+b(xw,yw,zw,w)y$$ will maintain its form with the condition $b(0,0,0,0)\neq 0$ in the corresponding chart of $w$.

After $p$ such blow-ups (and factoring the exceptional divisors) we obtain the form $$x^p-y^pzw^p.$$

The  order  $p$ locus $\Sing_p(X)$  has now two components $V(x,y)$ and $V(x,w)$. 
The intersection $V(x,w)\cap V(u)$ is equal to $V(x,y,w)$ in a neighborhood with $b(x,y,z,w)\neq 0$. Thus 
the component $V(x,w)$ is not 
contained in $V(u)$ which contradicts to the assumption on maximal contact.

\end{example}

%\section{Some applications. The Weak factorization theorem}

%\begin{theorem} Let $(X_1, D_1)$ and $(X_2, D_2)$  be two birationally equivalent complete toroidal embeddings which are  isomorphic over an  extendable open toroidal subset $(V, D_V)$, such that the complements $E_1=X_1\setminus V$, 
%and $E_2=X_2\setminus V$  are relatively SNC divisors.
%There exists a sequence of blow-ups and blow-downs at relatively smooth centers which have SNC with the transforms of $E$ which take $(X_1, D_1)$ to $(X_2, D_2)$.

%\end{theorem}

\bibliographystyle{amsalpha}
\bibliography{Binomial}

\end{document}